\documentclass[preprint]{imsart}

\RequirePackage[OT1]{fontenc}
\RequirePackage{amsthm,amsmath}
\RequirePackage[colorlinks,citecolor=blue,urlcolor=blue]{hyperref}
  
\usepackage{amsfonts}
\usepackage{amsgen,amsmath,amstext,amsbsy,amsopn,amssymb,subfigure,amsthm}
\usepackage[dvips]{graphicx}
\usepackage{comment}
\usepackage{enumitem}
\usepackage{natbib}
\usepackage{booktabs}
\usepackage{blkarray}
\usepackage{algorithm}
\usepackage{algpseudocode}
\usepackage{url}
\usepackage{longtable}
\usepackage{mathtools}
\DeclarePairedDelimiter\ceil{\lceil}{\rceil}

\usepackage[usenames]{color}
\newcommand{\red}{}

\definecolor{blue}{RGB}{000,000,200}
\definecolor{green}{RGB}{000,150,100}
\definecolor{purple}{RGB}{220,040,250}

\def\red{}

\newtheorem{theorem}{Theorem}

\newtheorem{lemma}{Lemma}
\newtheorem{remark}{Remark}

\newcommand{\be}{\begin{equation}}
\newcommand{\ee}{\end{equation}}
\newcommand{\bea}{\begin{eqnarray}}
\newcommand{\eea}{\end{eqnarray}}
\newcommand{\beas}{\begin{eqnarray*}}
	\newcommand{\eeas}{\end{eqnarray*}}

\newcommand{\0}{{\mathbf{0}}}

\renewcommand{\v}{{\mathbf{v}}}

\newcommand{\V}{{\rm V}}

\def\eps{\varepsilon}

\newcommand{\bfm}[1]{\ensuremath{\mathbf{#1}}}

\def\be{\bfm e}     \def\EE{\mathbb{E}}

\def\bj{\bfm j}

     \def\PP{\mathbb{P}}
     
     \def\RR{\mathbb{R}}
\def\bs{\bfm s}     
\def\bt{\bfm t}     
     
\def\bv{\bfm v}

\def\calE{{\cal  E}}

\def\calM{{\cal  M}} 
\def\calN{{\cal  N}}

\def\calR{{\cal  R}} 
\def\calS{{\cal  S}}

\allowdisplaybreaks

\newcommand{\rank}{{\rm rank}}

\newcommand{\tr}{{\rm tr}}

\def\tran{\mathsf{T}}

\makeatletter
\newcommand*{\rom}[1]{\expandafter\@slowromancap\romannumeral #1@}
\makeatother

\def\mybox#1{\vskip1mm \begin{center} \bf \red
		\hspace{.0\textwidth}\vbox{\hrule\hbox{\vrule\kern6pt
				\parbox{.95\textwidth}{\kern6pt#1\vskip6pt}\kern6pt\vrule}\hrule}
	\end{center} \vskip-5mm}

\arxiv{1901.00304}

\startlocaldefs
\theoremstyle{plain}

\endlocaldefs
 
 \usepackage{mathtools}

\begin{document}

\begin{frontmatter}
\title{Normal Approximation and Confidence Region of Singular Subspaces}
\runtitle{Normal Approximation of SVD}

\begin{aug}
\author{\fnms{Dong} \snm{Xia}\thanksref{}\ead[label=e2]{madxia@ust.hk}\ead[label=u2,url]{http://www.math.ust.hk/{\raise.17ex\hbox{$\scriptstyle\sim$}}madxia/}}

\runauthor{D. Xia}

\affiliation{Hong Kong University of Science and Technology}

\address{Department of Mathematics\\
Hong Kong University of Science and Technology\\
Clear Water Bay, Kowloon, Hong Kong.\\
\printead{e2}\\
\printead{u2}
}
\end{aug}

\begin{abstract}
This paper is on the normal approximation of singular subspaces when the noise matrix has i.i.d. entries. Our contributions are three-fold. First, we derive an explicit representation formula of the empirical spectral projectors. The formula is neat and holds for deterministic matrix perturbations. Second, we calculate the expected projection distance between the empirical singular subspaces and true singular subspaces. {Our method allows obtaining arbitrary $k$-th order approximation of the expected projection distance.} Third,
we prove the non-asymptotical normal approximation of the projection distance with different levels of bias corrections. {By the $\ceil{\log(d_1+d_2)}$-th order bias corrections, the asymptotical normality holds under optimal signal-to-noise ration (SNR) condition where $d_1$ and $d_2$ denote the matrix sizes.} In addition, it shows that higher order approximations are unnecessary when $|d_1-d_2|=O((d_1+d_2)^{1/2})$. 
Finally, we  provide comprehensive simulation results to merit our theoretic discoveries. 

{Unlike the existing results, our approach is non-asymptotical and the convergence rates are established. Our method allows the rank $r$ to diverge as fast as $o((d_1+d_2)^{1/3})$. Moreover, our method requires no eigen-gap condition (except the SNR) and no constraints between $d_1$ and $d_2$. }
\end{abstract}

  

\end{frontmatter}

\section{Introduction}
Matrix singular value decomposition (SVD) is a powerful tool for various purposes across diverse fields. In numerical linear algebra, SVD has been successfully applied for solving linear inverse problems, low-rank matrix approximation and etc. See, e.g., \citep{golub2012matrix}, for more examples. In many machine learning tasks, SVD is crucial for designing computationally efficient algorithms, such as matrix and tensor completion (\citep{cai2010singular}, \citep{keshavan2010matrix}, \citep{candes2010power}, \citep{xia2017polynomial}, \citep{xia2017statistically}), and phase retrieval (\citep{ma2017implicit}, \citep{candes2015phase}), where SVD is often applied for generating a warm initial point for non-convex optimization algorithms. In statistical data analysis, SVD is superior for denoising and dimension reduction. For instance, SVD, as a dimension reduction tool, is used for text classification in \citep{kim2005dimension}. See also \citep{li2007directional}. In \citep{shabalin2013reconstruction}, SVD shows appealing performances in low rank matrix denoising. More specifically, in \citep{donoho2014minimax}, they proved that statistically minimax optimal matrix denoising can be attained via precise singular value thresholding. Recently, matrix SVD is generalized to tensor SVD for tensor denoising, see \citep{xia2019sup} and \citep{zhang2018tensor}.

The perturbation analysis is critical for advancing the theoretical developments of SVD for low-rank matrix denoising where the observed data matrix often equals a low-rank information matrix plus a noise matrix. 
The deterministic perturbation bounds of matrix SVD have been well established by Davis-Kahan (\citep{davis1970rotation}, \citep{yu2014useful}) and Wedin (\citep{wedin1972perturbation}) many years ago. 
Among those deterministic perturbation bounds, one simple yet useful bound shows that the perturbation of singular vectors is governed by the so-called signal-to-noise ratio (SNR) where "signal" refers to the smallest non-zero singular value of the information matrix and the "noise" refers to the spectral norm of the noise matrix. It is a quite general result since the bound does not rely on the wellness of alignments between the singular subspaces of the information and of the noise matrices. Such a general bound turns out to be somewhat satisfactorily sharp when the noise matrix  contains i.i.d. random entries. However, more refined characterizations of singular vectors are needed on the frontiers of statistical inference for matrix SVD. The Davis-Kahan Theorem and Wedin's perturbation bounds are illustrated by the non-zero smallest singular value of the information matrix, where the effects of those large singular values are usually missing. Moreover, the exact numerical factor is also not well recognized. 

The behavior of singular values and singular vectors of low rank perturbations of large rectangular random matrices is popular in recent years. They play a key role in statistical inference with diverse applications. See \cite{li2018two}, \cite{naumov2017bootstrap}, \cite{tang2018limit} for some examples in network testing. 
The asymptotic limits of singular values and singular vectors were firstly developed by \citep{benaych2012singular}, where the convergence rate of the largest singular value was also established. Recently, by \citep{ding2017high}, more precise non-asymptotic concentration bounds for empirical singular values were obtained. Meanwhile, \citep{ding2017high} also proved non-asymptotic perturbation bounds of empirical singular vector  when the associated singular value has multiplicity $1$. 
In a recent work \citep{bao2018singular}, the authors studied the asymptotic limit distributions of the empirical singular subspaces when (scaled) singular values are bounded. Specifically, they showed that if the noise matrix has Gaussian distribution, then the limit distribution of the projection distance  is also Gaussian. {Unlike these prior arts (\cite{ding2017high}, \cite{bao2018singular}), we focus on the non-asymptotical normal approximations of the joint singular subspaces in a different regime. Our approach allows the rank to diverge, and imposes no constraints between $d_1$ and $d_2$. In addition, we establish the convergence rates and impose no eigen-gap conditions (except SNR).}

 In \citep{xia2018confidence}, the low rank matrix regression model is investigated where the author proposed a de-biased estimator built on nuclear normal penalized least squares estimator. The de-biased estimator ends up with an analogous form of the low rank perturbation of rectangular random matrices. Then, non-asymptotical normal approximation theory of  the projection distance is proved, under near optimal sample size requirement. The paramount observation is that the mean value in the limit normal distribution is significantly larger than its standard deviation. As a result,  a much larger than regular sample size requirement is necessary to tradeoff the estimation error of the expected projection distance. Most recently, \citep{chen2018asymmetry} revealed an interesting phenomenon of the perturbation of eigenvalues and eigenvectors of such non-asymmetric random perturbations, showing that the perturbation of eigen structures is much smaller than the singular structures. In addition, some non-asymptotic perturbation bounds of empirical singular vectors can be found in \citep{koltchinskii2016perturbation},\citep{bloemendal2016principal} and \citep{abbe2017entrywise}. The minimax optimal bounds of singular subspace estimation for low rank perturbations of large rectangular random matrices are established in \citep{cai2018rate}.

Our goal is to investigate the central limit theorems of singular subspaces in the low rank perturbation model of large rectangular random matrices.  As illustrated in \citep{xia2018confidence}, the major difficulty arises from how to precisely determine the expected projection distance. One conclusive contribution of this paper is an explicit representation formula of the empirical spectral projector. This explicit representation formula allows us to obtain  precise characterization of the {(non-asymptotical)} expected projection distance. After those higher order bias corrections, we prove normal approximation of the singular subspaces with optimal {(in the consistency regime)} SNR requirement. For better presenting the results and highlighting the contributions, let's begin with introducing the standard notations.  
We denote $M=U\Lambda V^{\tran}$ the unknown $d_1\times d_2$ matrix where $U\in\RR^{d_1\times r}$ and $V\in\RR^{d_2\times r}$ are its left and right singular vectors. The diagonal matrix $\Lambda={\rm diag}(\lambda_1,\cdots,\lambda_r)$ contains $M$'s non-increasing positive singular values. The observed data matrix $\hat M\in\RR^{d_1\times d_2}$ satisfies the additive model:
\begin{equation}\label{eq:hatM_def}
\hat M = M+Z\quad {\rm where}\quad Z_{j_1j_2}\stackrel{i.i.d.}{\sim} \calN(0,1)\quad {\rm for}\ 1\leq j_1\leq d_1, 1\leq j_2\leq d_2.
\end{equation}
Here, we fix the noise variance to be $1$, just for simplicity. Let $\hat U\in\RR^{d_1\times r}$ and $\hat V\in\RR^{d_2\times r}$ be the top-$r$ left and right singular vectors of $\hat M$. Let $\hat \Lambda={\rm diag}(\hat\lambda_1,\cdots,\hat{\lambda}_r)$ denote the top-$r$ singular values of $\hat M$. 
 We focus on the projection distance between the empirical and true singular subspaces which is defined by
\begin{equation}\label{eq:dist_def}
{\rm dist}^2[(\hat U,\hat V), (U,V)]:=\|\hat U\hat U^{\tran}-UU^{\tran}\|_{\rm F}^2+\|\hat V\hat V^{\tran}-VV^{\tran}\|_{\rm F}^2. 
\end{equation}
By Davis-Kahan Theorem (\citep{davis1970rotation}) or Wedin's $\sin\Theta$ theorem (\citep{wedin1972perturbation}), ${\rm dist}^2[(\hat U,\hat V), (U,V)]$ is non-trivial on the event $\{\lambda_r>2\|Z\|\}$. It is well-known that $\|Z\|=O_P(\sqrt{d_{\max}})$ where $\|\cdot\|$ denotes the spectral norm and $d_{\max}=\max\{d_1,d_2\}$. Therefore, it is convenient to consider $\lambda_r\gtrsim \sqrt{d_{\max}}$. {In this paper, we focus on the consistency regime\footnote{{We note that, in RMT literature (see, e.g., \citep{bao2018singular},\citep{ding2017high}), many works studied the problem when $\lambda_r=O(\sqrt{d_{\max}})$ and $\lambda_r\gtrsim (d_1d_2)^{1/4}$. In this paper, we focus on the regime when empirical singular subspaces are consistent, i.e., $\EE{\rm dist}^2[(\hat U,\hat V), (U,V)]\to 0$ when $d_{\max}\to \infty$. As shown in \citep{cai2018rate}, such consistency requires $\sqrt{rd_{\max}}/\lambda_r\to 0$.}} so that the empirical singular subspaces are consistent which requires $\lambda_r\gg \sqrt{rd_{\max}}$.}  See, e.g., \citep{tao2012topics}, \citep{koltchinskii2016perturbation},  \citep{cai2018rate} and \citep{vershynin2010introduction}. 

Our contributions are summarized as follows. 
\begin{enumerate}
\item An explicit representation formula of $\hat U\hat U^{\tran}$ and $\hat V\hat V^{\tran}$ is derived. In particular, $\hat U\hat U^{\tran}$ and $\hat V\hat V^{\tran}$ can be completely determined by a sum of a series of matrix product involving only $\Lambda, UU^{\tran}, U_{\perp}U_{\perp}^{\tran}, VV^{\tran}, V_{\perp}V_{\perp}^{\tran}$ and $Z$, where $U_{\perp}\in\RR^{d_1\times (d_1-r)}$ and $V\in\RR^{d_2\times (d_2-r)}$ are chosen so that $(U,U_{\perp})$ and $(V,V_{\perp})$ are orthonormal matrices. To derive such a useful representation formula, we apply the Reisz formula, combinatoric formulas, contour integrals, residue theorem and generalized Leibniz rule.  It worths to point out that the representation formula is deterministic  as long as $\|Z\|<\lambda_r/2$. We believe that this representation formula of spectral projectors should be of independent interest for various purposes. 

\item By the representation formula, we prove the normal approximation of $\hat\eps_1:=\big({\rm dist}^2[(\hat U,\hat V), (U,V)]-\EE{\rm dist}^2[(\hat U,\hat V), (U,V)]\big)/\big(\sqrt{8d_{\star}}\|\Lambda^{-2}\|_{\rm F}\big)$ where $d_{\star}=d_1+d_2-2r$. In particular, we show that $\hat\eps_1$ converges to a standard normal distribution as long as $\sqrt{rd_{\max}}/\lambda_r\to 0$ and $r^3/d_{\max}\to 0$ as $d_1, d_2\to\infty$. The required SNR is optimal in the consistency regime. {Note that our result allows $r$ to diverge as fast as $o((d_1+d_2)^{1/3})$. In addition, no conditions on the eigen-gaps (except $\lambda_r$) are required. The convergence rate is also established. }
The proof strategy is based on the Gaussian isoperimetric inequality and Berry-Esseen theorem. 

\item The unknown $\EE{\rm dist}^2[(\hat U,\hat V), (U,V)]$ plays the role of centering in $\hat\eps_1$. To derive user-friendly normal approximations of ${\rm dist}^2[(\hat U,\hat V), (U,V)]$, it suffices to explicitly calculate its expectation (non-asymptotically). By the representation formula of $\hat U\hat U^{\tran}$ and $\hat V \hat V^{\tran}$, we obtain approximations of $\EE{\rm dist}^2[(\hat U,\hat V), (U,V)]$. Different levels of approximating $\EE{\rm dist}^2[(\hat U,\hat V), (U,V)]$ ends up with different levels of bias corrections.  These levels of approximations are
\begin{enumerate}
\item Level-$1$ approximation: $B_1=2d_{\star}\|\Lambda^{-1}\|_{\rm F}^2$. The approximation error is
$$
\Big|\EE{\rm dist}^2[(\hat U,\hat V), (U,V)]-B_1 \Big|=O\Big(\frac{rd_{\max}^2}{\lambda_r^4}\Big).
$$
\item Level-$2$ approximation: $B_2=2(d_{\star}\|\Lambda^{-1}\|_{\rm F}^2-\Delta_d^2\|\Lambda^{-2}\|_{\rm F}^2)$ where $\Delta_d=d_1-d_2$. Then, 
$$
\Big|\EE{\rm dist}^2[(\hat U,\hat V), (U,V)]-B_2 \Big|=O\Big(\frac{rd_{\max}^3}{\lambda_r^6}\Big).
$$
\item Level-$k$ approximation: $B_k=2d_{\star}\|\Lambda^{-1}\|_{\rm F}^2-2\sum_{k_0=2}^{k}(-1)^{k_0}\Delta_d(d_{1-}^{k_0-1}-d_{2-}^{k_0-1})\|\Lambda^{-k_0}\|_{\rm F}^2$ where $d_{1-}=d_1-r$ and $d_{2-}=d_2-r$. Then, for all $k\geq 2$, 
\begin{align*}
\Big|\EE{\rm dist}^2[(&\hat U,\hat V), (U,V)]-B_k \Big|\\
=&O\Big(\frac{r^2d_{\max}}{\lambda_r^4}+\frac{r^2}{\sqrt{d_{\max}}}\cdot \Big(\frac{d_{\max}}{\lambda_r^2}\Big)^3+r\Big(\frac{C_2d_{\max}}{\lambda_r^2}\Big)^{k+1}\Big)
\end{align*}
where $C_2>0$ is some absolute constant.
\end{enumerate}
The aforementioned approximation errors hold whenever $C_2d_{\max}/\lambda_r^2<1$. 
{ Explicit formula for $B_\infty$ is also derived. }
An intriguing fact is that if $|d_1-d_2|=O(\sqrt{d_{\max}})$, i.e., the two dimensions of $M$ are comparable, then higher level approximations have similar effects as the Level-$1$ approximation. Simulation results show that Level-$1$ approximation by $B_1$ is indeed satisfactorily accurate when $d_1=d_2$. 

\item By replacing $\EE{\rm dist}^2[(\hat U,\hat V), (U,V)]$ with $B_k$, we prove the normal approximation of ${\rm dist}^2[(\hat U,\hat V), (U,V)]$. Different levels of bias corrections require different levels of SNR conditions for the asymptotical normality. For instance, we prove the normal approximation of $\hat\eps_2:=\big({\rm dist}^2[(\hat U,\hat V), (U,V)]-B_{\ceil{\log d_{\max}}}\big)/\big(\sqrt{8d_{\star}}\|\Lambda^{-2}\|_{\rm F}\big)$ with the $\ceil{\log d_{\max}}$-th order bias correction. More exactly, we show the asymptotical normality of $\hat\eps_2$ when $\sqrt{rd_{\max}}/\lambda_r\to 0$  and $r^3/d_{\max}\to 0$ as $d_1, d_2\to\infty$. {As far as we know, this is the first result about the limiting distribution of singular subspaces which allows the rank $r$ to diverge. Meanwhile, no eigen-gap conditions (except SNR) are needed. 
Since our normal approximation is non-asymptotical, we impose no constraints on the relation between $d_1$ and $d_2$.}
\end{enumerate}
The rest of the paper is organized as follows. In Section~\ref{sec:hat_Theta}, we derive the explicit representation formula of empirical spectral projector. The representation formula is established under deterministic perturbation. We prove normal approximation of  ${\rm dist}^2[(\hat U,\hat V), (U,V)] $ in Section~\ref{sec:normal}. Especially, we show that ${\rm dist}^2[(\hat U,\hat V), (U,V)] $ is asymptotically normal under optimal SNR conditions. In Section~\ref{sec:approx_Edist} and ~\ref{sec:explicit_CLT}, we develop the arbitrarily $k$-th level approximations of $\EE{\rm dist}^2[(\hat U,\hat V), (U,V)] $ and its corresponding normal approximation, where requirements for SNR are specifically developed. In Section~\ref{sec:CR}, we propose confidence regions and discuss about data-adaptive shrinkage estimator of singular values.
We then display  comprehensive simulation results in Section~\ref{sec:sim}, where, for instance, we show the importance of higher order approximations of $\EE{\rm dist}^2[(\hat U,\hat V), (U,V)] $ when the matrix has unbalanced sizes and the effectiveness of shrinkage estimation of singular values.  The proofs are collected in Section~\ref{sec:proof} and Appendix~\ref{sec:proof}.

\section{Representation formula of spectral projectors}\label{sec:hat_Theta}
Let $A$ and $X$ be $d\times d$ symmetric matrices. The matrix $A$ has rank $r=\rank(A)\leq d$. Denote the eigen-decomposition of $A$,
$$
A=\Theta\Lambda \Theta^{\tran}=\sum_{j=1}^r \lambda_j \theta_j\theta_j^{\tran}
$$
where $\Lambda={\rm diag}(\lambda_1,\cdots,\lambda_r)$ contains the non-zero non-increasing eigenvalues of $A$. The $d\times r$ matrix $\Theta=(\theta_1,\cdots,\theta_r)$ consists of $A$'s eigenvectors. The noise matrix $X$ satisfies $\|X\|< \min_{1\leq i\leq r}\frac{|\lambda_i|}{2}$ where $\|\cdot\|$ denotes the matrix operator norm. Given  $\hat A=A+X$ where $A$ and $X$ are unknown, our goal is to estimate $\Theta$. 
We denote $\hat\Theta=(\hat\theta_1,\cdots,\hat\theta_r)$ the $d\times r$ matrix containing the eigenvectors of $\hat A$ with largest $r$ eigenvalues in absolute values. Therefore, $\hat\Theta$ represents the empirical version of $\Theta$. We derive the representation formula of $\hat \Theta \hat \Theta^{\tran}$ for deterministic $X$. The formula is useful for various of purposes. 

To this end, define $\Theta_{\perp}=(\theta_{r+1},\cdots,\theta_{d})$ the $d\times (d-r)$ matrix such that $(\Theta,\Theta_{\perp})$ is orthonormal. Define the spectral projector,
$$
\mathfrak{P}^{\perp}=\sum_{j=r+1}^{d}\theta_j\theta_j^{\tran}=\Theta_{\perp}\Theta_{\perp}^{\tran}.
$$
Also, define
$$
\mathfrak{P}^{-1}:=\sum_{j=1}^r\lambda_j^{-1}\theta_j\theta_j^{\tran}=\Theta\Lambda^{-1}\Theta^{\tran}.
$$
Meanwhile, we write $\mathfrak{P}^{-k}=\Theta\Lambda^{-k}\Theta^{\tran}$ for all $k\geq 1$.  For notational simplicity, we denote $\mathfrak{P}^0=\mathfrak{P}^{\perp}$ and  denote the $k$-th order perturbation term
\begin{equation}\label{eq:calS_def}
\calS_{A,k}(X)=\sum_{\bs: s_1+\cdots+s_{k+1}=k}(-1)^{1+\tau(\bs)}\cdot\mathfrak{P}^{-s_1}X\mathfrak{P}^{-s_2}X\cdots X\mathfrak{P}^{-s_{k+1}}
\end{equation}
where $\bs=(s_1,\cdots,s_{k+1})$ contains non-negative integer indices and 
$$
\tau(\bs)=\sum_{j=1}^{k+1}\mathbb{I}(s_{j}>0)
$$
denotes the number of positive indices in $\bs$. For instance, if $k=1$, we have 
$$
\calS_{A,1}(X)=\mathfrak{P}^{-1}X\mathfrak{P}^{\perp}+\mathfrak{P}^{\perp}X\mathfrak{P}^{-1}.
$$
If $k=2$, by considering $s_1+s_2+s_3=2$ for $s_1,s_2,s_3\geq 0$ in (\ref{eq:calS_def}), we  have  
\begin{align*}
\calS_{A,2}(X)=&\big(\mathfrak{P}^{-2}X\mathfrak{P}^{\perp}X\mathfrak{P}^{\perp}+\mathfrak{P}^{\perp}X\mathfrak{P}^{-2}X\mathfrak{P}^{\perp}+\mathfrak{P}^{\perp}X\mathfrak{P}^{\perp}X\mathfrak{P}^{-2}\big)\\
&-\big(\mathfrak{P}^{\perp}X\mathfrak{P}^{-1}X\mathfrak{P}^{-1}+\mathfrak{P}^{-1}X\mathfrak{P}^{\perp}X\mathfrak{P}^{-1}+\mathfrak{P}^{-1}X\mathfrak{P}^{-1}X\mathfrak{P}^{\perp}\big).
\end{align*}

\begin{theorem}\label{thm:hat_Theta}
If $\|X\|<\min_{1\leq i\leq r}\frac{|\lambda_i|}{2}$, then 
$$
\hat\Theta\hat\Theta^{\tran}-\Theta\Theta^{\tran}=\sum_{k\geq 1}\calS_{A,k}(X)
$$
where $\calS_{A,k}(X)$ is defined in (\ref{eq:calS_def}) and  we set $\mathfrak{P}^{0}=\mathfrak{P}^{\perp}=\Theta_{\perp}\Theta_{\perp}^{\tran}$ for notational simplicity. 
\end{theorem}

Apparently,  by eq. (\ref{eq:calS_def}), a simple fact is
$$
\big\|\calS_{A,k}(X) \big\|\leq {2k \choose k}\cdot \frac{\|X\|^k}{\lambda_r^k}\leq \Big(\frac{4\|X\|}{\lambda_r}\Big)^k,\quad \forall \ k\geq 1.
$$

\section{Normal approximation of spectral projectors}\label{sec:normal}
Recall from (\ref{eq:hatM_def}) that $\hat M=M+Z\in\RR^{d_1\times d_2}$ with  $M=U\Lambda V^{\tran}$ where $U\in\RR^{d_1\times r}$ and $V\in\RR^{d_2\times r}$ satisfying $U^{\tran}U=I_r$ and $V^{\tran}V=I_r$. The diagonal matrix $\Lambda={\rm diag}(\lambda_1,\cdots,\lambda_r)$ contains non-increasing positive singular values of $M$.  Let $\hat U$ and $\hat V$ be $\hat M$'s top-$r$ left and right singular vectors. We derive the normal approximation of 
$$
{\rm dist}^2[(\hat U,\hat V), (U,V)]=\|\hat U\hat U^{\tran}-UU^{\tran}\|_{\rm F}^2+\|\hat V\hat V^{\tran}-VV^{\tran}\|_{\rm F}^2,
$$
which is often called the (squared) projection distance on Grassmannians.  To this end, we clarify important notations which shall appear frequently throughout the paper. 

To apply the representation formula from Theorem~\ref{thm:hat_Theta}, we turn $\hat M, M$ and $Z$ into symmetric matrices. 
For notational consistency,  we create $(d_1+d_2)\times (d_1+d_2)$ symmetric matrices as 
$$
\hat A =\left(\begin{array}{cc}0&\hat M\\ \hat M^{\tran}&0 \end{array}\right),\quad A =\left(\begin{array}{cc}0& M\\ M^{\tran}&0 \end{array}\right)\quad {\rm and}\quad X =\left(\begin{array}{cc}0& Z\\ Z^{\tran}&0 \end{array}\right).
$$
The model (\ref{eq:hatM_def}) is thus translated into $\hat A=A+X$. 
The symmetric matrix $A$ has eigenvalues $\lambda_1\geq \cdots\geq \lambda_r\geq \lambda_{-r}\geq \cdots\geq \lambda_{-1}$ where $\lambda_{-i}=-\lambda_i$ for $1\leq i\leq r$. 
The eigenvectors corresponding to $\lambda_i$ and $\lambda_{-i}$ are, respectively, 
$$
\theta_i=\frac{1}{\sqrt{2}}\left(\begin{array}{c}u_i\\ v_i\end{array}\right)\quad {\rm and}\quad \theta_{-i}=\frac{1}{\sqrt{2}}\left(\begin{array}{c}u_i\\ -v_i\end{array}\right)
$$
for $1\leq i\leq r$, where $\{u_i\}_{i=1}^r$ and $\{v_i\}_{i=1}^r$ are the columns of $U$ and $V$. Here, $\{\theta_i\}_{i=1}^r$ may not be uniquely defined if the singular value $\lambda_i$ has multiplicity larger than $1$. However, the spectral projector $UU^{\tran}$ and $VV^{\tran}$ are unique regardless of the multiplicities of $M$'s singular values. 

Following the same routine of notations, we denote 
$$
\Theta=(\theta_1,\cdots,\theta_r,\theta_{-r},\cdots,\theta_{-1})\in\RR^{(d_1+d_2)\times 2r}
$$ 
and $\Theta_{\perp}\in\RR^{(d_1+d_2)\times (d_1+d_2-2r)}$ such that $(\Theta,\Theta_{\perp})$ is an orthonormal matrix. Then, 
$$
\Theta\Theta^{\tran}=\sum_{1\leq |j|\leq r}\theta_j\theta_j^{\tran}=\left(\begin{array}{cc}UU^{\tran}&0\\0&VV^{\tran}\end{array}\right)
$$
and 
$$
\hat\Theta\hat\Theta^{\tran}=\sum_{1\leq |j|\leq r}\hat\theta_j\hat\theta_j^{\tran}=\left(\begin{array}{cc}\hat U\hat U^{\tran}&0\\0&\hat V\hat V^{\tran}\end{array}\right)
$$
where $\hat U$ and $\hat V$ represent $\hat M$'s top-$r$ left and right singular vectors. Similarly, for all $k\geq 1$, denote 
\begin{align*}
\mathfrak{P}^{-k}=\sum_{1\leq |j|\leq r}\frac{1}{\lambda_j^k}\theta_j\theta_j^{\tran}
=\begin{cases}
\left(\begin{array}{cc}0& U\Lambda^{-k}V^{\tran}\\V\Lambda^{-k}U^{\tran}&0\end{array}\right)&\textrm{if } k \textrm{ is odd}\\
&\\
\left(\begin{array}{cc}U\Lambda^{-k}U^{\tran}&0\\ 0&V\Lambda^{-k}V^{\tran}\end{array}\right)&\textrm{if } k \textrm{ is even}.
\end{cases}
\end{align*}
The orthogonal spectral projector is written as 
$$
\mathfrak{P}^{\perp}=\Theta_{\perp}\Theta_{\perp}^{\tran}=\left(\begin{array}{cc}U_{\perp}U_{\perp}^{\tran}&0\\ 0&V_{\perp}V_{\perp}^{\tran}\end{array}\right)
$$ 
where $(U,U_{\perp})$ and $(V,V_{\perp})$ are orthonormal matrices. Actually, the columns of $\Theta_{\perp}$ can be explicitly expressed by the columns of $U_{\perp}$ and $V_{\perp}$. Indeed, if we denote the columns of $\Theta_{\perp}\in\RR^{(d_1+d_2)\times (d_1+d_2-2r)}$ by 
$$
\Theta_{\perp}=(\theta_{r+1},\cdots,\theta_{d_1}, \theta_{-r-1},\cdots,\theta_{-d_2})
$$
, then we can write
$$
\theta_{j_1}=\left(\begin{array}{c}u_{j_1}\\ 0\end{array}\right)\quad {\rm and}\quad \theta_{-j_2}=\left(\begin{array}{c}0\\v_{j_2}\end{array}\right)
$$
for $r+1\leq j_1\leq d_1$ and $r+1\leq j_2\leq d_2$. 

By the above notations, it is clear that 
\begin{align*}
{\rm dist}^2[(\hat U,\hat V), (U,V)]=\|\hat\Theta\hat\Theta^{\tran}-\Theta\Theta^{\tran}\|_{\rm F}^2.
\end{align*}
It suffices to prove the normal approximation of $\|\hat\Theta\hat\Theta^{\tran}-\Theta\Theta^{\tran}\|_{\rm F}^2$. Observe that 
\begin{align*}
\||\hat\Theta\hat\Theta^{\tran}-\Theta\Theta^{\tran}\|_{\rm F}^2=&4r-2\big<\Theta\Theta^{\tran},\hat\Theta\hat\Theta^{\tran}\big>=-2\big<\Theta\Theta^{\tran},\hat\Theta\hat\Theta^{\tran}-\Theta\Theta^{\tran}\big>.
\end{align*}
By Theorem~\ref{thm:hat_Theta} and $\Theta\Theta^{\tran}\mathfrak{P}^{\perp}=0$, we can write 
\begin{align}
{\rm dist}^2[(\hat U, \hat V), (U,V)]=&-2\sum_{k\geq 2}\big<\Theta\Theta^{\tran}, \calS_{A,k}(X)\big>\nonumber\\
=&2\|\mathfrak{P}^{\perp}X\mathfrak{P}^{-1}\|_{\rm F}^2-2\sum_{k\geq 3}\big<\Theta\Theta^{\tran}, \calS_{A,k}(X)\big>.\label{eq:dist_formula}
\end{align}
where we used the fact $\mathfrak{P}^{\perp}\mathfrak{P}^{\perp}=\mathfrak{P}^{\perp}$ so that 
\begin{align*}
-2\big<\Theta\Theta^{\tran}, \calS_{A,2}\big>=&2\big<\Theta\Theta^{\tran},\mathfrak{P}^{-1}X\mathfrak{P}^{\perp}X\mathfrak{P}^{-1}\big>\\
=&2\tr\big(\mathfrak{P}^{-1}X\mathfrak{P}^{\perp}X\mathfrak{P}^{-1}\big)=2\|\mathfrak{P}^{\perp}X\mathfrak{P}^{-1}\|_{\rm F}^2.
\end{align*}
We prove CLT of ${\rm dist}^2[(\hat U, \hat V), (U,V)]$ with an explicit normalizing factor.  
\begin{theorem}\label{thm:distE_normal}
Suppose $d_{\max}\geq 3r$ where $d_{\max}=\max\{d_1,d_2\}$. There exist absolute constants $C_1,C_2,c_1>0$ such that if $\lambda_r\geq C_1\sqrt{d_{\max}}$, then for any $s\geq 1$,
\begin{align*}
\sup_{x\in\RR}\bigg|\PP\bigg(&\frac{{\rm dist}^2[(\hat U, \hat V), (U,V)]-\EE\ {\rm dist}^2[(\hat U, \hat V), (U,V)]}{\sqrt{8d_{\star}}\|\Lambda^{-2}\|_{\rm F}}\leq x\bigg)-\Phi(x) \bigg|\\
&\quad \leq C_2s^{1/2}\Big(\frac{\sqrt{r}}{\|\Lambda^{-2}\|_{\rm F}\lambda_r^2}\Big)\cdot \frac{(rd_{\max})^{1/2}}{\lambda_r}+e^{-s}+e^{-c_1d_{\max}}\\
&\hspace{2cm}+C_2\Big(\frac{\|\Lambda^{-1}\|_{\rm F}^4}{\|\Lambda^{-2}\|_{\rm F}^2}\Big)^{3/2}\cdot \frac{1}{\sqrt{d_{\max}}},
\end{align*}
where $d_{\star}=d_1+d_2-2r$ and $\Phi(x)$ denotes the c.d.f. of standard normal distributions. By setting $s=\frac{\lambda_r}{\sqrt{rd_{\max}}}$, we conclude that 
\begin{align*}
\sup_{x\in\RR}\bigg|\PP\bigg(&\frac{{\rm dist}^2[(\hat U, \hat V), (U,V)]-\EE\ {\rm dist}^2[(\hat U, \hat V), (U,V)]}{\sqrt{8d_{\star}}\|\Lambda^{-2}\|_{\rm F}}\leq x\bigg)-\Phi(x) \bigg|\\
&\quad \leq C_2\Big(\frac{\sqrt{r}}{\|\Lambda^{-2}\|_{\rm F}\lambda_r^2}\Big)\cdot \sqrt{\frac{(rd_{\max})^{1/2}}{\lambda_r}}+e^{-\lambda_r/\sqrt{rd_{\max}}}+e^{-c_1d_{\max}}\\
&\hspace{2cm}+C_2\Big(\frac{\|\Lambda^{-1}\|_{\rm F}^4}{\|\Lambda^{-2}\|_{\rm F}^2}\Big)^{3/2}\cdot \frac{1}{\sqrt{d_{\max}}}.
\end{align*}
\end{theorem}
By Theorem~\ref{thm:distE_normal}, the asymptotical normality holds as long as 
\begin{equation}\label{eq:asy_cond}
\Big(\frac{\sqrt{r}}{\|\Lambda^{-2}\|_{\rm F}\lambda_r^2}\Big)\cdot \sqrt{\frac{(rd_{\max})^{1/2}}{\lambda_r}}\to 0\quad {\rm and }\quad \Big(\frac{\|\Lambda^{-1}\|_{\rm F}^4}{\|\Lambda^{-2}\|_{\rm F}^2}\Big)^{3/2}\cdot \frac{1}{\sqrt{d_{\max}}}\to 0
\end{equation}
as $d_1,d_2\to \infty$.  If $\sqrt{r}=O(\lambda_r^2\|\Lambda^{-2}\|_{\rm F})$, then the first condition in (\ref{eq:asy_cond}) is equivalent to $\frac{\sqrt{rd_{\max}}}{\lambda_r}\to 0$. 
Such SNR condition {is optimal in the consistency regime}. In addition, Cauchy-Schwartz inequality implies that $\|\Lambda^{-1}\|_{\rm F}^4\leq r\cdot \|\Lambda^{-2}\|_{\rm F}^2$. Therefore, the second condition in (\ref{eq:asy_cond}) holds when 
$$
\frac{r^{3}}{d_{\max}}\to 0\quad \textrm{ as } d_1, d_2\to \infty.
$$
{Therefore,  $r$ is allowed to grow as fast as $o\big((d_1+d_2)^{1/3}\big)$. }

\begin{remark}
The normalization factor $\sqrt{8d_{\star}}\|\Lambda^{-2}\|_{\rm F}$ comes from the fact 
$$
{\rm Var}\big(2\|\mathfrak{P}^{-1}X\mathfrak{P}^{\perp}\|_{\rm F}^2\big)=8d_{\star}\|\Lambda^{-2}\|_{\rm F}^2. 
$$
{ We remark that Theorem~\ref{thm:distE_normal} is non-asymptotical and no constraints between $d_1$ and $d_2$ are needed.}
\end{remark}
Note that $\EE{\rm dist}^2[(\hat U,\hat V), (U,V)]$ in Theorem~\ref{thm:distE_normal} is not transparent yet. 
Calculating $\EE{\rm dist}^2[(\hat U,\hat V), (U,V)]$ needs delicate analysis. If we approximate $\EE{\rm dist}^2[(\hat U,\hat V), (U,V)]$ by its leading term $2\EE\|\mathfrak{P}^{-1}X\mathfrak{P}^{\perp}\|_{\rm F}^2$, we obtain 
$$
\EE\ {\rm dist}^2[(\hat U,\hat V), (U,V)]=[2+o(1)]\cdot d_{\star}\|\Lambda^{-1}\|_{\rm F}^2.
$$
The primary subject of section~\ref{sec:approx_Edist} is to approximate $\EE{\rm dist}^2[(\hat U,\hat V), (U,V)]$ to a higher accuracy. 

\section{Approximating the bias}\label{sec:approx_Edist}
Recall (\ref{eq:dist_formula}), we have
\begin{align*}
\EE\ {\rm dist}^2[(\hat U,\hat V), (U,V)]=2\EE\|\mathfrak{P}^{\perp}X\mathfrak{P}^{-1}\|_{\rm F}^2-2\sum_{k\geq 2}\EE\big<\Theta\Theta^{\tran}, \calS_{A,2k}(X)\big>
\end{align*}
where we used the fact $\EE\ \calS_{A,2k+1}(X)=0$ for any positive integer $k\geq 1$. 
We aim to determine $\EE\|\mathfrak{P}^{\perp}X\mathfrak{P}^{-1}\|_{\rm F}^2$ and $\EE\big<\Theta\Theta^{\tran}, \calS_{A,{2k}}(X)\big>$ for all $k\geq 2$. Apparently, by obtaining explicit formulas of $\EE\big<\Theta\Theta^{\tran}, \calS_{A,2k}(X)\big>$ for larger $k$s, we end up with more precise approximation of $\EE\ {\rm dist}^2[(\hat U,\hat V), (U,V)]$. In Lemma~\ref{lem:first_order_approx}-\ref{lem:fourth_order_approx}, we provide arbitrarily $k$-th order approximation of the bias.

\begin{lemma}[First order approximation]\label{lem:first_order_approx}
The following equation holds
$$
\EE\|\mathfrak{P}^{\perp}X\mathfrak{P}^{\perp}\|_{\rm F}^2=d_{\star}\|\Lambda^{-1}\|_{\rm F}^2
$$
where $d_{\star}=d_1+d_2-2r$. Moreover, if $\lambda_r\geq C_1\sqrt{d_{\max}}$ for some large enough constant $C_1>0$, then 
$$
\Big| \EE\ {\rm dist}^2[(\hat U,\hat V), (U,V)]-2d_{\star}\|\Lambda^{-1}\|_{\rm F}^2\Big|\leq C_2r\Big(\frac{d_{\max}}{\lambda_r^2}\Big)^2
$$
where $C_2>0$ is an absolute constant (depending on the constant $C_1$). 
\end{lemma}

In Lemma~\ref{lem:second_order_approx}, we calculate $\EE\big<\Theta\Theta^{\tran}, \calS_{A,4}(X)\big>$. 
 It yields the second order approximation of $\EE\ {\rm dist}^2[(\hat U,\hat V),(U,V)]$. 
\begin{lemma}[Second order approximation]\label{lem:second_order_approx}
The following fact holds
$$
\Big|\EE\big<\Theta\Theta^{\tran},\calS_{A,4}(X)\big>-\Delta_d^2\|\Lambda^{-2}\|_{\rm F}^2\Big|\leq C_2\frac{r^2d_{\max}}{\lambda_r^4}
$$
where  $d_{\star}=d_1+d_2-2r$ and $\Delta_d=d_1-d_2$ and $C_2$ is an absolute constant. 
Moreover, if $\lambda_r\geq C_1\sqrt{d_{\max}}$ for some large enough constant $C_1>0$, then 
\begin{align*}
\Big|\EE\ {\rm dist}^2[(\hat U,\hat V), (U,V)]-2\big(d_{\star}\|\Lambda^{-1}\|_{\rm F}^2-&\Delta_d^2\|\Lambda^{-2}\|_{\rm F}^2\big) \Big|\\
&\leq C_1\frac{r^2d_{\max}}{\lambda_r^4}+C_2r\Big(\frac{d_{\max}}{\lambda_r^2}\Big)^3
\end{align*}
where  $C_2,C_3>0$ are absolute constants (depending on $C_1$). 
\end{lemma}

In general, we calculate the arbitrary $k$-th order approximation in Lemma~\ref{lem:fourth_order_approx}. 

\begin{lemma}[Arbitrary $k$-th order approximation]\label{lem:fourth_order_approx}
For a positive integer $k\geq 2$ and $\sqrt{d_{\max}}\geq \log^2d_{\max}$ and $e^{-c_1d_{\max}}\leq \frac{1}{\sqrt{d_{\max}}}$,  the following fact holds
\begin{align*}
\Big|\EE\big<\Theta\Theta^{\tran},\calS_{A,2k}(X)\big> - (-1)^k(d_{1-}^{k-1}-d_{2-}^{k-1})&(d_{1-}-d_{2-})\|\Lambda^{-k}\|_{\rm F}^2 \big|\\
&\leq \frac{C_1(r^2+k)}{\sqrt{d_{\max}}}\cdot\Big(\frac{C_2d_{\max}}{\lambda_r^2}\Big)^k
\end{align*}
where  $c_1, C_1, C_2>0$ are some absolute constants. Then, the following bound holds
\begin{align*}
\Big|\EE\ {\rm dist}^2&[(\hat U,\hat V), (U,V)] - B_k\Big|\\
\leq& C_4\frac{r^2d_{\max}}{\lambda_r^4}+\frac{C_5r^2}{\sqrt{d_{\max}}}\cdot \Big(\frac{d_{\max}}{\lambda^2_r}\Big)^3+C_6r\Big(\frac{C_3d_{\max}}{\lambda_r^2}\Big)^{k+1}
\end{align*}
where $C_3,C_4,C_5,C_6$ are some absolute constants and $B_k$ is defined by
\begin{align}\label{eq:Bk_def}
B_k=2d_{\star}\|\Lambda^{-1}\|_{\rm F}^2-2\sum_{k_0=2}^k(-1)^{k_0}(d_{1-}^{k_0-1}-d_{2-}^{k_0-1})(d_{1-}-d_{2-})\|\Lambda^{-k_0}\|_{\rm F}^2.
\end{align}
\end{lemma}

The second and higher order terms involve the dimension difference $\Delta_d=d_1-d_2$. If $d_1=d_2$, these higher order approximations essentially have similar effects as the first order approximation. 

\begin{remark}\label{rmk:Binfty}
By choosing $k=\ceil{\log d_{\max}}$ so that $(C_3d_{\max}/\lambda_r^2)^{k+1}\lesssim  (d_{\max}/\lambda_r^2)^3/\sqrt{d_{\max}}$, we get 
$$
\Big|\EE\ {\rm dist}^2[(\hat U,\hat V), (U,V)] - B_{\ceil{\log d_{\max}}}\Big|\leq C_4\frac{r^2d_{\max}}{\lambda_r^4}+C_5\frac{r^2}{\sqrt{d_{\max}}}\cdot \Big(\frac{d_{\max}}{\lambda_r^2}\Big)^3
$$
for some absolute constants $C_4, C_5>0$. In addition, for each $1\leq j\leq r$, we have 
$$
2d_{1-}\lambda_j^{-2}-2\sum_{k= 2}^{\infty}(-1)^k(d_{1-}-d_{2-})d_{1-}^{k-1}\lambda_j^{-2k}=\frac{2d_{1-}(\lambda_j^2+d_{2-})}{\lambda_j^2(\lambda_j^2+d_{1-})}
$$
which matches $\EE\|\hat u_j\hat u_j^{\tran}-u_ju_j^{\tran}\|_{\rm F}^2$ developed in \cite[Theorem~2.9]{bao2018singular} if $\min\{\lambda_j-\lambda_{j+1}, \lambda_{j-1}-\lambda_j\}$ is bounded away from $0$ and {$r$ is fixed}. Similarly, we have 
$$
2d_{2-}\lambda_j^{-2}-2\sum_{k=2}^{\infty}(-1)^k(d_{2-}-d_{1-})d_{2-}^{k-1}\lambda_j^{-2k}=\frac{2d_{2-}(\lambda_j^2+d_{1-})}{\lambda_j^2(\lambda_j^2+d_{2-})}
$$
which matches $\EE\|\hat v_j\hat v_j^{\tran}-v_jv_j^{\tran}\|_{\rm F}^2$ developed in \cite[Theorem~2.3]{bao2018singular}. {Compared with \cite{bao2018singular}, our results are non-asymptotical. We impose no eigen-gap conditions and no upper bounds on $r$. }
\end{remark}

\begin{remark}
The proof of Lemma~\ref{lem:fourth_order_approx} imply that if $\lambda_r\geq C_1\sqrt{d_{\max}}$, then
$$
\EE\|\hat U\hat U^{\tran}-UU^{\tran}\|_{\rm F}^2=2\sum_{j=1}^r\frac{d_{1-}(\lambda_j^2+d_{2-})}{\lambda_j^2(\lambda_j^2+d_{1-})}+O\Big(\frac{r^2d_{\max}}{\lambda_r^4}+\frac{r^2}{\sqrt{d_{\max}}}\cdot\frac{d_{\max}^3}{\lambda_r^6}\Big)
$$
and
$$
\EE\|\hat V\hat V^{\tran}-VV^{\tran}\|_{\rm F}^2=2\sum_{j=1}^r\frac{d_{2-}(\lambda_j^2+d_{1-})}{\lambda_j^2(\lambda_j^2+d_{2-})}+O\Big(\frac{r^2d_{\max}}{\lambda_r^4}+\frac{r^2}{\sqrt{d_{\max}}}\cdot\frac{d_{\max}^3}{\lambda_r^6}\Big).
$$
\end{remark}

\section{Normal approximation after bias corrections}\label{sec:explicit_CLT}
In this section, we prove the normal approximation of ${\rm dist}^2[(\hat U,\hat V),(U,V)]$ with explicit centering and normalizing terms. 
By Theorem~\ref{thm:distE_normal}, it  suffices to substitute $\EE\ {\rm dist}^2[(\hat U,\hat V),(U,V)]$ with the explicit formulas from Lemma~\ref{lem:first_order_approx}-\ref{lem:fourth_order_approx}.

Similarly as in Section~\ref{sec:approx_Edist}, we consider arbitrarily $k$-th levels of bias corrections for  ${\rm dist}^2[(\hat U,\hat V),(U,V)]$. Higher order bias corrections, while involving more complicate bias reduction terms, require lower levels of SNR to guarantee the asymptotical normality. For instance, the first order bias correction in Theorem~\ref{thm:first_order_CLT} requires  $\lambda_r\gg \sqrt{rd_{\max}^{3/2}}$ for asymptotical normality, while the $\ceil{\log d_{\max}}$-th order bias correction in Theorem~\ref{thm:second_order_CLT} only requires optimal $\lambda_r\gg \sqrt{rd_{\max}}$ for asymptotical normality. {Again, the rank $r$ is allowed to diverge as fast as $o\big((d_1+d_2)^{1/3}\big)$.}

\begin{theorem}[First order CLT]\label{thm:first_order_CLT}
Suppose $d_{\max}\geq 3r$. There exist absolute constants $C_1,C_2,C_3,c_1>0$ such that if $\lambda_r\geq C_1\sqrt{d_{\max}}$, then, 
\begin{align*}
\sup_{x\in\RR}\bigg|\PP\bigg(&\frac{ {\rm dist}^2[(\hat U,\hat V),(U,V)]-B_1}{\sqrt{8d_{\star}}\|\Lambda^{-2}\|_{\rm F}}\leq x\bigg)-\Phi(x) \bigg|\\
\leq& C_2\Big(\frac{\sqrt{r}}{\|\Lambda^{-2}\|_{\rm F}\lambda_r^2}\Big)\cdot \sqrt{\frac{(rd_{\max})^{1/2}}{\lambda_r}}+e^{-c_1d_{\max}}+e^{-\lambda_r/\sqrt{rd_{\max}}}\\
&\hspace{2cm}+C_2\Big(\frac{\|\Lambda^{-1}\|_{\rm F}^4}{\|\Lambda^{-2}\|_{\rm F}^2}\Big)^{3/2}\cdot \frac{1}{\sqrt{d_{\max}}}+C_3\frac{rd_{\max}^{3/2}}{\lambda_r^2},
\end{align*}
where $d_{\star}=d_1+d_2-2r$ and $B_1$ is defined by (\ref{eq:Bk_def}). 
\end{theorem}
By Theorem~\ref{thm:first_order_CLT}, we conclude that
$$
\frac{{\rm dist}^2[(\hat U,\hat V),(U,V)]-2d_{\star}\|\Lambda^{-1}\|_{\rm F}^2}{\sqrt{8d_{\star}}\|\Lambda^{-2}\|_{\rm F}} \stackrel{{\rm d}}{\longrightarrow}\calN(0,1)
$$
as $d_1,d_2\to\infty$ if $\sqrt{r}=O(\|\Lambda^{-2}\|_{\rm F}\lambda_r^2)$ and
$$
\frac{\sqrt{rd_{\max}}+\sqrt{rd_{\max}^{3/2}}}{\lambda_r}\to 0\quad{\rm and}\quad  \Big(\frac{\|\Lambda^{-1}\|_{\rm F}^4}{\|\Lambda^{-2}\|_{\rm F}^2}\Big)^{3/2}\cdot \frac{1}{\sqrt{d_{\max}}}\to0.
$$
The above conditions require $\lambda_r\gg \sqrt{rd_{\max}^{3/2}}$ and $r^3\ll d_{\max}$. The order $d_{\max}^{3/4}$ is larger than the optimal rate $\sqrt{d_{\max}}$. It is improvable if we apply higher order bias corrections. 

\begin{theorem}[Arbitrary $k$-th order CLT]\label{thm:second_order_CLT}
Suppose that $d_{\max}\geq 3r$  and $k\geq 3$. There exist absolute constants $C_0,C_1,C_2,C_3,c_1>0$ such that if $\lambda_r\geq C_1\sqrt{d_{\max}}$, then, 
\begin{align*}
\sup_{x\in\RR}\bigg|\PP\bigg(&\frac{{\rm dist}^2[(\hat U,\hat V),(U,V)]-B_k}{\sqrt{8d_{\star}}\|\Lambda^{-2}\|_{\rm F}}\leq x\bigg)-\Phi(x) \bigg|\\
\leq& C_2\Big(\frac{\sqrt{r}}{\|\Lambda^{-2}\|_{\rm F}\lambda_r^2}\Big)\cdot \sqrt{\frac{(rd_{\max})^{1/2}}{\lambda_r}}+e^{-c_1d_{\max}}+e^{-\lambda_r/\sqrt{rd_{\max}}}+C_0\frac{r^2\sqrt{d_{\max}}}{\lambda_r^2}\\
&+C_2\Big(\frac{\|\Lambda^{-1}\|_{\rm F}^4}{\|\Lambda^{-2}\|_{\rm F}^2}\Big)^{3/2}\cdot \frac{1}{\sqrt{d_{\max}}}+C_1\frac{r^2d_{\max}^2}{\lambda_r^4}+C_2r\sqrt{d_{\max}}\cdot\Big(\frac{C_3d_{\max}}{\lambda_r^2}\Big)^k,
\end{align*}
where $B_k$ is defined by (\ref{eq:Bk_def}). 
\end{theorem}

By Theorem~\ref{thm:second_order_CLT}, the asymptotical normality of $\big({\rm dist}^2[(\hat U,\hat V),(U,V)]-B_k\big)/\sqrt{8d_{\star}}\|\Lambda^{-2}\|_{\rm F}$ requires
$$
\frac{\sqrt{r d_{\max}}+rd_{\max}^{1/4}+\sqrt{d_{\max}}\cdot (r^2d_{\max})^{1/4k}}{\lambda_r}\to 0
$$
as $d_{1}, d_{2}\to \infty$ when $\sqrt{r}=O(\|\Lambda^{-2}\|_{\rm F}\lambda_r^2)$. By choosing $k=\ceil{\log d_{\max}}$, it boils down to $\sqrt{rd_{\max}}/\lambda_r\to 0$ which is optimal in the consistency regime. Similarly as in Theorem~\ref{thm:distE_normal}, the condition $(\|\Lambda^{-1}\|_{\rm F}^4/\|\Lambda^{-2}\|_{\rm F}^2)^{3/2}/\sqrt{d_{\max}}\to 0$ requires that $r^3/d_{\max}\to 0$ as $d_1,d_2\to \infty$. 

\begin{remark}\label{rem:Binfty_1}
To avoid computing the sum of $k$ terms in $B_k$ (\ref{eq:Bk_def}), it suffices to apply $B_{\infty}$ which by Remark~\ref{rmk:Binfty} is
$$
B_{\infty}=2\sum_{j=1}^r\frac{1}{\lambda_j^2}\Big(d_{1-}\cdot\frac{\lambda_j^2+d_{2-}}{\lambda_j^2+d_{1-}}+d_{2-}\cdot\frac{\lambda_j^2+d_{1-}}{\lambda_j^2+d_{2-}}\Big).
$$
By setting $k=\infty$ in Theorem~\ref{thm:second_order_CLT}, we obtain 
\begin{align*}
\frac{{\rm dist}^2[(\hat U,\hat V),(U,V)]-B_\infty}{\sqrt{8d_{\star}}\|\Lambda^{-2}\|_{\rm F}}\to \mathcal{N}(0,1)
\end{align*}
as long as $\sqrt{rd_{\max}}/\lambda_r\to0$ and $r^3/d_{\max}\to 0$ when $d_{1},d_{2}\to \infty$. 
\end{remark}


\section{Confidence regions of singular subspaces}\label{sec:CR}
By the normal approximation of ${\rm dist}^2[(\hat U, \hat V), (U,V)]$ in Theorem~\ref{thm:second_order_CLT}, we construct confidence regions of $U$ and $V$.
The confidence regions of $(U,V)$ attain the pre-determined confidence level asymptotically. In the asymptotic scheme, we shall consider $d_1, d_2\to\infty$. Therefore, the  parameters $r^{(d_1,d_2)}$, $U^{(d_1,d_2)},V^{(d_1,d_2)}$ and $\Lambda^{(d_1,d_2)}$ also depend on  $d_1,d_2$. For notational simplicity, we omit the superscripts $(d_1,d_2)$ without causing confusions.

In particular, we set $k=\ceil{\log d_{\max}}$ in Theorem~\ref{thm:second_order_CLT} and get
$$
\frac{{\rm dist}^2[(\hat U,\hat V), (U,V)]-B_{\ceil{\log d_{\max}}}}{\sqrt{8d_{\star}}\|\Lambda^{-2}\|_{\rm F}} \stackrel{{\rm d}}{\longrightarrow} \calN(0,1)
$$
as $d_1,d_2\to+\infty$ when $\sqrt{r}=O(\lambda_r^2\|\Lambda^{-2}\|_{\rm F})$ and 
\begin{equation}\label{eq:asymp_cond}
\lim_{d_1,d_2\to\infty}\max\bigg\{\frac{\sqrt{r d_{\max}}+rd_{\max}^{1/4}}{\lambda_r}+\bigg(\frac{\|\Lambda^{-1}\|_{\rm F}^4}{\|\Lambda^{-2}\|_{\rm F}^2}\bigg)^{3/2}\cdot \frac{1}{\sqrt{d_{\max}}} \bigg\}=0.
\end{equation}
We define the confidence region based on $(\hat U,\hat V)$ by 
\begin{align*}
\calM_{\alpha}(\hat U, \hat V):=&\bigg\{(L,R): L\in\RR^{d_1\times r}, R\in\RR^{d_2\times r}, L^{\tran}L=R^{\tran}R=I_r\\
&,\big|{\rm dist}^2[(L,R),(\hat U,\hat V)]-  {B}_{\ceil{\log d_{\max}}}\big|\leq \sqrt{8d_{\star}}z_{\alpha/2}\|\Lambda^{-2}\|_{\rm F}\bigg\}
\end{align*}
where  $z_{\alpha}$ denotes the critical value of standard normal distribution, i.e., $z_{\alpha}=\Phi^{-1}(1-\alpha)$. Theorem~\ref{thm:data_CR} follows immediately from Theorem~\ref{thm:second_order_CLT}.
\begin{theorem}\label{thm:data_CR}
Suppose that conditions in Theorem~\ref{thm:second_order_CLT} hold. Then, for any $\alpha\in(0,1)$, we get 
\begin{align*}
\Big|\PP\big((U,V)&\in\calM_{\alpha}(\hat U,\hat V)\big) -(1-\alpha)\Big|\\
\leq&C_1\frac{\sqrt{r}}{\lambda_r^2\|\Lambda^{-2}\|_{\rm F}}\cdot \sqrt{\frac{(rd_{\max})^{1/2}}{\lambda_r}}+2e^{-c_1d_{\max}}+e^{-\lambda_r/\sqrt{rd_{\max}}}\\
&+C_2\bigg(\frac{\|\Lambda^{-1}\|_{\rm F}^4}{\|\Lambda^{-2}\|_{\rm F}^2}\bigg)^{3/2}\cdot \frac{1}{\sqrt{d_{\max}}}+C_3\frac{r^2\sqrt{d_{\max}}}{\lambda_r^2}+C_4\frac{r^2d_{\max}^2}{\lambda_r^4}
\end{align*}
for some absolute constants $C_1,C_2,C_3,C_4,c_1>0$. If condition (\ref{eq:asymp_cond}) holds, then 
$$
\lim_{d_1,d_2\to\infty}\PP\Big((U,V)\in\calM_\alpha(\hat U,\hat V)\Big)=1-\alpha. 
$$
\end{theorem}

\begin{remark}
We can also simply replace $B_{\ceil{\log d_{\max}}}$ with $B_{\infty}$ and Theorem~\ref{thm:data_CR} still holds under the same conditions. 
\end{remark}

\begin{remark}
Note that  $\Lambda$ is usually unknown. 
An immediate  choice is the empirical singular values $\hat\Lambda={\rm diag}(\hat\lambda_1,\cdots,\hat\lambda_r)$, i.e., top-$r$ singular values of $\hat M$.
It is well known that $\{\hat\lambda_j\}_{j=1}^r$ are biased estimators of $\{\lambda_j\}_{j=1}^r$. See \citep{benaych2012singular} and \citep{ding2017high} for more details. 

By \cite[{Theorem~2.2}]{ding2017high}, if $\lambda_r=O(\sqrt{d_{\max}})$ and some eigen-gap conditions hold, then with probability at least $1-d_{\max}^{-2}$, for all $1\leq j\leq r$,
\begin{equation}\label{eq:hat_lambdaj}
\Big|\hat \lambda_j^2-\Big(\lambda_j^2+(d_1+d_2)+\frac{d_1d_2}{\lambda_j^2}\Big)\Big|\leq C_1d_{\max}^{1/4}\lambda_j^{1/2}
\end{equation}
where $C_1>0$ is some absolute constant. { 
In the non-asymptotical settings, (\ref{eq:hat_lambdaj}) suggests that $|\hat\lambda_j^2-\lambda_j^2|\geq c_0 (d_1+d_2)$. Then,
\begin{align*}
\frac{d_{\star}\big|\|\Lambda^{-1}\|_{\rm F}^2-\|\hat\Lambda^{-1}\|_{\rm F}^2\big|}{\sqrt{8d_{\star}}\|\Lambda^{-2}\|_{\rm F}} \geq c_1 \frac{d_{\max}^{3/2}}{\lambda_r^2}
\end{align*}
for some absolute constants $c_0, c_1>0$. 
If we directly use $\|\hat\Lambda^{-1}\|_{\rm F}^2$  in Theorem~\ref{thm:first_order_CLT}, the non-asymptotical convergence rate reads $d_{\max}^{3/2}/\lambda_r^2$. It is indeed observed in simulations. See Section~\ref{sec:sim_data} for more details. 
}


{ Bound (\ref{eq:hat_lambdaj}) inspires  the following shrinkage estimator of $\lambda_j^2$:
\begin{equation}\label{eq:tilde_lambdaj_def}
\tilde\lambda_j^2=\frac{\hat\lambda_j^2-(d_1+d_2)}{2}+\frac{\sqrt{(\hat\lambda_j^2-(d_1+d_2))^2-4d_1d_2}}{2}\quad \textrm{for all } 1\leq j\leq r.
\end{equation}
By replacing $\Lambda$ with data-dependent estimates $\tilde\Lambda={\rm diag}(\tilde\lambda_1,\cdots,\tilde\lambda_r)$, it works extremely well in simulations. See Section~\ref{sec:sim_data} for more details. 

However, in order to theoretically justify these data-dependent estimates, we shall prove bound (\ref{eq:hat_lambdaj}) in the regime $\lambda_r\gg \sqrt{d_{\max}}$ and for divergent $r$. It is beyond the scope of this paper and we leave it as a future work.  Note that we can still apply (\ref{eq:tilde_lambdaj_def}) in practice since real-world applications are always in the non-asymptotic settings. 
}
\end{remark}

\section{Numerical experiments}\label{sec:sim}
For all the simulation cases considered below, we choose the rank $r=6$ and the singular values are set as $\lambda_i=2^{r-i}\cdot \lambda$ for $i=1,\cdots,r$ for some positive number $\lambda$. As a result, the signal strength is determined by $\lambda$. The true singular vectors $U\in\RR^{d_1\times r}$ and $V\in\RR^{d_2\times r}$ are computed from the left and right singular subspaces of a $d_1\times d_2$ Gaussian random matrix. 

\subsection{Higher order approximations of bias and normal approximation}
In {\it Simulation $1$}, we show the effectiveness of approximating $\EE{\rm dist}^2[(\hat U,\hat V), (U,V)]$ by the first order approximation $2d_{\star}\|\Lambda^{-1}\|_{\rm F}^2$ where $d_{\star}=d_1+d_2-2r$. Meanwhile, we show the inefficiency of first order approximation when $|d_1-d_2|\gtrsim \min(d_1,d_2)$. In {\it Simulation $2$}, we demonstrate the benefits of higher order approximations  when $|d_1-d_2|\gtrsim \min(d_1,d_2)$. 

{\it Simulation $1$.}  In this simulation, we study the accuracy of first order approximation and its relevance with $\Delta_d=d_1-d_2$. First, we set $d_1=d_2=d$ where $d=100, 200, 300$. The signal strength $\lambda$ is chosen as $30, 30.5,\cdots, 40$. For each given $\lambda$, the first order approximation $2d_{\star}\|\Lambda^{-1}\|_{\rm F}^2$ is recorded. To obtain $\EE{\rm dist}^2[(\hat U,\hat V), (U,V)]$, we repeat the experiments for $500$ times for each $\lambda$ and the average of ${\rm dist}^2[(\hat U,\hat V), (U,V)]$  is recorded, which denotes the simulated value of $\EE{\rm dist}^2[(\hat U,\hat V), (U,V)]$. We compare the simulated $\EE{\rm dist}^2[(\hat U,\hat V), (U,V)]$ with  $2d_{\star}\|\Lambda^{-1}\|_{\rm F}^2$, which is displayed in Figure~\ref{fig:E_1}. Since $d_1=d_2=d$, the first order approximation has similar effect as higher order approximation which is verified  by Figure~\ref{fig:E_1}.  Second, we set $d_1=\frac{d_2}{2}=d$ for $d=100,200,300$. As a result, $\Delta_d=d_2-d_1=d$ which is significantly large. Similar experiments are conducted and the results are displayed in Figure~\ref{fig:E_2}, which clearly shows that first order approximation is insufficient to estimate $\EE{\rm dist}^2[(\hat U,\hat V), (U,V)]$. Therefore, if $|d_1-d_2|\gg 0$, we need higher order approximation of $\EE{\rm dist}^2[(\hat U,\hat V), (U,V)]$. 
\begin{figure}
\centering     
\subfigure[First order approximation $2d_{\star}\|\Lambda^{-1}\|_{\rm F}^2$ is accurate when $\Delta_d=d_1-d_2=0$ and rank $r=6$. Here $d_{\star}=d_1+d_2-2r$. 
There is no need for higher order approximations.]{\label{fig:E_1}\includegraphics[height=100mm,width=60mm]{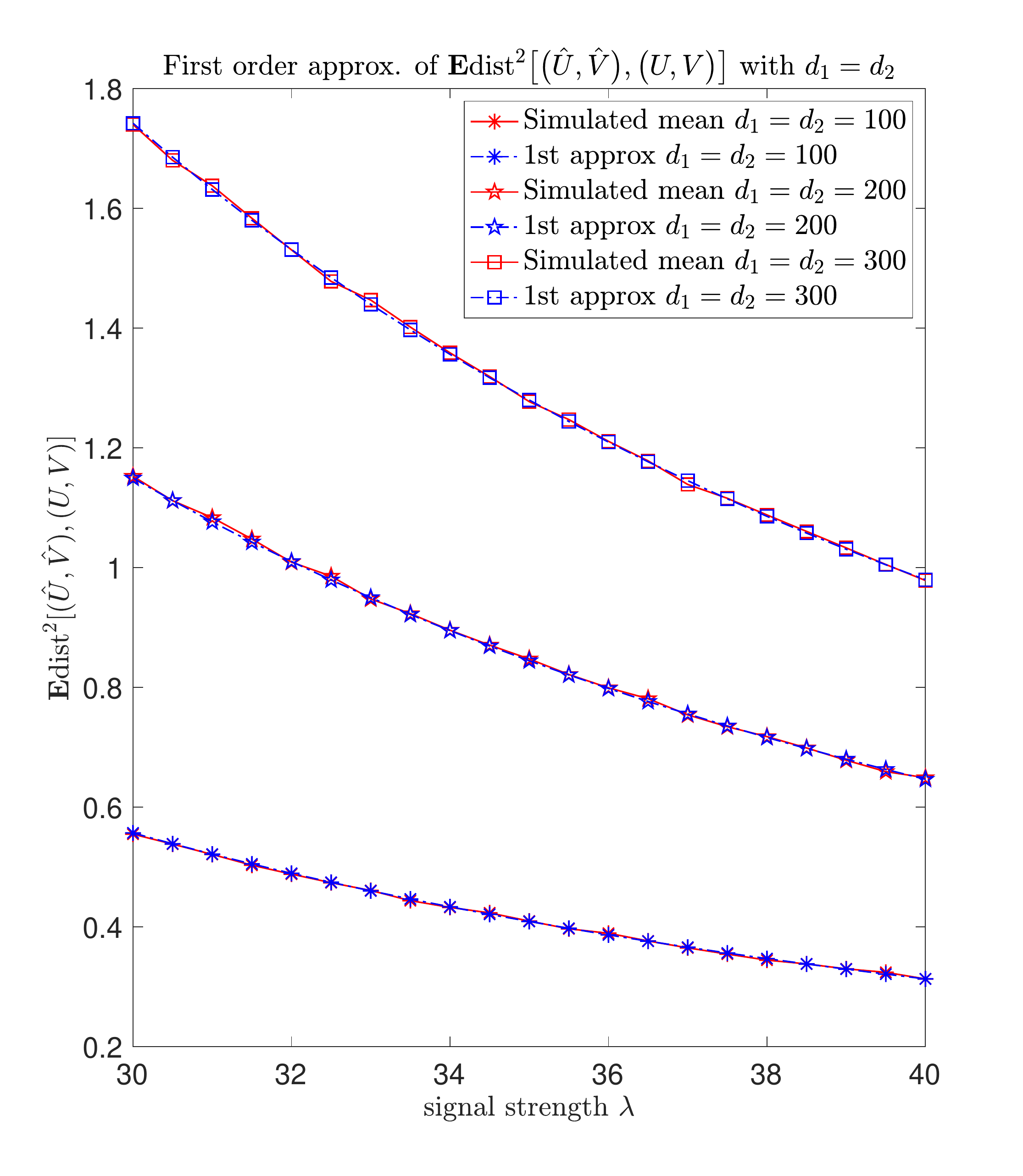}}
\subfigure[First order approximation $2d_{\star}\|\Lambda^{-1}\|_{\rm F}^2$ is not sufficiently accurate when $|d_1-d_2|\gg 0$. Here $d_{\star}=d_1+d_2-2r$ and rank $r=6$. 
The higher order approximations are indeed necessary.]{\label{fig:E_2}\includegraphics[height=100mm,width=60mm]{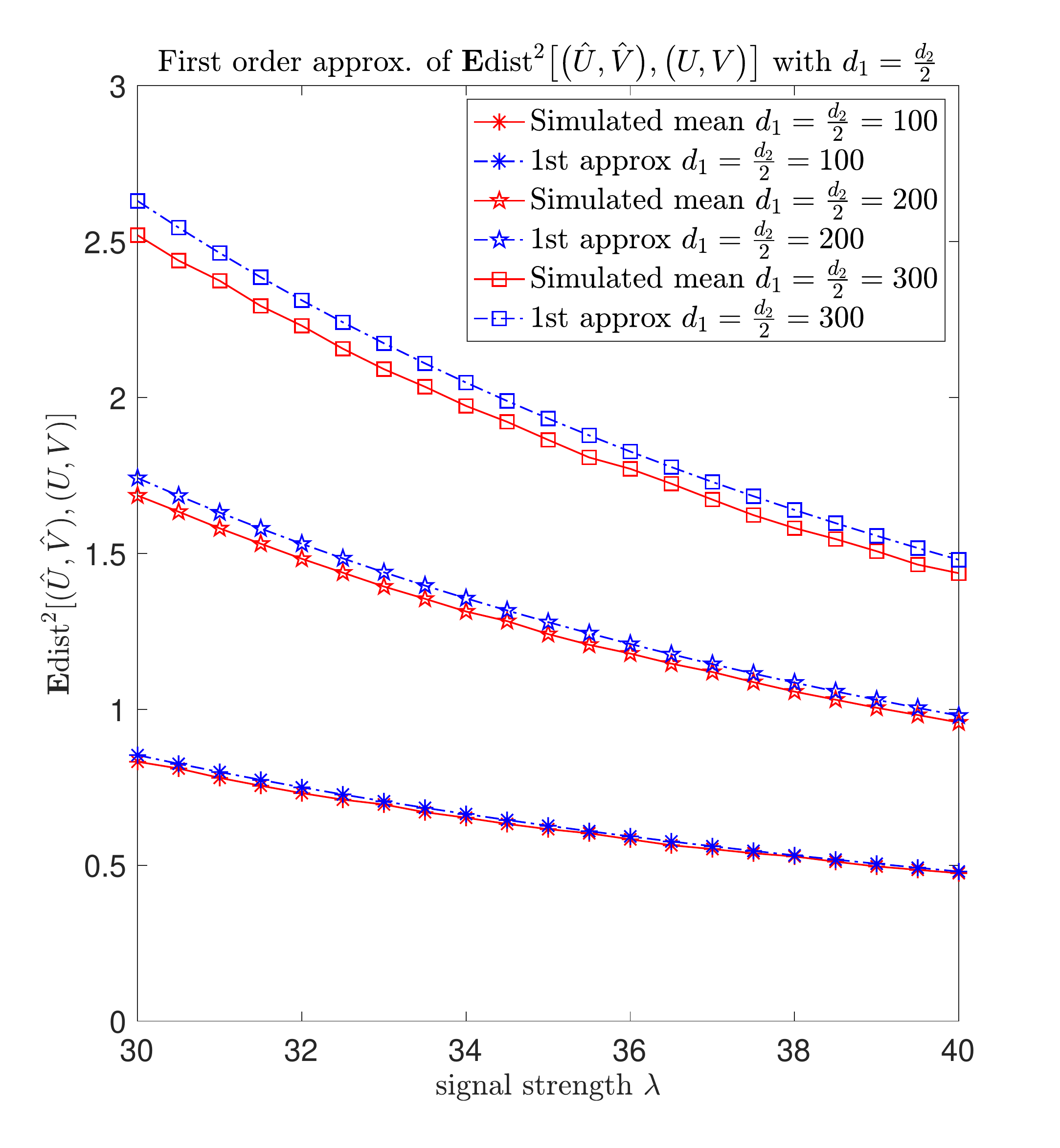}}
\caption{Comparison between $\EE{\rm dist}^2[(\hat U,\hat V), (U,V)]$ and the first order approximation: $2d_{\star}\|\Lambda^{-1}\|_{\rm F}^2$. It verifies that the accuracy of first order approximation depends on the dimension difference $\Delta_d=d_1-d_2$. Here the red curves represent the simulated mean $\EE{\rm dist}^2[(\hat U,\hat V), (U,V)]$ based on $500$ realizations of ${\rm dist}^2[(\hat U,\hat V), (U,V)]$. The blue curves are the theoretical first order approximations $2d_{\star}\|\Lambda^{-1}\|_{\rm F}^2$ based on Lemma~\ref{lem:first_order_approx}. The above left figure clearly shows that first order approximation is accurate if $d_1=d_2$. 
}
\end{figure}

{\it Simulation $2$.} In this simulation, we study the effects of higher order approximations when $|d_1-d_2|\gg 0$. More specifically, we choose $d_1=500$ and $d_2=1000$. The signal strength $\lambda=50,51,\cdots,60$. For each $\lambda$, we repeat the experiments for $500$ times producing $500$ realizations of ${\rm dist}^2[(\hat U,\hat V), (U,V)]$ whose average is recorded as the simulated $\EE{\rm dist}^2[(\hat U,\hat V), (U,V)]$. Meanwhile, for each $\lambda$, we record the $1$st-$4$th order approximations $B_1, B_2, B_3$ and $B_4$ which are defined by (\ref{eq:Bk_def}) . All the results are displayed in Figure~\ref{fig:E_3}. It verifies that higher order bias corrections indeed improve the accuracy of approximating $\EE{\rm dist}^2[(\hat U,\hat V), (U,V)]$. It also shows that the $1$st and $3$rd order approximations over-estimate $\EE{\rm dist}^2[(\hat U,\hat V), (U,V)]$; while, the $2$nd and $4$th order approximations under-estimate $\EE{\rm dist}^2[(\hat U,\hat V), (U,V)]$.
\begin{figure}
\centering
\includegraphics[height=140mm,width=120mm]{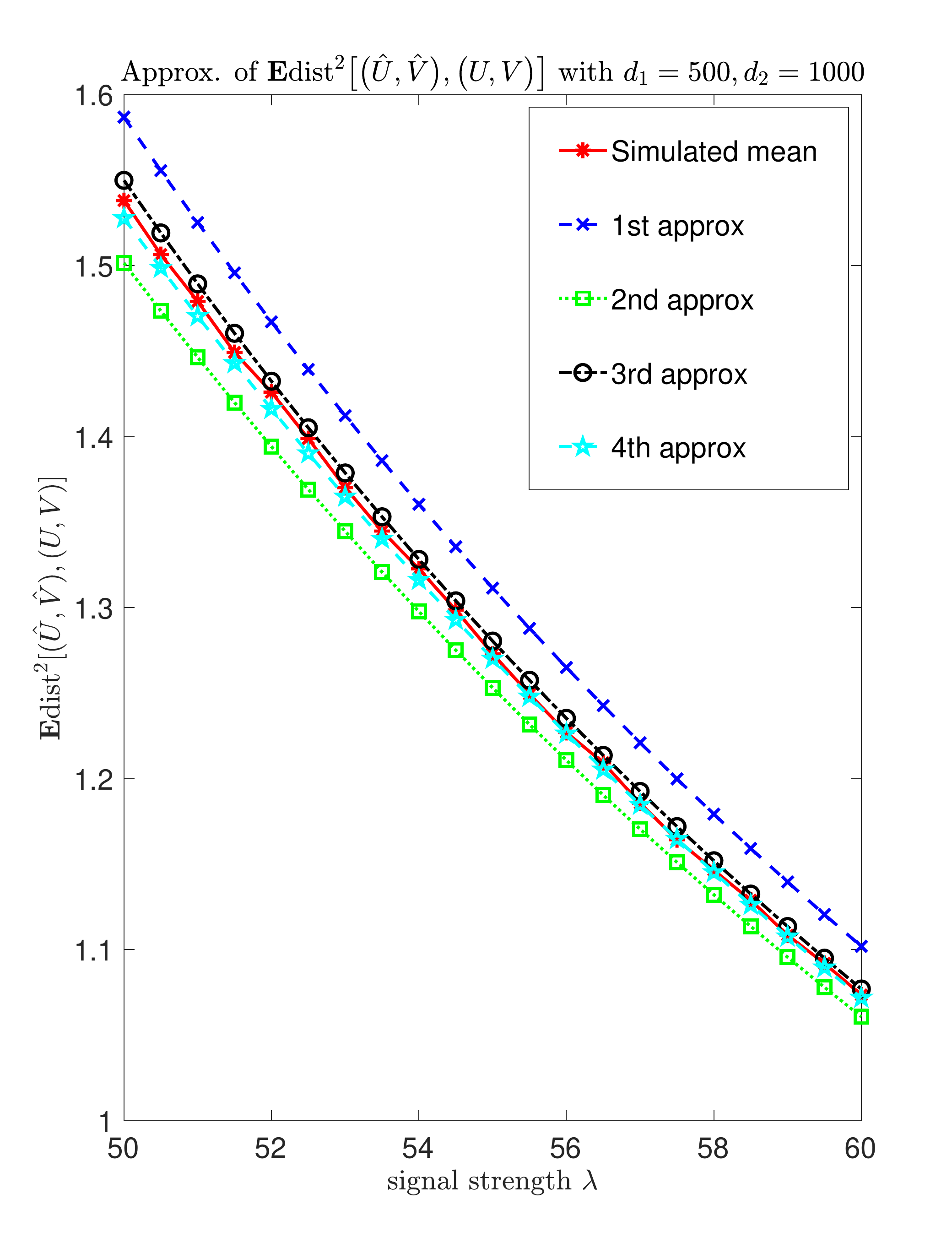}
\caption{The higher order approximations of $\EE{\rm dist}^2[(\hat U,\hat V), (U,V)]$. The simulated mean represents $\EE{\rm dist}^2[(\hat U,\hat V), (U,V)]$ calculated by the average of $500$ realizations of ${\rm dist}^2[(\hat U,\hat V), (U,V)]$. The $1$st order approximation is $2d_{\star}\|\Lambda^{-1}\|_{\rm F}^2$;  $2$nd order approximation is $2(d_{\star}\|\Lambda^{-1}\|_{\rm F}^2-\Delta_d^2\|\Lambda^{-2}\|_{\rm F}^2)$,  $3$rd order approximation is $2(d_{\star}\|\Lambda^{-1}\|_{\rm F}^2-\Delta_d^2\|\Lambda^{-2}\|_{\rm F}^2+d_{\star}\Delta_d^2\|\Lambda^{-3}\|_{\rm F}^2)$ and $4$th order approximation is $2(d_{\star}\|\Lambda^{-1}\|_{\rm F}^2-\Delta_d^2\|\Lambda^{-2}\|_{\rm F}^2+d_{\star}\Delta_d^2\|\Lambda^{-3}\|_{\rm F}^2-(d_{\star}^2-d_{1-}d_{2-})\Delta_d^2\|\Lambda^{-4}\|_{\rm F}^2)$ where $\Delta_d=d_1-d_2$, $d_{1-}=d_1-r$, $d_{2-}=d_2-r$ and $d_{\star}=d_{1-}+d_{2-}$ with $r=6$. Clearly, the $3$rd and $4$th order approximations are already close to the simulated mean.  We observe that the $1$st and $3$rd order approximations over-estimate $\EE{\rm dist}^2[(\hat U,\hat V), (U,V)]$; while, the $2$nd and $4$th order approximations under-estimate $\EE{\rm dist}^2[(\hat U,\hat V), (U,V)]$.}
\label{fig:E_3}
\end{figure}

{\it Simulation $3$.} We apply higher order approximations and show the normal approximation of $\big({\rm dist}^2[(\hat U,\hat V), (U,V)]- B_k\big)/\sqrt{8d_{\star}}\|\Lambda^{-2}\|_{\rm F}$ when $d_1=100, d_2=600$ and rank $r=6$.  We fixed the signal strength $\lambda=50$. The density histogram is based on $5000$ realizations from independent experiments.  We consider $1$st-$4$th order approximations, denoted by $\{B_k\}_{k=1}^4$. More specifically, 
$$
 B_1=2d_{\star}\|\Lambda^{-1}\|_{\rm F}^2,\quad{\rm and}\quad  B_2=2(d_{\star}\|\Lambda^{-1}\|_{\rm F}^2-\Delta_d^2\|\Lambda^{-2}\|_{\rm F}^2)
$$
and 
$$
 B_3=2(d_{\star}\|\Lambda^{-1}\|_{\rm F}^2-\Delta_d^2\|\Lambda^{-2}\|_{\rm F}^2+d_{\star}\Delta_d^2\|\Lambda^{-3}\|_{\rm F}^2)
$$
and 
$$
 B_4=2\big(d_{\star}\|\Lambda^{-1}\|_{\rm F}^2-\Delta_d^2\|\Lambda^{-2}\|_{\rm F}^2+d_{\star}\Delta_d^2\|\Lambda^{-3}\|_{\rm F}^2-(d_{1-}^3-d_{2-}^3)\Delta_d\|\Lambda^{-4}\|_{\rm F}^2\big).
$$
The results are shown in Figure~\ref{fig:normal_3}. This experiment aims to demonstrate the necessity of higher order bias corrections. Indeed, by the density histograms in Figure~\ref{fig:normal_3}, the first and second order bias corrections are not sufficiently strong to guarantee the normal approximations, at least when $\lambda\leq 50$, where the density histograms either shift leftward or rightward compared with the standard normal curve. On the other hand, after third or fourth order corrections, the normal approximation is very satisfactory at the same level of signal strength $\lambda=50$. 
\begin{figure}
\centering     
\subfigure[$ B_1=2d_{\star}\|\Lambda^{-1}\|_{\rm F}^2$]{\includegraphics[height=50mm,width=60mm]{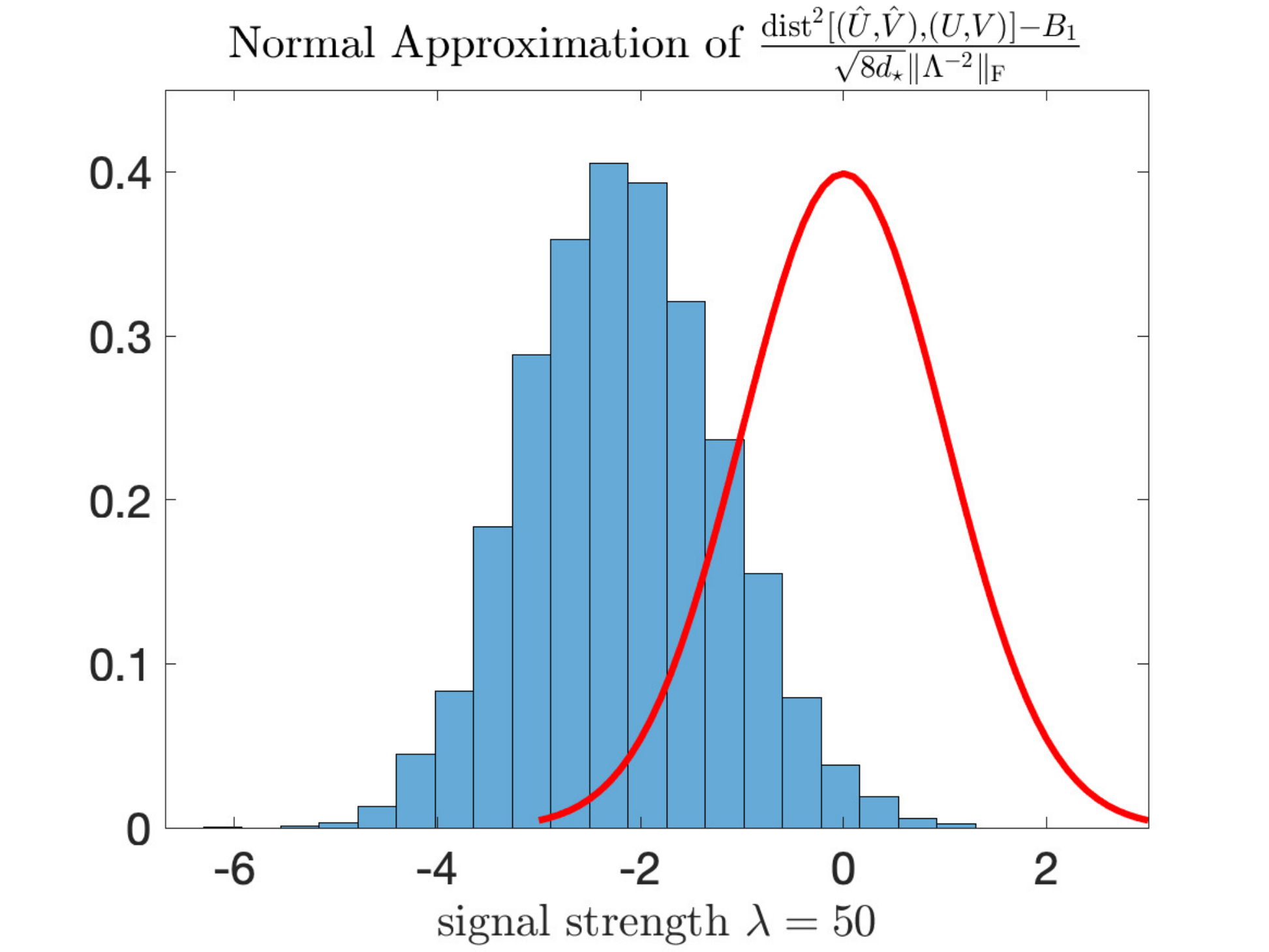}}
\subfigure[$ B_2=2(d_{\star}\|\Lambda^{-1}\|_{\rm F}^2-\Delta_d^2\|\Lambda^{-2}\|_{\rm F}^2)$]{\includegraphics[height=50mm,width=60mm]{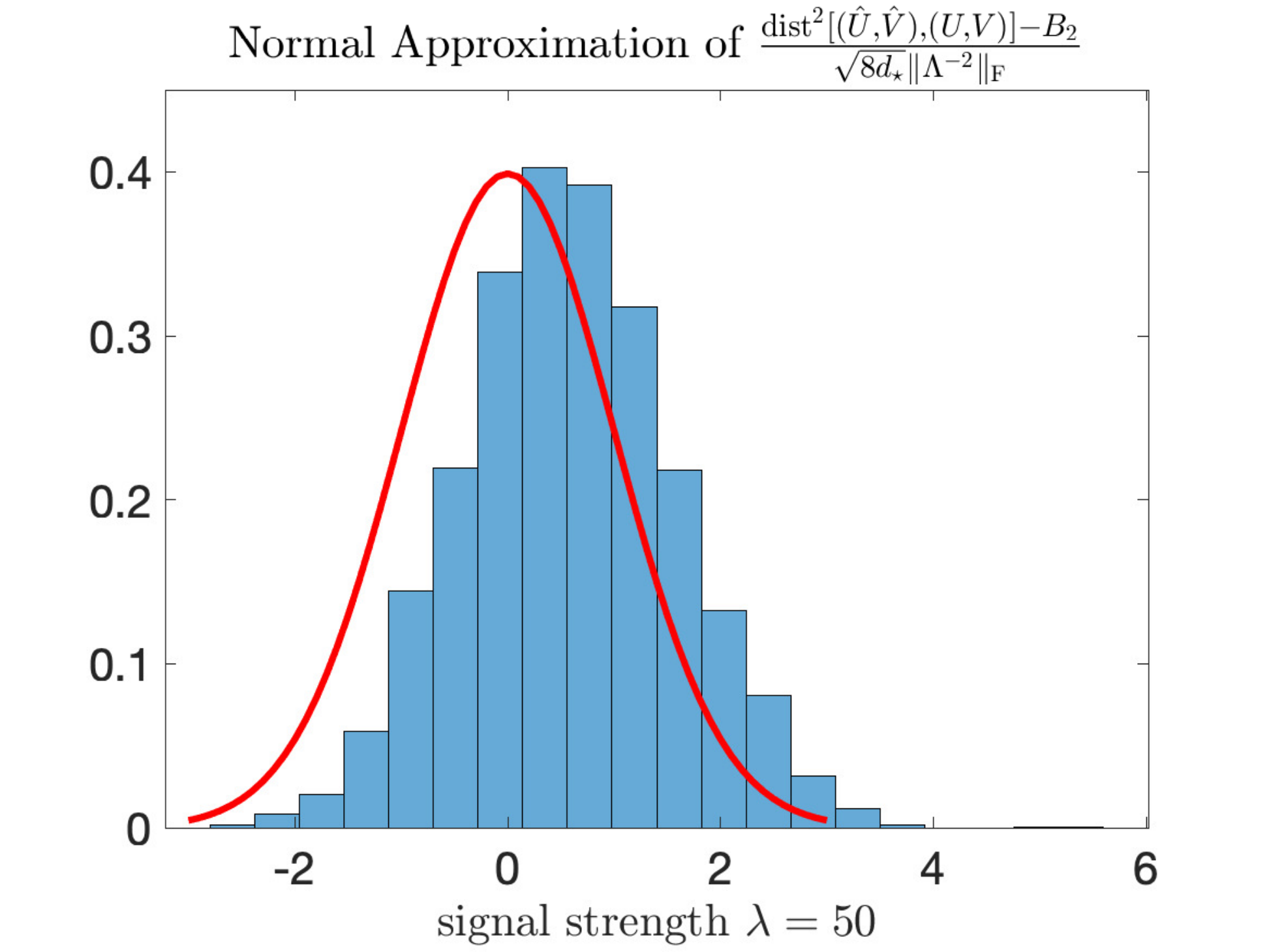}}\\
\subfigure[$ B_3=2(d_{\star}\|\Lambda^{-1}\|_{\rm F}^2-\Delta_d^2\|\Lambda^{-2}\|_{\rm F}^2+d_{\star}\Delta_d^2\|\Lambda^{-3}\|_{\rm F}^2$]{\includegraphics[height=50mm,width=60mm]{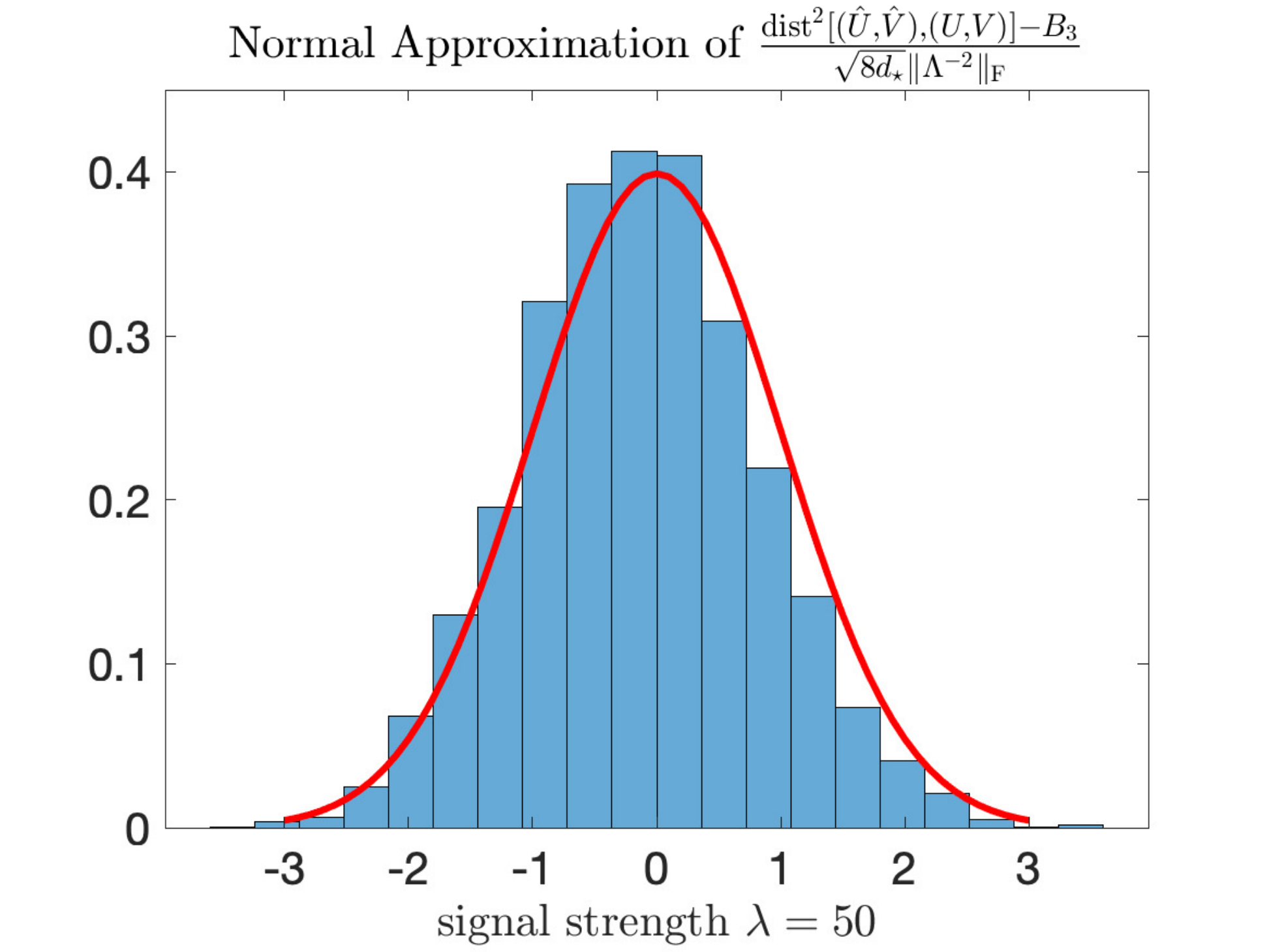}}
\subfigure[$ B_4=2(d_{\star}\|\Lambda^{-1}\|_{\rm F}^2-\Delta_d^2\|\Lambda^{-2}\|_{\rm F}^2+d_{\star}\Delta_d^2\|\Lambda^{-3}\|_{\rm F}^2-(d_1^3-d_2^3)(d_1-d_2)\|\Lambda^{-4}\|_{\rm F}^2)$]{\includegraphics[height=50mm,width=60mm]{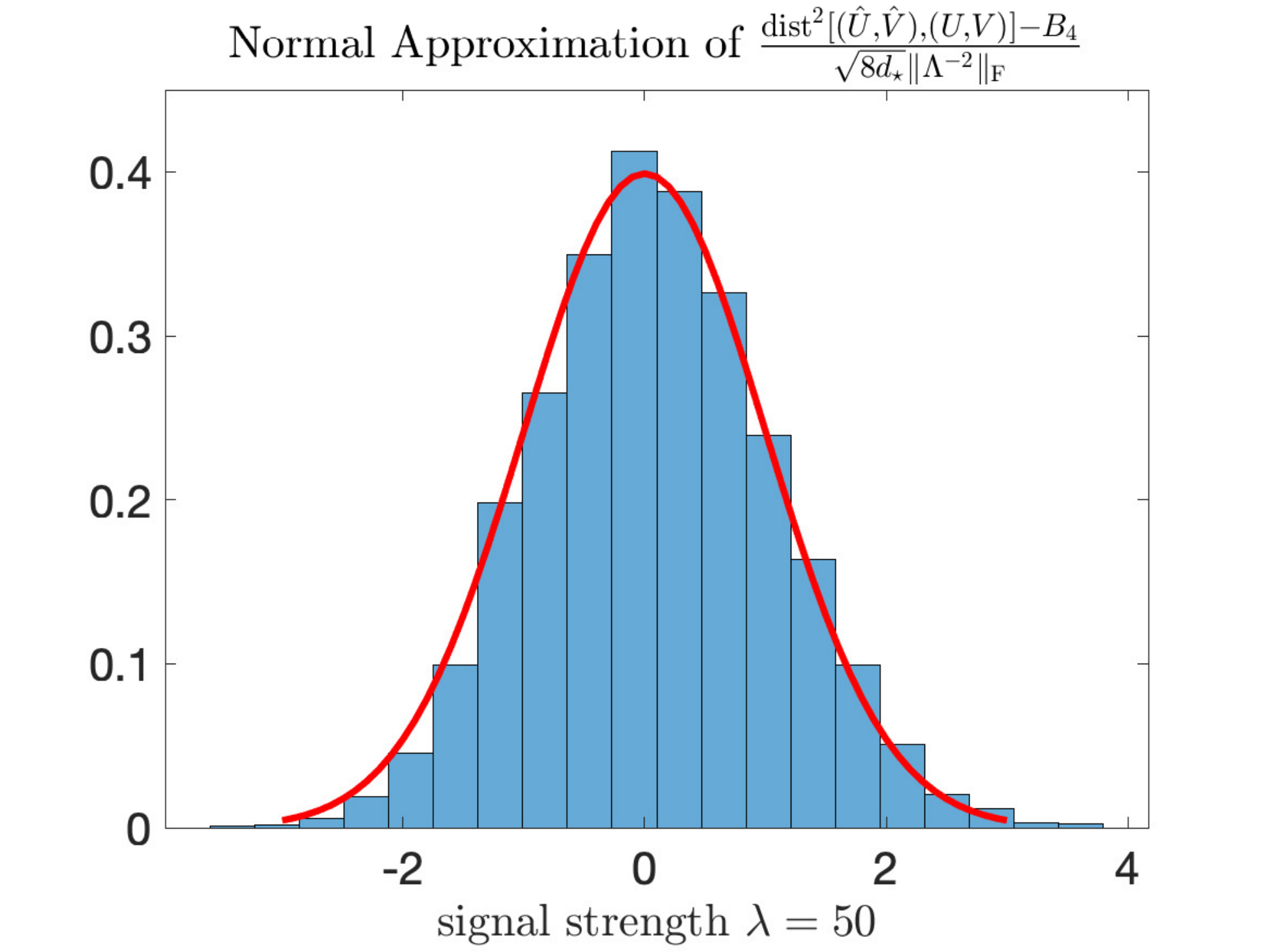}}
\caption{Normal approximation of $\frac{{\rm dist}^2[(\hat U,\hat V), (U,V)]- B_k}{\sqrt{8d_{\star}}\|\Lambda^{-2}\|_{\rm F}}$ with higher order bias corrections when $d_1=100, d_2=600$ and $r=6$. The density histogram is based on $5000$ realizations from independent experiments.  The red curve presents  p.d.f. of  standard normal distributions. Since $|d_1-d_2|\gg0$, this experiment demonstrates the necessity of higher order bias corrections. The bias correction $\hat B_k$ can be $1$st -$4$th order bias corrections.
}
\label{fig:normal_3}
\end{figure}

\subsection{Normal approximation with data-dependent bias corrections}\label{sec:sim_data}
Next, we show normal approximations of ${\rm dist}^2[(\hat U,\hat V), (U,V)]$ with data-dependent bias corrections and normalization factors. 

{\it Simulation $4$.} We apply the $1$st order approximation and show normal approximation of $\big({\rm dist}^2[(\hat U,\hat V), (U,V)]-2d_{\star}\|\hat\Lambda^{-1}\|_{\rm F}^2\big)/\sqrt{8d_{\star}}\|\hat\Lambda^{-2}\|_{\rm F}$ when $d_1=d_2=100$ and $r=6$.  Here, $\hat \Lambda={\rm diag}(\hat\lambda_1,\cdots,\hat\lambda_r)$ denotes the top-$r$ empirical singular values of $\hat M$. The signal strength $\lambda=25,50,65,75$. For each $\lambda$, we record $\big({\rm dist}^2[(\hat U,\hat V), (U,V)]-2d_{\star}\|\hat\Lambda^{-1}\|_{\rm F}^2\big)/\sqrt{8d_{\star}}\|\hat\Lambda^{-2}\|_{\rm F}$ from $5000$ thousand independent experiments and draw the density histogram. The p.d.f. of standard normal distribution is displayed by the red curve. The results are shown in Figure~\ref{fig:normal_1}. Since  each $\hat\lambda_j$ over-estimates the true $\lambda_j$, the bias correction $2d_{\star}\|\hat\Lambda^{-1}\|_{\rm F}^2$ is not sufficiently significant. It explains why the density histograms shift rightward compared with the standard normal curve, especially when signal strength $\lambda$ is moderately strong. 

\begin{figure}
\centering     
\subfigure[$\lambda=25$]{\includegraphics[height=50mm,width=60mm]{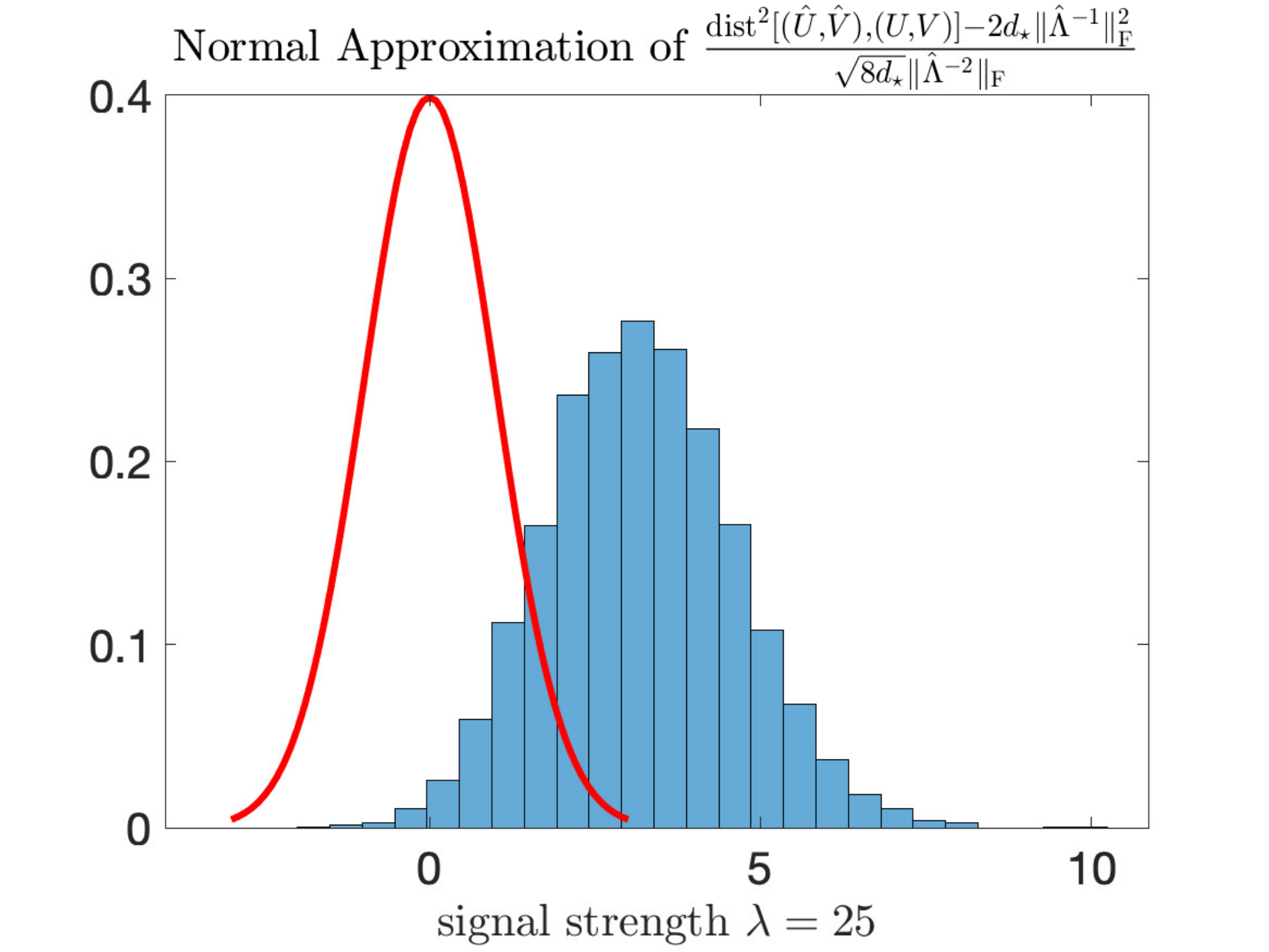}}
\subfigure[$\lambda=50$]{\includegraphics[height=50mm,width=60mm]{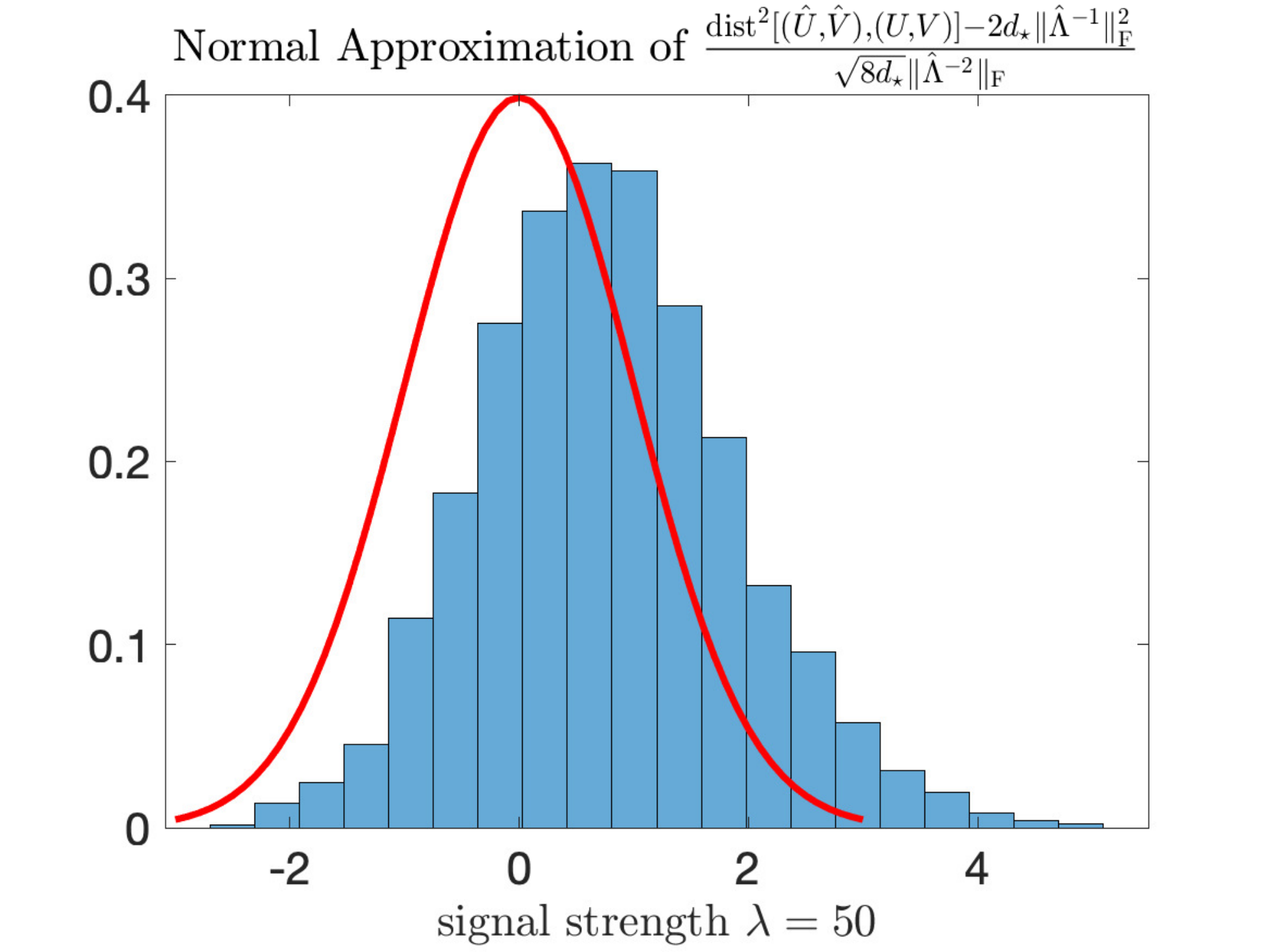}}\\
\subfigure[$\lambda=65$]{\includegraphics[height=50mm,width=60mm]{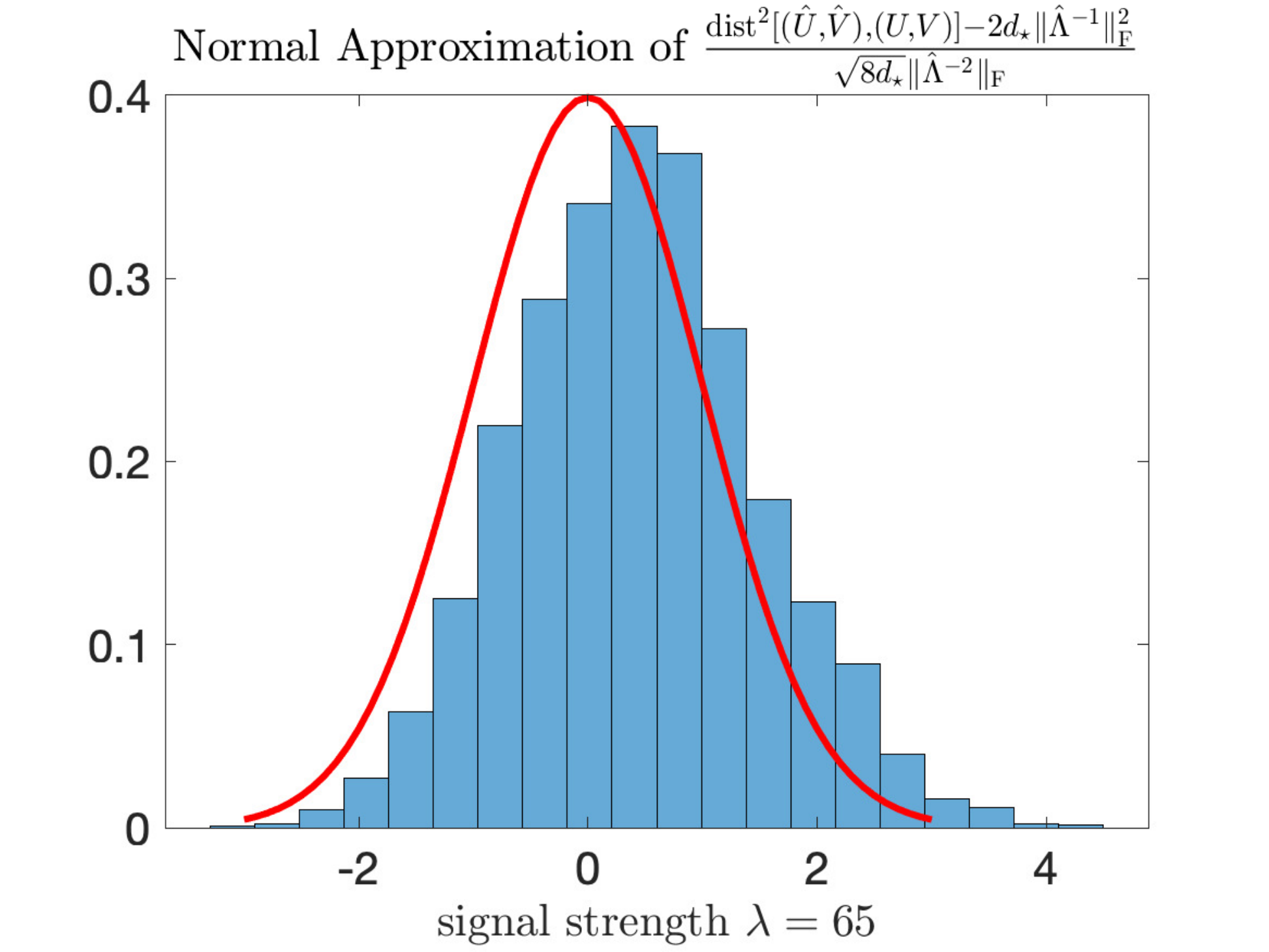}}
\subfigure[$\lambda=75$]{\includegraphics[height=50mm,width=60mm]{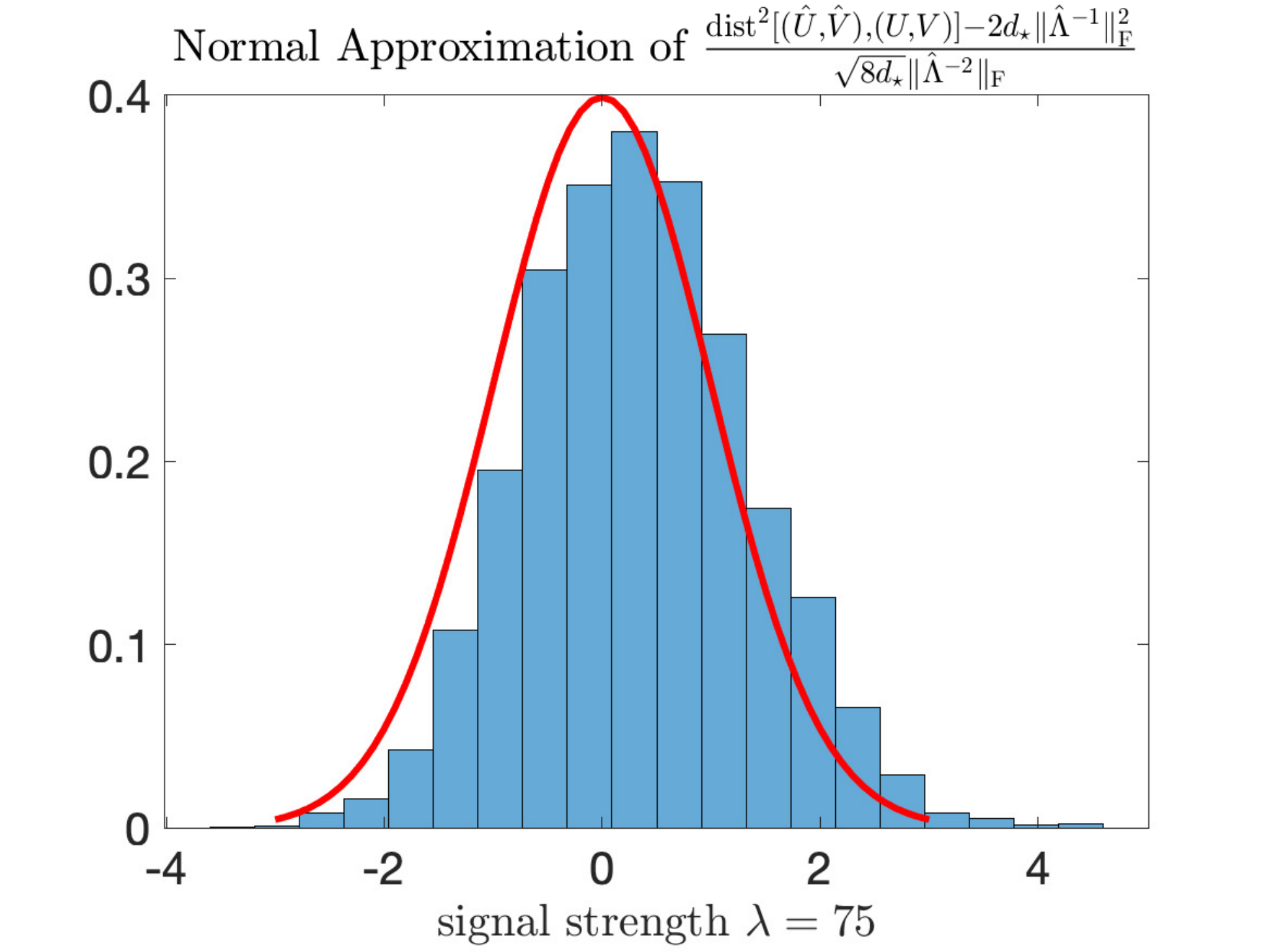}}
\caption{Normal approximation of $\frac{{\rm dist}^2[(\hat U,\hat V), (U,V)]-2d_{\star}\|\hat\Lambda^{-1}\|_{\rm F}^2}{\sqrt{8d_{\star}}\|\hat\Lambda^{-2}\|_{\rm F}}$ with $d_1=d_2=100$ and $r=6$. The density histogram is based on $5000$ realizations from independent experiments. The empirical singular values $\hat\Lambda={\rm diag}(\hat\lambda_1,\cdots,\hat\lambda_r)$ are calculated from $\hat M$. The red curve presents p.d.f. of  standard normal distributions. Since $\hat\lambda_j$ over-estimates $\lambda_j$, it explains why the density histogram shifts to the right compared with the standard normal curve, especially when signal strength $\lambda$ is not significantly strong. }
\label{fig:normal_1}
\end{figure}

{\it Simulation $5$.} 
We apply the $1$st order approximation and show normal approximation of $\big({\rm dist}^2[(\hat U,\hat V), (U,V)]-2d_{\star}\|\tilde\Lambda^{-1}\|_{\rm F}^2\big)/\sqrt{8d_{\star}}\|\tilde\Lambda^{-2}\|_{\rm F}$ when $d_1=d_2=100$ and $r=6$.   Here, $\tilde \Lambda={\rm diag}(\tilde\lambda_1,\cdots,\tilde\lambda_r)$ denotes the top-$r$ shrinkage estimators of $\lambda_j$s as in (\ref{eq:tilde_lambdaj_def}). The signal strength $\lambda=25,50,65,75$. For each $\lambda$, we record $\big({\rm dist}^2[(\hat U,\hat V), (U,V)]-2d_{\star}\|\tilde\Lambda^{-1}\|_{\rm F}^2\big)/\sqrt{8d_{\star}}\|\tilde\Lambda^{-2}\|_{\rm F}$ from $5000$ thousand independent experiments and draw the density histogram.  The results are shown in Figure~\ref{fig:normal_2}. In comparison with {\it Simulation $4$} and Figure~\ref{fig:normal_1}, we conclude that $2d_{\star}\|\tilde\Lambda^{-1}\|_{\rm F}^2$ works better than $2d_{\star}\|\hat\Lambda^{-1}\|_{\rm F}^2$ for bias corrections. Indeed, normal approximation of $\big({\rm dist}^2[(\hat U,\hat V), (U,V)]-2d_{\star}\|\tilde\Lambda^{-1}\|_{\rm F}^2\big)/\sqrt{8d_{\star}}\|\tilde\Lambda^{-2}\|_{\rm F}$  is already satisfactory when signal strength $\lambda=35$, compared with $\lambda\geq 75$ when $\hat\Lambda$ is used. 

\begin{figure}
\centering     
\subfigure[$\lambda=25$]{\includegraphics[height=50mm,width=60mm]{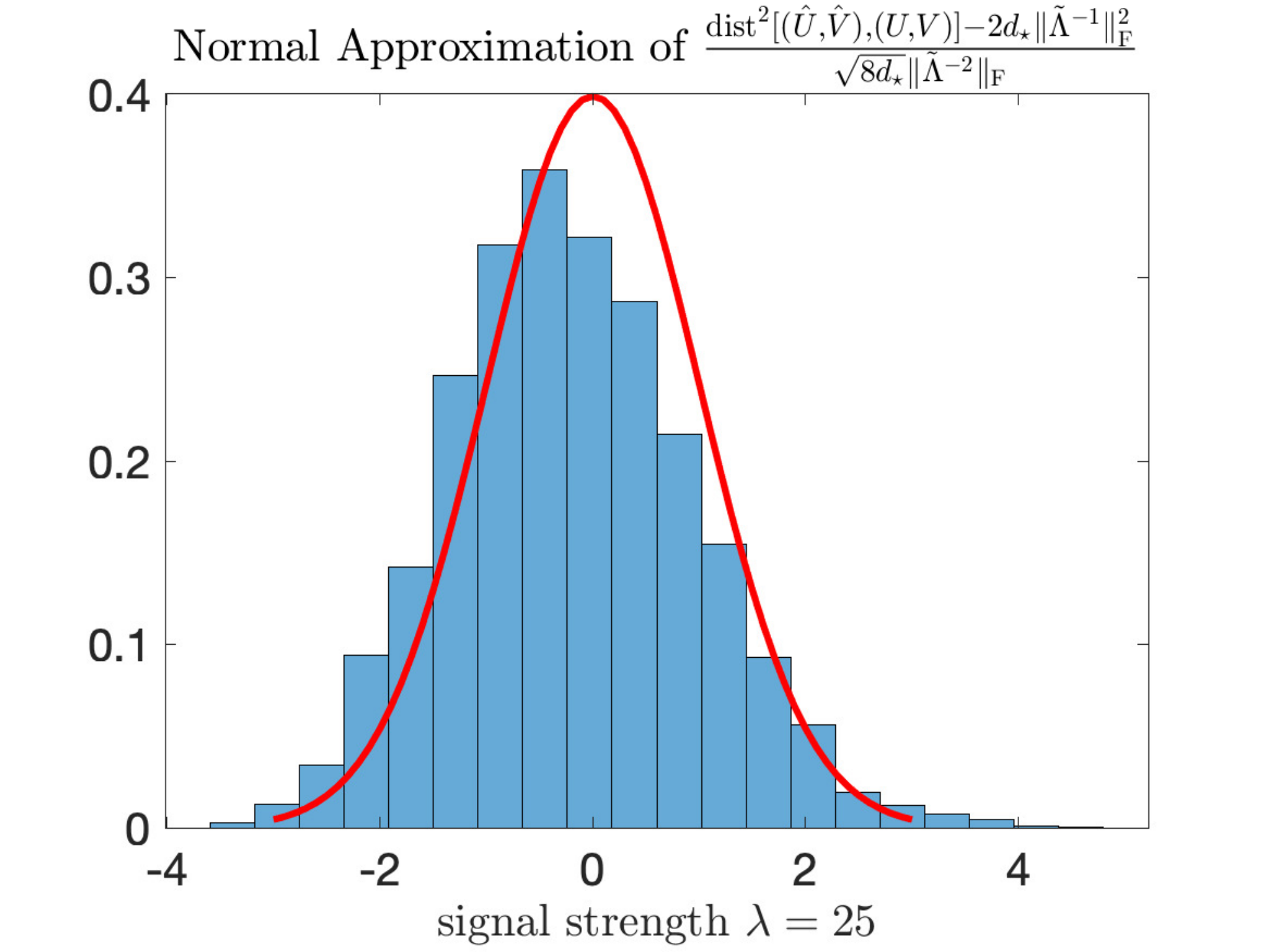}}
\subfigure[$\lambda=35$]{\includegraphics[height=50mm,width=60mm]{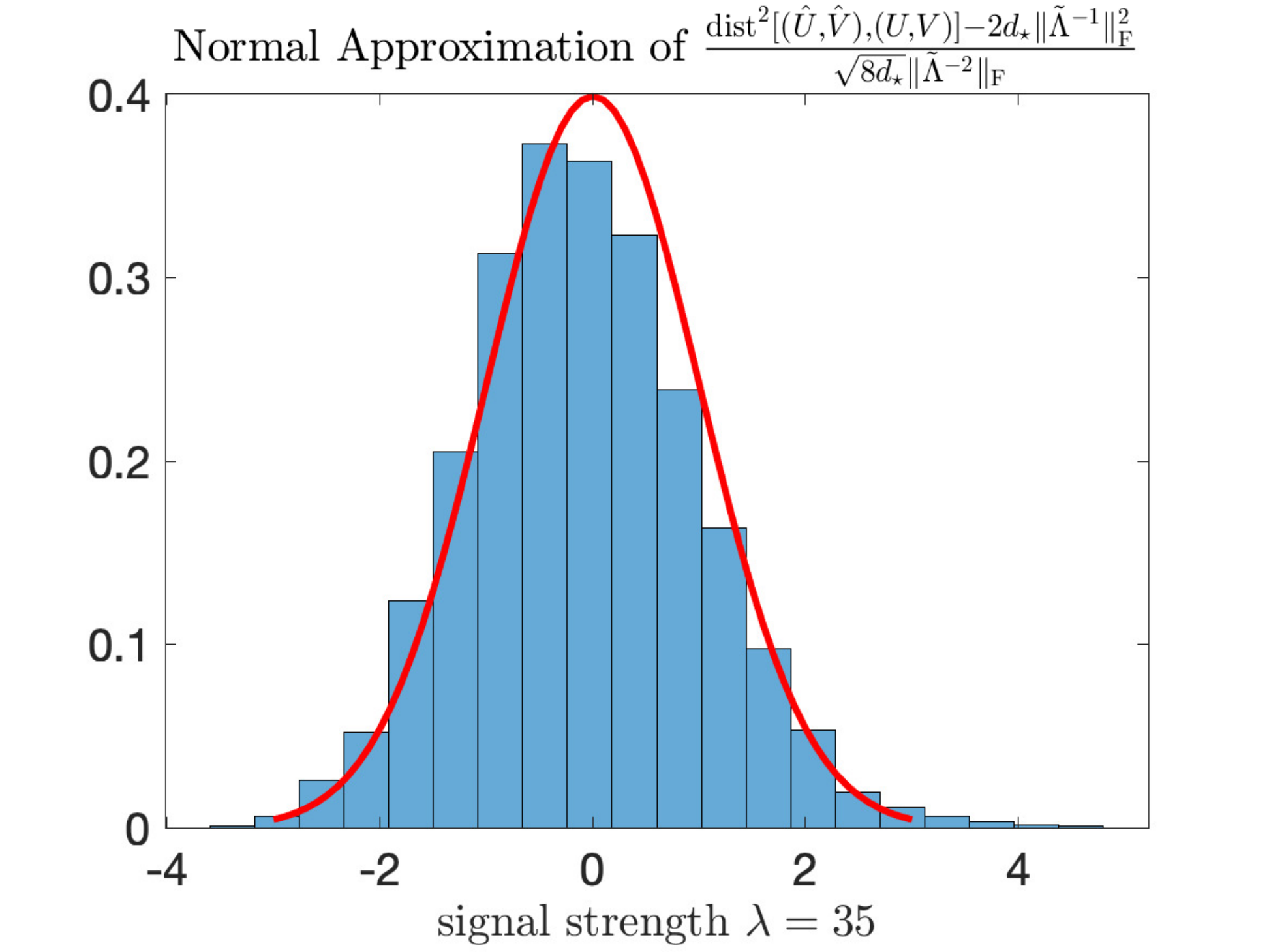}}\\
\subfigure[$\lambda=45$]{\includegraphics[height=50mm,width=60mm]{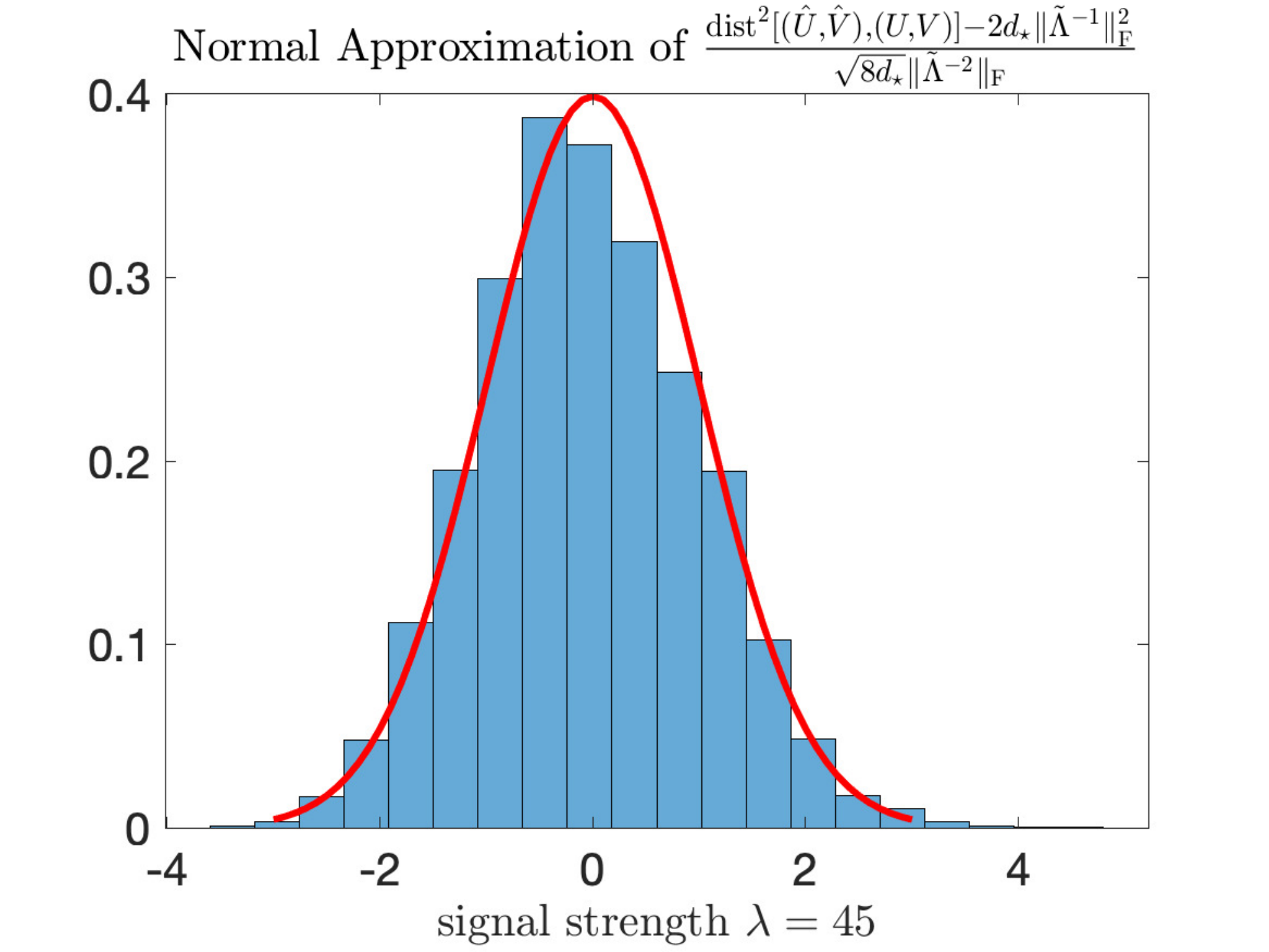}}
\subfigure[$\lambda=55$]{\includegraphics[height=50mm,width=60mm]{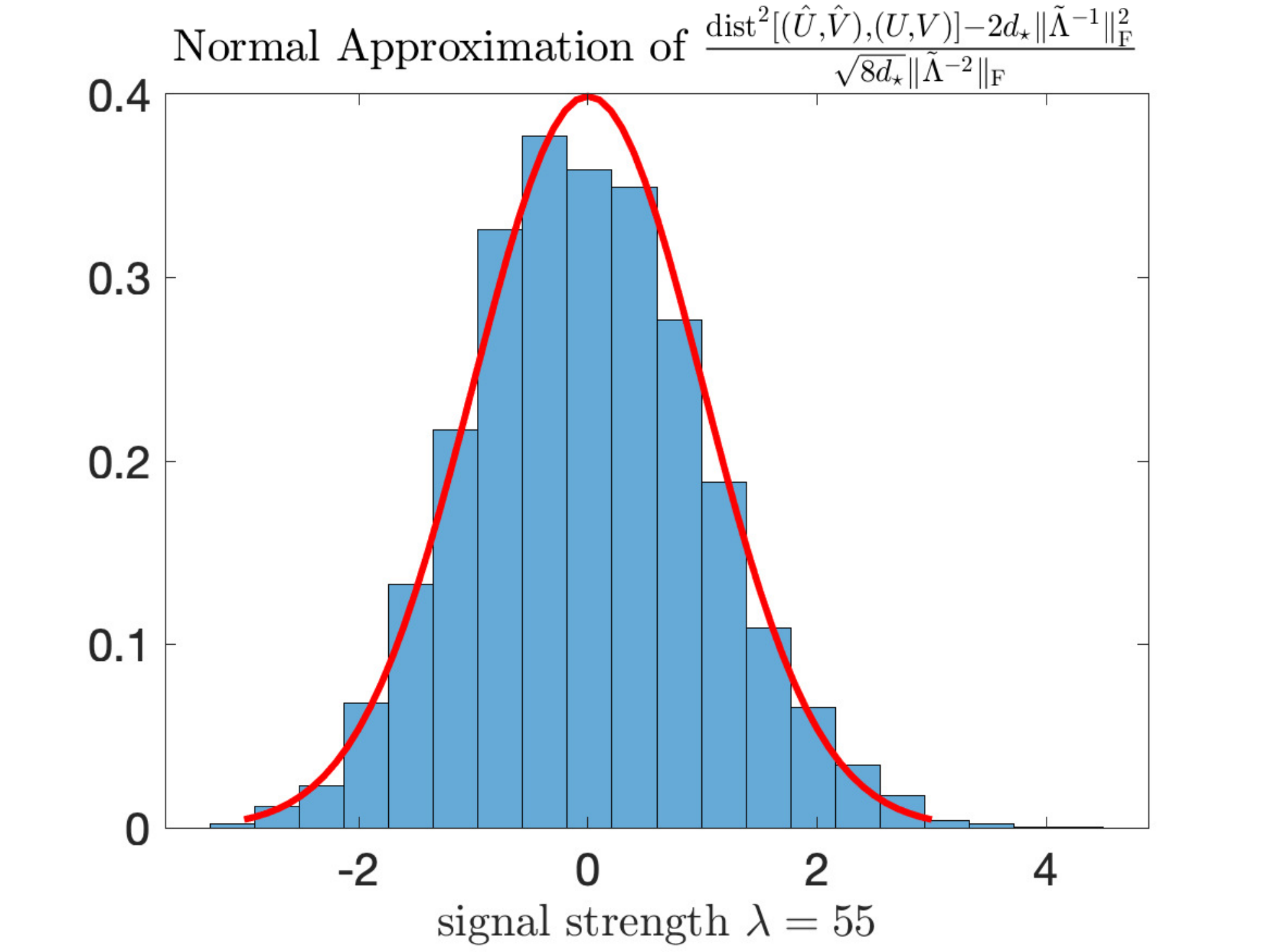}}
\caption{Normal approximation of $\frac{{\rm dist}^2[(\hat U,\hat V), (U,V)]-2d_{\star}\|\tilde\Lambda^{-1}\|_{\rm F}^2}{\sqrt{8d_{\star}}\|\tilde\Lambda^{-2}\|_{\rm F}}$ with $d_1=d_2=100$ and $r=6$. The density histogram is based on $5000$ realizations from independent experiments. The shrinkage estimators $\tilde\Lambda={\rm diag}(\tilde\lambda_1,\cdots,\tilde\lambda_r)$ are calculated as eq. (\ref{eq:tilde_lambdaj_def}). The red curve presents p.d.f. of  standard normal distributions. Since $d_1=d_2$, we apply  first order bias corrections to ${\rm dist}^2[(\hat U,\hat V), (U,V)]$. In comparison with {\it Simulation $4$} and Figure~\ref{fig:normal_1} where $\hat\Lambda$ is used instead of $\tilde\Lambda$, we conclude that $2d_{\star}\|\tilde\Lambda^{-1}\|_{\rm F}^2$ is more accurate than $2d_{\star}\|\hat\Lambda^{-1}\|_{\rm F}^2$ for bias corrections. Indeed, we see that normal approximation of $\frac{{\rm dist}^2[(\hat U,\hat V), (U,V)]-2d_{\star}\|\tilde\Lambda^{-1}\|_{\rm F}^2}{\sqrt{8d_{\star}}\|\tilde\Lambda^{-2}\|_{\rm F}}$  is already satisfactory when signal strength $\lambda=35$.
}
\label{fig:normal_2}
\end{figure}

\section{Acknowledgement}
The author would like to thank Yik-Man Chiang for the insightful recommendations on applying the Residue theorem, and Jeff Yao for the encouragements on improving the former results.

\section{Proofs}\label{sec:proof}
We only provide the proof of Theorem~\ref{thm:hat_Theta} in this section. Proofs of other theorems are collected in the supplementary file. 
\subsection{Proof of Theorem~\ref{thm:hat_Theta}}
For notational simplicity., we assume $\lambda_i>0$ for $1\leq i\leq r$, i.e., the matrix $A$ is positively semidefinite. The proof is almost identical if $A$ has negative eigenvalues. 

Since $A$ is positively semidefinite, we have $\min_{1\leq i\leq r}|\lambda_i|=\lambda_r$. The condition in Theorem~\ref{thm:hat_Theta} is equivalent to $\lambda_r>2\|X\|$. Recall that  $\{\hat\lambda_i, \hat\theta_i\}_{i=1}^{d}$ denote the singular values and singular vectors of $\hat A$. 
Define the following contour plot $\gamma_A$  on the complex plane (shown as in Figure~\ref{fig:contour}):
\begin{figure}[H]
\centering
\includegraphics[scale=0.5]{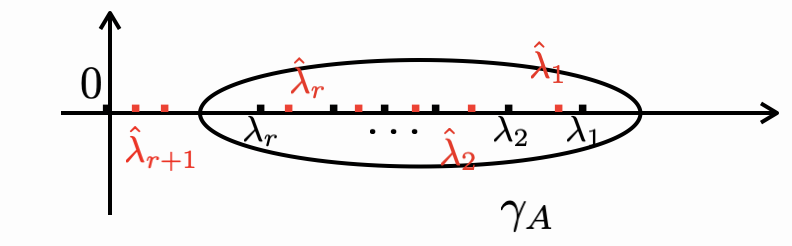}
\caption{The contour plot $\gamma_A$ which includes $\{\hat{\lambda}_i, \lambda_i\}_{i=1}^r$ leaving out $0$ and $\{\hat \lambda_i\}_{i=r+1}^d$. }
\label{fig:contour}
\end{figure}
{\noindent}, where the contour $\gamma_A$ is chosen such that $\min_{\eta\in\gamma_A}\min_{1\leq i\leq r}|\eta-\lambda_i|=\frac{\lambda_r}{2}$. 

Weyl's lemma implies that  $\max_{1\leq i\leq r}|\hat\lambda_i-\lambda_i|\leq \|X\|$. We observe that, when $\|X\|<\frac{\lambda_r}{2}$,  all $\{\hat\lambda_i\}_{i=1}^r$
are inside the contour $\gamma_A$ while $0$ and $\{\hat\lambda_i\}_{i=r+1}^d$ are outside of the contour $\gamma_A$. By Cauchy's integral formula, we get
\begin{align*}
\frac{1}{2\pi i}\oint_{\gamma_A}(\eta I-\hat A)^{-1}d\eta=&\sum_{i=1}^r\frac{1}{2\pi i}\oint_{\gamma_A}\frac{d\eta}{\eta-\hat \lambda_i}(\hat\theta_i \hat\theta_i^{\tran})+\sum_{i=r+1}^d\frac{1}{2\pi i}\oint_{\gamma_A}\frac{d\eta}{\eta-\hat \lambda_i}(\hat\theta_i \hat\theta_i^{\tran})\\
=&\sum_{i=1}^r \hat\theta_i\hat\theta_i^{\tran}=\hat\Theta\hat\Theta^{\tran}.
\end{align*}
As a result, we have
\begin{equation}\label{eq:hatThetahatTheta}
\hat\Theta\hat\Theta^{\tran}=\frac{1}{2\pi i}\oint_{\gamma_A}(\eta I-\hat A)^{-1}d\eta.
\end{equation}
Note that
\begin{align*}
(\eta I-\hat A)^{-1}=(\eta I - A- X)^{-1}=&\big[(\eta I-A)\big(I-\calR_A(\eta)X\big)\big]^{-1}\\
=&\big(I-\calR_A(\eta)X\big)^{-1}\calR_A(\eta)
\end{align*}
where $\calR_A(\eta):=(\eta I-A)^{-1}$. clearly 
$$
\big\|\calR_A(\eta)X \big\|\leq \|\calR_A(\eta)\|\|X\|\leq \frac{2\|X\|}{\lambda_r}< 1.
$$
Therefore, we write the Neumann series:
\begin{equation}\label{eq:neumann}
\big(I-\calR_A(\eta)X\big)^{-1}=I+\sum_{j\geq 1}[\calR_A(\eta)X]^j.
\end{equation}
By (\ref{eq:neumann}) and (\ref{eq:hatThetahatTheta}),  we get
\begin{align*}
\hat\Theta\hat\Theta^{\tran}=&\frac{1}{2\pi i}\oint_{\gamma_A}(\eta I-\hat A)^{-1}d\eta\\
=&\frac{1}{2\pi i}\oint_{\gamma_A}\calR_A(\eta) d\eta +\sum_{j\geq 1}\frac{1}{2\pi i}\oint_{\gamma_A}\big[\calR_A(\eta)X\big]^j\calR_A(\eta)d\eta.
\end{align*}
Clearly, $\frac{1}{2\pi i}\oint_{\gamma_A}\calR_A(\eta)d\eta=\Theta\Theta^{\tran}$, we end up with
$$
\hat\Theta\hat\Theta^{\tran}-\Theta\Theta^{\tran}=\calS_{A}(X):=\sum_{j\geq 1}\frac{1}{2\pi i}\oint_{\gamma_A}\big[\calR_A(\eta)X\big]^j\calR_A(\eta)d\eta.
$$
For $k\geq 1$, we define
\begin{equation}\label{eq:calS_Ak}
\calS_{A,k}(X)=\frac{1}{2\pi i}\oint_{\gamma_A}\big[\calR_A(\eta)X\big]^k\calR_A(\eta)d\eta
\end{equation}
which is essentially the $k$-th order perturbation. Therefore, we obtain
\begin{equation}\label{eq:hatTheta-Theta}
\hat\Theta\hat\Theta^{\tran}-\Theta\Theta^{\tran}=\sum_{k\geq 1}\calS_{A,k}(X).
\end{equation}
By (\ref{eq:hatTheta-Theta}), it suffices to derive explicit expression formulas for $\{\calS_{A,k}(X)\}_{k\geq 1}$.  Before dealing with general $k$, let us  derive $\calS_{A,k}(X)$ for $k=1,2$ to interpret the shared styles. 

To this end, we denote $I_r$ the $r\times r$ identity matrix and write 
$$
\calR_A(\eta)=\Theta(\eta\cdot I_r-\Lambda)^{-1}\Theta^{\tran}+\eta^{-1}\Theta_{\perp}\Theta_{\perp}^{\tran}=\sum_{j=1}^{d}\frac{1}{\eta-\lambda_j}\theta_j\theta_j^{\tran}
$$
where we set $\lambda_j=0$ for all $r+1\leq j\leq d$. Denote $P_j=\theta_j\theta_j^{\tran}$ for all $1\leq j\leq d$ which represents the spectral projector onto $\theta_j$. 

\paragraph{Derivation of $\calS_{A,1}(X)$.}
By the definition of $\calS_{A,1}(X)$,
\begin{align}
\calS_{A,1}(X)=&\frac{1}{2\pi i}\oint_{\gamma_A}\calR_{A}(\eta)X\calR_{A}(\eta)d\eta\nonumber\\
=&\sum_{j_1=1}^{d}\sum_{j_2=1}^{d}\frac{1}{2\pi i}\oint_{\gamma_A}\frac{d\eta}{(\eta-\lambda_{j_1})(\eta-\lambda_{j_2})}P_{j_1}XP_{j_2}.\label{eq:lemSM1_eq1}
\end{align}
{\it Case 1:} both $j_1$ and $j_2$ are greater than $r$. In this case, the contour integral in (\ref{eq:lemSM1_eq1}) is zero by Cauchy integral formula.\\
{\it Case 2}: only one of $j_1$ and $j_2$ is greater than $r$. W.L.O.G, let $j_2>r$, we get 
\begin{align*}
\sum_{j_1=1}^r\sum_{j_2>r}^d\frac{1}{2\pi i}&\oint_{\gamma_A} \frac{\eta^{-1}d\eta}{\eta-\lambda_{j_1}}P_{j_1}XP_{j_2}
=\sum_{j_1=1}^r\sum_{j_2>r}\lambda_{j_1}^{-1}P_{j_1}XP_{j_2}=\mathfrak{P}^{-1}X\mathfrak{P}^{\perp}. 
\end{align*}
{\it Case 3}: none of $j_1$ and $j_2$ is greater than $r$. Clearly, the contour integral in (\ref{eq:lemSM1_eq1}) is zero. 

To sum up, we conclude with $\calS_{A,1}(X)=\mathfrak{P}^{-1}X\mathfrak{P}^{\perp}+\mathfrak{P}^{\perp}X\mathfrak{P}^{-1}$.

\paragraph{Derivation of $\calS_{A,2}(X)$.}
By the definition of $\calS_{A,2}(X)$, 
\begin{align}
\calS_{A,2}(X)=&\frac{1}{2\pi i}\oint_{\gamma_A}\calR_A(\eta)X\calR_{A}(\eta)X\calR_{A}(\eta)d\eta\nonumber\\
=&\sum_{j_1=1}^{d}\sum_{j_2=1}^{d}\sum_{j_3=1}^{d}\frac{1}{2\pi i}\oint_{\gamma_A}\frac{d\eta}{(\eta-\lambda_{j_1})(\eta-\lambda_{j_2})(\eta-\lambda_{j_3})}P_{j_1}XP_{j_2}XP_{j_3}.\label{eq:lemSM2_eq1}
\end{align}
{\it Case 1}: all $j_1, j_2, j_3$ are greater than $r$. The contour integral in (\ref{eq:lemSM2_eq1}) is zero by Cauchy integral formula. \\
{\it Case 2}: two of $j_1,j_2,j_3$ are greater than $r$. W.L.O.G., let $j_1\leq r$ and $j_2,j_3>r$, we get
\begin{align*}
\sum_{j_1=1}^{r}\sum_{j_2,j_3>r}^{d}\frac{1}{2\pi i }&\oint_{\gamma_A}\frac{\eta^{-2}d\eta}{\eta-\lambda_{j_1}}P_{j_1}XP_{j_2}XP_{j_3}\\
=&\sum_{j_1=1}^{r}\sum_{j_2,j_3>r}^{d}\frac{1}{\lambda_{j_1}^2}P_{j_1}XP_{j_2}XP_{j_3}=\mathfrak{P}^{-2}X\mathfrak{P}^{\perp}X\mathfrak{P}^{\perp}. 
\end{align*}
{\it Case 3}: one of $j_1,j_2,j_3$ is greater than $r$. W.L.O.G., let $j_1,j_2\leq r$ and $j_3>r$, we get
\begin{align*}
\sum_{j_1,j_2=1}^{r}\sum_{j_3>r}^{d}\frac{1}{2\pi i}&\oint_{\gamma_A}\frac{\eta^{-1} d\eta}{(\eta-\lambda_{j_1})(\eta-\lambda_{j_2})}P_{j_1}XP_{j_2}XP_{j_3}\\
=&\sum_{j_1=j_2=1}^{r}\sum_{j_3>r}^{d}\frac{1}{2\pi i}\oint_{\gamma_A}\frac{\eta^{-1} d\eta}{(\eta-\lambda_{j_1})^2}P_{j_1}XP_{j_1}XP_{j_3}\\
&\quad +\sum_{j_1\neq j_2\geq 1}^{r}\sum_{j_3>r}^{d}\frac{1}{2\pi i}\oint_{\gamma_A}\frac{\eta^{-1} d\eta}{(\eta-\lambda_{j_1})(\eta-\lambda_{j_2})}P_{j_1}XP_{j_2}XP_{j_3}\\
=&-\sum_{j_1=1}^r\lambda_{j_1}^{-2}P_{j_1}XP_{j_1}X\mathfrak{P}^{\perp}-\sum_{j_1\neq j_2\geq 1}^{r}(\lambda_{j_1}\lambda_{j_2})^{-1}P_{j_1}XP_{j_2}X\mathfrak{P}^{\perp}\\
=&-\mathfrak{P}^{-1}X\mathfrak{P}^{-1}X\mathfrak{P}^{\perp}.
\end{align*}
{\it Case 4:} none of $j_1,j_2,j_3$ is greater than $r$. Clearly, the contour integral in (\ref{eq:lemSM2_eq1}) is zero. 

To sum up, we obtain 
\begin{align*}
\calS_{A,2}(X)=&\big(\mathfrak{P}^{-2}X\mathfrak{P}^{\perp}X\mathfrak{P}^{\perp}+\mathfrak{P}^{\perp}X\mathfrak{P}^{-2}X\mathfrak{P}^{\perp}+\mathfrak{P}^{\perp}X\mathfrak{P}^{\perp}X\mathfrak{P}^{-2}\big)\\
-&\big(\mathfrak{P}^{\perp}X\mathfrak{P}^{-1}X\mathfrak{P}^{-1}+\mathfrak{P}^{-1}X\mathfrak{P}^{\perp}X\mathfrak{P}^{-1}+\mathfrak{P}^{-1}X\mathfrak{P}^{-1}X\mathfrak{P}^{\perp}\big).
\end{align*}

\paragraph{Derivation of $\calS_{A,k}(X)$ for general k.}
Recall the definition of $\calS_{A,k}(X)$, we write
\begin{align}
\calS_{A,k}(X)=\sum_{j_1,\cdots,j_{k+1}\geq 1}^d\frac{1}{2\pi i}\oint_{\gamma_A}\Big(\prod_{i=1}^{k+1}\frac{1}{\eta-\lambda_{j_i}}\Big)d\eta P_{j_1}XP_{j_2}X\cdots P_{j_k}XP_{j_{k+1}}.\label{eq:lemSMk_eq1}
\end{align}
We consider components of summations in (\ref{eq:lemSMk_eq1}). For instance, consider the cases that some $k_1$ indices from $\{j_1,\cdots,j_{k+1}\}$ are not larger than $r$.
W.L.O.G., let $j_1,\cdots,j_{k_1}\leq r$ and $j_{k_1+1},\cdots,j_{k+1}>r$. By Cauchy integral formula, the integral in (\ref{eq:lemSMk_eq1}) is zero if $k_1=0$ or $k_1=k+1$. Therefore, we only focus on the cases that $1\leq k_1\leq k$.
 Then,  
\begin{align*}
\sum_{j_1,\cdots,j_{k_1}\geq 1}^r&\sum_{j_{k_1+1},\cdots,j_{k+1}>r}^d\frac{1}{2\pi i}\oint_{\gamma_A}\Big(\prod_{i=1}^{k_1}\frac{1}{\eta-\lambda_{j_i}}\Big)\eta^{k_1-k-1}d\eta P_{j_1}XP_{j_2}X\cdots P_{j_k}XP_{j_{k+1}}\\
=&\sum_{j_1,\cdots,j_{k_1}\geq 1}^r\frac{1}{2\pi i}\oint_{\gamma_A}\Big(\prod_{i=1}^{k_1}\frac{1}{\eta-\lambda_{j_i}}\Big)\eta^{k_1-k-1}d\eta P_{j_1}XP_{j_2}X\cdots P_{j_{k_1}}X\mathfrak{P}^{\perp}X\cdots X\mathfrak{P}^{\perp}.
\end{align*}
Recall that our goal is to prove
$$
\calS_{A,k}(X)=\sum_{\bs:s_1+\cdots+s_{k+1}=k}(-1)^{1+\tau(\bs)}\cdot \mathfrak{P}^{-s_1}X\mathfrak{P}^{-s_2}X\cdots X\mathfrak{P}^{-s_{k+1}}.
$$
Accordingly, in the above summations, we consider the components, where $s_1,\cdots,s_{k_1}\geq 1$ and $s_{k_1+1}=\cdots=s_{k+1}=0$, namely, 
$$
\sum_{\substack{s_1+\cdots+s_{k_1}=k\\ s_j\geq 1}}(-1)^{k_1+1}\mathfrak{P}^{-s_1}X\cdots X\mathfrak{P}^{-s_{k_1}}X\mathfrak{P}^{\perp}\cdots X\mathfrak{P}^{\perp}.
$$
It turns out that we need to prove
\begin{align*}
\sum_{j_1,\cdots,j_{k_1}\geq 1}^r\frac{1}{2\pi i}\oint_{\gamma_A}&\Big(\prod_{i=1}^{k_1}\frac{1}{\eta-\lambda_{j_i}}\Big)\eta^{k_1-k-1}d\eta P_{j_1}XP_{j_2}X\cdots P_{j_{k_1}}\nonumber\\
=&\sum_{j_1,\cdots,j_{k_1}\geq 1}^r\sum_{\substack{s_1+\cdots+s_{k_1}=k\\ s_j\geq 1}}(-1)^{k_1+1}\frac{1}{\lambda_{j_1}^{s_1}\cdots \lambda_{j_{k_1}}^{s_{k_1}}} P_{j_1}XP_{j_2}X\cdots XP_{j_{k_1}}.
\end{align*}
It suffices to prove that for all $\bj=(j_1,\dots,j_{k_1})\in\{1,\cdots,r\}^{k_1}$, 
\begin{equation}
\frac{1}{2\pi i}\oint_{\gamma_A}\frac{d\eta}{(\eta-\lambda_{j_1})\cdots(\eta-\lambda_{j_{k_1}})\eta^{k+1-k_1}}=\sum_{\substack{s_1+\cdots+s_{k_1}=k\\ s_j\geq 1}}(-1)^{k_1+1}\frac{1}{\lambda_{j_1}^{s_1}\cdots \lambda_{j_{k_1}}^{s_{k_1}}}.\label{eq:lemSMk_eq2}
\end{equation}
To prove (\ref{eq:lemSMk_eq2}), we rewrite its right hand side. Given any $\bj=(j_1,\cdots,j_{k_1})\in\{1,\cdots,r\}^{k_1}$, define
$$
\bv_i(\bj):=\big\{1\leq t\leq k_1: j_t=i\big\} \quad {\rm for}\ 1\leq i\leq r
$$
, that is, $\bv_i(\bj)$ contains the location $s$ such that $\lambda_{j_s}=\lambda_i$.  Meanwhile, denote $v_i(\bj)={\rm Card}\big(\bv_i(\bj)\big)$. Then, the right hand side of (\ref{eq:lemSMk_eq2}) is written as
\begin{align*}
\sum_{\substack{s_1+\cdots+s_{k_1}=k\\ s_j\geq 1}}(-1)^{k_1+1}\frac{1}{\lambda_{j_1}^{s_1}\cdots \lambda_{j_{k_1}}^{s_{k_1}}}=(-1)^{k_1+1}\sum_{\substack{s_1+\cdots+s_{k_1}=k\\s_j\geq 1}}\lambda_1^{-\sum_{p\in\bv_1(\bj)}s_p}\cdots \lambda_r^{-\sum_{p\in\bv_r(\bj)}s_p}.
\end{align*}
Now, we denote $t_i(\bj)=\sum_{p\in\bv_i(\bj)}s_p$ for $1\leq i\leq r$, we can write the above equation as 
\begin{align*}
\sum_{\substack{s_1+\cdots+s_{k_1}=k\\ s_j\geq 1}}(-1)^{k_1+1}&\frac{1}{\lambda_{j_1}^{s_1}\cdots \lambda_{j_{k_1}}^{s_{k_1}}}=(-1)^{k_1+1}\sum_{\substack{t_1(\bj)+\cdots+t_r(\bj)=k\\ t_i(\bj)\geq v_i(\bj)\\t_i(\bj)=0\ {\rm if}\ v_i(\bj)=0} }\prod_{i: v_i(\bj)\geq 1}{t_i(\bj)-1 \choose v_i(\bj)-1}\lambda_i^{-t_i(\bj)}\\
=&(-1)^{k_1+1}\sum_{\substack{t_1(\bj)+\cdots+t_r(\bj)=k-k_1\\ t_i(\bj)=0\ {\rm if}\ v_i(\bj)=0} }\prod_{i: v_i(\bj)\geq 1}{t_i(\bj)+v_i(\bj)-1 \choose v_i(\bj)-1}\lambda_i^{-t_i(\bj)-v_i(\bj)}
\end{align*}
where the last equality is due to the fact $v_1(\bj)+\cdots+v_r(\bj)=k_1$. Similarly, the left hand side of (\ref{eq:lemSMk_eq2}) can be written as 
$$
\frac{1}{2\pi i}\oint_{\gamma_A}\frac{d\eta}{(\eta-\lambda_{j_1})\cdots(\eta-\lambda_{j_{k_1}})\eta^{k+1-k_1}}=\frac{1}{2\pi i}\oint_{\gamma_A}\frac{d\eta}{(\eta-\lambda_{1})^{v_1(\bj)}\cdots(\eta-\lambda_{j_r})^{v_r(\bj)}\eta^{k+1-k_1}}.
$$
Therefore, in order to prove (\ref{eq:lemSMk_eq2}), it suffices to prove that for any $\bj=(j_1,\cdots,j_{k_1})$ the following equality holds
\begin{equation}
\frac{1}{2\pi i}\oint_{\gamma_A}\frac{d\eta}{(\eta-\lambda_{1})^{v_1}\cdots(\eta-\lambda_{j_r})^{v_r}\eta^{k+1-k_1}}=(-1)^{k_1+1}\sum_{\substack{t_1+\cdots+t_r=k-k_1\\t_i=0\ {\rm if}\ v_i=0} }\prod_{i: v_i\geq 1}^r{t_i+v_i-1 \choose v_i-1}\lambda_i^{-t_i-v_i}\label{eq:lemSMk_obj}
\end{equation}
where we omitted the index $\bj$ in definitions of $v_i(\bj)$ and $t_i(\bj)$ without causing any confusions. The non-negative numbers $v_1+\cdots+v_{r}=k_1$. We define the function 
$$
\varphi(\eta)=\frac{1}{(\eta-\lambda_1)^{v_1}\cdots(\eta-\lambda_r)^{v_r}\eta^{k+1-k_1}}
$$
and we will calculate $\frac{1}{2\pi i}\oint_{\gamma_A}\varphi(\eta)d\eta$ by Residue theorem. Indeed, by Residue theorem,
$$
\frac{1}{2\pi i}\oint_{\gamma_A}\varphi(\eta)d\eta=-{\rm Res}(\varphi,\eta=\infty)-{\rm Res}(\varphi,\eta=0).
$$
Clearly, ${\rm Res}(\varphi, \eta=\infty)=0$ and it suffices to calculate ${\rm Res}(\varphi,\eta=0)$. To this end, let $\gamma_0$ be a contour plot around $0$ where none of $\{\lambda_k\}_{k=1}^r$ is inside it. Then,
\begin{align*}
{\rm Res}(\varphi,\eta=0)=\frac{1}{2\pi i}\oint_{\gamma_0}\varphi(\eta)d\eta.
\end{align*}
By Cauchy integral formula, we obtain
\begin{align*}
{\rm Res}(\varphi,\eta=0)=\frac{1}{(k-k_1)!}\Big[\prod_{i: v_i\geq 1}^r(\eta-\lambda_{i})^{-v_i}\Big]^{(k-k_1)}\Big|_{\eta=0}
\end{align*}
where we denote by $f(x)^{(k-k_1)}$ the $k-k_1$-th order differentiation of $f(x)$. Then,  we use general Leibniz rule and get
\begin{align*}
{\rm Res}(\varphi,\eta=0)=&\frac{1}{(k-k_1)!}\sum_{\substack{t_1+\cdots+t_r=k-k_1\\ t_i=0\ {\rm if}\ v_i=0}} \frac{(k-k_1)!}{t_1!t_2!\cdots t_r!}\prod_{i: v_i\geq 1}^r\Big[(\eta-\lambda_{i})^{-v_i}\Big]^{(t_i)}\Big|_{\eta=0}\\
=&(-1)^{k-k_1}\sum_{\substack{t_1+\cdots+t_r=k-k_1\\ t_i=0\ {\rm if}\ v_i=0}} \prod_{i:v_i\geq 1}^r\frac{v_i(v_i+1)\cdots(v_i+t_i-1)}{t_i!}(-\lambda_i)^{-v_i-t_i}\\
=&(-1)^{k-k_1}\sum_{\substack{t_1+\cdots+t_r=k-k_1\\ t_i=0\ {\rm if}\ v_i=0}} \prod_{i:v_i\geq 1}^r{t_i+v_i-1\choose v_i-1}(-\lambda_i)^{-v_i-t_i}\\
=&(-1)^{2k-k_1}\sum_{\substack{t_1+\cdots+t_r=k-k_1\\ t_i=0\ {\rm if}\ v_i=0}} \prod_{i:v_i\geq 1}^r{t_i+v_i-1\choose v_i-1}\lambda_i^{-v_i-t_i}.
\end{align*}
Therefore, 
$$
\frac{1}{2\pi i}\oint_{\gamma_A}\varphi(\eta)d\eta=(-1)^{k_1+1}\sum_{\substack{t_1+\cdots+t_r=k-k_1\\ t_i=0\ {\rm if}\ v_i=0}} \prod_{i:v_i\geq 1}^r{t_i+v_i-1\choose v_i-1}\lambda_i^{-v_i-t_i}
$$
which proves (\ref{eq:lemSMk_obj}). We conclude the proof of Theorem~\ref{thm:hat_Theta}.


\bibliographystyle{apa}
\bibliography{refer}

\begin{thebibliography}{}

\bibitem[\protect\astroncite{Abbe et~al.}{2017}]{abbe2017entrywise}
Abbe, E., Fan, J., Wang, K., and Zhong, Y. (2017).
\newblock Entrywise eigenvector analysis of random matrices with low expected
  rank.
\newblock {\em arXiv preprint arXiv:1709.09565}.

\bibitem[\protect\astroncite{Bao et~al.}{2018}]{bao2018singular}
Bao, Z., Ding, X., and Wang, K. (2018).
\newblock Singular vector and singular subspace distribution for the matrix
  denoising model.
\newblock {\em arXiv preprint arXiv:1809.10476}.

\bibitem[\protect\astroncite{Benaych-Georges and
  Nadakuditi}{2012}]{benaych2012singular}
Benaych-Georges, F. and Nadakuditi, R.~R. (2012).
\newblock The singular values and vectors of low rank perturbations of large
  rectangular random matrices.
\newblock {\em Journal of Multivariate Analysis}, 111:120--135.

\bibitem[\protect\astroncite{Berry}{1941}]{berry1941accuracy}
Berry, A.~C. (1941).
\newblock The accuracy of the gaussian approximation to the sum of independent
  variates.
\newblock {\em Transactions of the american mathematical society},
  49(1):122--136.

\bibitem[\protect\astroncite{Bloemendal et~al.}{2016}]{bloemendal2016principal}
Bloemendal, A., Knowles, A., Yau, H.-T., and Yin, J. (2016).
\newblock On the principal components of sample covariance matrices.
\newblock {\em Probability theory and related fields}, 164(1-2):459--552.

\bibitem[\protect\astroncite{Cai et~al.}{2010}]{cai2010singular}
Cai, J.-F., Cand{\`e}s, E.~J., and Shen, Z. (2010).
\newblock A singular value thresholding algorithm for matrix completion.
\newblock {\em SIAM Journal on Optimization}, 20(4):1956--1982.

\bibitem[\protect\astroncite{Cai and Zhang}{2018}]{cai2018rate}
Cai, T.~T. and Zhang, A. (2018).
\newblock Rate-optimal perturbation bounds for singular subspaces with
  applications to high-dimensional statistics.
\newblock {\em The Annals of Statistics}, 46(1):60--89.

\bibitem[\protect\astroncite{Candes et~al.}{2015}]{candes2015phase}
Candes, E.~J., Li, X., and Soltanolkotabi, M. (2015).
\newblock Phase retrieval via wirtinger flow: Theory and algorithms.
\newblock {\em IEEE Transactions on Information Theory}, 61(4):1985--2007.

\bibitem[\protect\astroncite{Cand{\`e}s and Tao}{2010}]{candes2010power}
Cand{\`e}s, E.~J. and Tao, T. (2010).
\newblock The power of convex relaxation: Near-optimal matrix completion.
\newblock {\em IEEE Transactions on Information Theory}, 56(5):2053--2080.

\bibitem[\protect\astroncite{Chen et~al.}{2018}]{chen2018asymmetry}
Chen, Y., Cheng, C., and Fan, J. (2018).
\newblock Asymmetry helps: Eigenvalue and eigenvector analyses of
  asymmetrically perturbed low-rank matrices.
\newblock {\em arXiv preprint arXiv:1811.12804}.

\bibitem[\protect\astroncite{Davis and Kahan}{1970}]{davis1970rotation}
Davis, C. and Kahan, W.~M. (1970).
\newblock The rotation of eigenvectors by a perturbation. iii.
\newblock {\em SIAM Journal on Numerical Analysis}, 7(1):1--46.

\bibitem[\protect\astroncite{Ding}{2017}]{ding2017high}
Ding, X. (2017).
\newblock High dimensional deformed rectangular matrices with applications in
  matrix denoising.
\newblock {\em arXiv:1702.06975}.

\bibitem[\protect\astroncite{Donoho and Gavish}{2014}]{donoho2014minimax}
Donoho, D. and Gavish, M. (2014).
\newblock Minimax risk of matrix denoising by singular value thresholding.
\newblock {\em The Annals of Statistics}, 42(6):2413--2440.

\bibitem[\protect\astroncite{Esseen}{1942}]{esseen1942liapunoff}
Esseen, C.-G. (1942).
\newblock On the liapunoff limit of error in the theory of probability.
\newblock {\em Arkiv for matematik, astronomi och fysik}, A28:1--19.

\bibitem[\protect\astroncite{Golub and Van~Loan}{2012}]{golub2012matrix}
Golub, G.~H. and Van~Loan, C.~F. (2012).
\newblock {\em Matrix computations}, volume~3.
\newblock JHU Press.

\bibitem[\protect\astroncite{G{\"o}tze and Tikhomirov}{2011}]{gotze2011rate}
G{\"o}tze, F. and Tikhomirov, A. (2011).
\newblock On the rate of convergence to the marchenko--pastur distribution.
\newblock {\em arXiv preprint arXiv:1110.1284}.

\bibitem[\protect\astroncite{Keshavan et~al.}{2010}]{keshavan2010matrix}
Keshavan, R.~H., Montanari, A., and Oh, S. (2010).
\newblock Matrix completion from a few entries.
\newblock {\em IEEE transactions on information theory}, 56(6):2980--2998.

\bibitem[\protect\astroncite{Kim et~al.}{2005}]{kim2005dimension}
Kim, H., Howland, P., and Park, H. (2005).
\newblock Dimension reduction in text classification with support vector
  machines.
\newblock {\em Journal of Machine Learning Research}, 6(Jan):37--53.

\bibitem[\protect\astroncite{Koltchinskii and
  Lounici}{2016}]{koltchinskii2016asymptotics}
Koltchinskii, V. and Lounici, K. (2016).
\newblock Asymptotics and concentration bounds for bilinear forms of spectral
  projectors of sample covariance.
\newblock In {\em Annales de l'Institut Henri Poincar{\'e}, Probabilit{\'e}s et
  Statistiques}, volume~52, pages 1976--2013. Institut Henri Poincar{\'e}.

\bibitem[\protect\astroncite{Koltchinskii and
  Lounici}{2017}]{koltchinskii2017normal}
Koltchinskii, V. and Lounici, K. (2017).
\newblock Normal approximation and concentration of spectral projectors of
  sample covariance.
\newblock {\em The Annals of Statistics}, 45(1):121--157.

\bibitem[\protect\astroncite{Koltchinskii and
  Xia}{2016}]{koltchinskii2016perturbation}
Koltchinskii, V. and Xia, D. (2016).
\newblock Perturbation of linear forms of singular vectors under gaussian
  noise.
\newblock In {\em High Dimensional Probability VII}, pages 397--423. Springer.

\bibitem[\protect\astroncite{Li and Wang}{2007}]{li2007directional}
Li, B. and Wang, S. (2007).
\newblock On directional regression for dimension reduction.
\newblock {\em Journal of the American Statistical Association},
  102(479):997--1008.

\bibitem[\protect\astroncite{Li and Li}{2018}]{li2018two}
Li, Y. and Li, H. (2018).
\newblock Two-sample test of community memberships of weighted stochastic block
  models.
\newblock {\em arXiv preprint arXiv:1811.12593}.

\bibitem[\protect\astroncite{Ma et~al.}{2017}]{ma2017implicit}
Ma, C., Wang, K., Chi, Y., and Chen, Y. (2017).
\newblock Implicit regularization in nonconvex statistical estimation: Gradient
  descent converges linearly for phase retrieval, matrix completion and blind
  deconvolution.
\newblock {\em arXiv preprint arXiv:1711.10467}.

\bibitem[\protect\astroncite{Mingo and Speicher}{2017}]{mingo2017free}
Mingo, J.~A. and Speicher, R. (2017).
\newblock {\em Free probability and random matrices}, volume~35.
\newblock Springer.

\bibitem[\protect\astroncite{Naumov et~al.}{2017}]{naumov2017bootstrap}
Naumov, A., Spokoiny, V., and Ulyanov, V. (2017).
\newblock Bootstrap confidence sets for spectral projectors of sample
  covariance.
\newblock {\em Probability Theory and Related Fields}, pages 1--42.

\bibitem[\protect\astroncite{Shabalin and
  Nobel}{2013}]{shabalin2013reconstruction}
Shabalin, A.~A. and Nobel, A.~B. (2013).
\newblock Reconstruction of a low-rank matrix in the presence of gaussian
  noise.
\newblock {\em Journal of Multivariate Analysis}, 118:67--76.

\bibitem[\protect\astroncite{Tang et~al.}{2018}]{tang2018limit}
Tang, M., Priebe, C.~E., et~al. (2018).
\newblock Limit theorems for eigenvectors of the normalized laplacian for
  random graphs.
\newblock {\em The Annals of Statistics}, 46(5):2360--2415.

\bibitem[\protect\astroncite{Tao}{2012}]{tao2012topics}
Tao, T. (2012).
\newblock {\em Topics in random matrix theory}, volume 132.
\newblock American Mathematical Soc.

\bibitem[\protect\astroncite{Vershynin}{2010}]{vershynin2010introduction}
Vershynin, R. (2010).
\newblock Introduction to the non-asymptotic analysis of random matrices.
\newblock {\em arXiv preprint arXiv:1011.3027}.

\bibitem[\protect\astroncite{Wedin}{1972}]{wedin1972perturbation}
Wedin, P.-{\AA}. (1972).
\newblock Perturbation bounds in connection with singular value decomposition.
\newblock {\em BIT Numerical Mathematics}, 12(1):99--111.

\bibitem[\protect\astroncite{Xia}{2019}]{xia2018confidence}
Xia, D. (2019).
\newblock Confidence interval of singular subspaces for high-dimensional and
  low-rank matrix regression.
\newblock {\em IEEE Transactions on Information Theory}.

\bibitem[\protect\astroncite{Xia and Yuan}{2018}]{xia2017polynomial}
Xia, D. and Yuan, M. (2018+).
\newblock On polynomial time methods for exact low rank tensor completion.
\newblock {\em Foundations of Computational Mathematics}.

\bibitem[\protect\astroncite{Xia et~al.}{2017}]{xia2017statistically}
Xia, D., Yuan, M., and Zhang, C.-H. (2017).
\newblock Statistically optimal and computationally efficient low rank tensor
  completion from noisy entries.
\newblock {\em arXiv preprint arXiv:1711.04934}.

\bibitem[\protect\astroncite{Xia and Zhou}{2019}]{xia2019sup}
Xia, D. and Zhou, F. (2019).
\newblock The sup-norm perturbation of hosvd and low rank tensor denoising.
\newblock {\em Journal of Machine Learning Research}, 20(61):1--42.

\bibitem[\protect\astroncite{Yu et~al.}{2014}]{yu2014useful}
Yu, Y., Wang, T., and Samworth, R.~J. (2014).
\newblock A useful variant of the davis--kahan theorem for statisticians.
\newblock {\em Biometrika}, 102(2):315--323.

\bibitem[\protect\astroncite{Zhang and Xia}{2018}]{zhang2018tensor}
Zhang, A. and Xia, D. (2018).
\newblock Tensor svd: Statistical and computational limits.
\newblock {\em IEEE Transactions on Information Theory}.

\end{thebibliography}

\newpage

\setcounter{page}{1}
\setcounter{section}{0}
\appendix

\begin{center}
	{\LARGE Supplement to ``Normal Approximation and Confidence Region of Singular Subspaces"\footnote{Dong Xia is an Assistant Professor in Department of Mathematics at Hong Kong University of Science and Technology, Kowloon, Hong Kong.  E-mail: madxia@ust.hk.}}
	
	\bigskip\medskip
	
	 Dong Xia
	\ \par Hong Kong University of Science and Technology
\end{center}

\begin{abstract}
	In this Supplement, we provide proofs for the main results and technical lemmas. 
\end{abstract}

\section{Proofs}

\subsection{Proof of Theorem~\ref{thm:distE_normal}}
By $\rank(\hat\Theta)=\rank(\Theta)=2r$,  we get
$$
{\rm dist}^2[(\hat U, \hat V), (U,V)]=\|\hat\Theta\hat\Theta^{\tran}-\Theta\Theta^{\tran}\|_{\rm F}^2=4r-2\big<\hat\Theta\hat\Theta^{\tran}, \Theta\Theta^{\tran}\big>.
$$
 Since $X$ is random, we shall take care of the ``size" of $X$. 
Observe that $\|X\|=\|Z\|$ and the operator norm of $Z$ is well-known (see, e.g., \citep{tao2012topics} and \citep{vershynin2010introduction} ). Indeed, there exist some absolute constants $C_1,C_2,c_1>0$ such that 
\begin{equation}\label{eq:Z_norm}
\EE \|X\|\leq C_1\sqrt{d_{\max}}\quad {\rm and}\quad \PP\big(\|X\|\geq C_2\sqrt{d_{\max}}\big)\leq e^{-c_1d_{\max}} 
\end{equation}
where $d_{\max}=\max\{d_1,d_2\}$. Meanwhile, $\EE^{1/p}\|X\|^p\leq C_3\sqrt{d_{\max}}$ for all integer $p\geq 1$. See \cite[Lemma~3]{koltchinskii2016perturbation}.

Denote the event $\calE_1:=\{\|X\|\leq C_2\sqrt{{d}_{\max}}\}$ so that $\PP(\calE_1)\geq 1-e^{-c_1d_{\max}}$. Assume that $\lambda_r>2C_2\sqrt{d_{\max}}$, our analysis is conditioned on $\calE_1$. By Theorem~\ref{thm:hat_Theta}, on event $\calE_1$, we have 
$$
\hat\Theta\hat\Theta^{\tran}=\Theta\Theta^{\tran}+\calS_{A,1}(X)+\calS_{A,2}(X)+\sum_{k\geq 3}\calS_{A,3}(X)
$$
where $\calS_{A,1}(X)=\mathfrak{P}^{-1}X\mathfrak{P}^{\perp}+\mathfrak{P}^{\perp}X\mathfrak{P}^{-1}$ and 
\begin{align*}
\calS_{A,2}(X)=&\big(\mathfrak{P}^{-2}X\mathfrak{P}^{\perp}X\mathfrak{P}^{\perp}+\mathfrak{P}^{\perp}X\mathfrak{P}^{-2}X\mathfrak{P}^{\perp}+\mathfrak{P}^{\perp}X\mathfrak{P}^{\perp}X\mathfrak{P}^{-2}\big)\\
-&\big(\mathfrak{P}^{\perp}X\mathfrak{P}^{-1}X\mathfrak{P}^{-1}+\mathfrak{P}^{-1}X\mathfrak{P}^{\perp}X\mathfrak{P}^{-1}+\mathfrak{P}^{-1}X\mathfrak{P}^{-1}X\mathfrak{P}^{\perp}\big).
\end{align*}
Therefore,  we get
\begin{align*}
\|\hat\Theta\hat\Theta^{\tran}-\Theta\Theta^{\tran}\|_{\rm F}^2=&2\tr\big(\mathfrak{P}^{-1}X\mathfrak{P}^{\perp}X\mathfrak{P}^{-1}\big)-2\sum_{k\geq 3}\big<\Theta\Theta^{\tran}, \calS_{A,k}(X)\big>\\
=&2\|\mathfrak{P}^{-1}X\mathfrak{P}^{\perp}\|_{\rm F}^2-2\sum_{k\geq 3}\big<\Theta\Theta^{\tran}, \calS_{A,k}(X)\big>.
\end{align*}
Then,
\begin{align*}
{\rm dist}^2[(\hat U,&\hat V), (U,V)]-\EE\ {\rm dist}^2[(\hat U,\hat V), (U,V)]\\
=&\Big(2\|\mathfrak{P}^{-1}X\mathfrak{P}^{\perp}\|_{\rm F}^2-2\EE \|\mathfrak{P}^{-1}X\mathfrak{P}^{\perp}\|_{\rm F}^2\Big)
-2\sum_{k\geq 3}\big<\Theta\Theta^{\tran}, \calS_{A,k}(X)-\EE \calS_{A,k}(X)\big>.
\end{align*}
We investigate the normal approximation of 
$$
\frac{2\|\mathfrak{P}^{-1}X\mathfrak{P}^{\perp}\|_{\rm F}^2-2\EE\|\mathfrak{P}^{-1}X\mathfrak{P}^{\perp}\|_{\rm F}^2}{\sqrt{8(d_1+d_2-2r)}\cdot \|\Lambda^{-2}\|_{\rm F}}
$$
and show that 
$$
 \frac{2\sum_{k\geq 3}\big<\Theta\Theta^{\tran},\calS_{A,k}(X)-\EE\calS_{A,k}(X)\big>}{\sqrt{8(d_1+d_2-2r)}\cdot \|\Lambda^{-2}\|_{\rm F}}
$$
is ignorable when signal strength $\lambda_r$ is sufficiently strong.  For some $t>0$ which shall be determined later, define a function
\begin{equation}\label{eq:ftX_def}
f_{t}(X)=2\sum_{k\geq 3}\big<\Theta\Theta^{\tran}, \calS_{A,k}(X)\big> \cdot \phi\Big(\frac{\|X\|}{t\cdot \sqrt{d_{\max}}}\Big)
\end{equation}
where we view $X$ as a variable in $\RR^{(d_1+d_2)\times (d_1+d_2)}$ and the function $\phi(\cdot): \RR_+\mapsto \RR_+$ is defined by
$$
\phi(s):=\begin{cases}
1& {\rm if}\ s\leq 1,\\
2-s& {\rm if}\ 1<s\leq 2,\\
0&{\rm if}\ s>2.
\end{cases}
$$
Clearly, $\phi(s)$ is Lipschitz with constant $1$. 
 Lemma~\ref{lem:fX_Lip} shows that $f(\cdot)$ is  Lipschitz  when $\lambda_r\geq C_4\sqrt{d_{\max}}$. The proof of Lemma~\ref{lem:fX_Lip} is in Appendix, Section~\ref{sec:supp_lem}. 
\begin{lemma}\label{lem:fX_Lip}
There exist absolute constants $C_3, C_4>0$ so that if $\lambda_r\geq C_3t^2\sqrt{d_{\max}}$, then 
$$
\big|f_t(X_1)-f_t(X_2) \big|\leq C_4t^2\frac{rd_{\max}}{\lambda_r^3}\cdot \|X_1-X_2\|_{\rm F}
$$
where $f_t(X)$ is defined by (\ref{eq:ftX_def}). 
\end{lemma}
By Lemma~\ref{lem:fX_Lip} and Gaussian isoperimetric inequality (see, e.g., \citep{koltchinskii2016asymptotics, koltchinskii2017normal}), it holds with probability at least $1-e^{-s}$ for any $s\geq 1$ that
\begin{align}
\bigg|2\sum_{k\geq 3}\big<\Theta\Theta^{\tran}, \calS_{A,k}(X)\big> \cdot \phi\Big(\frac{\|X\|}{t\cdot \sqrt{d_{\max}}}\Big)&-\EE2\sum_{k\geq 3}\big<\Theta\Theta^{\tran}, \calS_{A,k}(X)\big> \cdot \phi\Big(\frac{\|X\|}{t\cdot \sqrt{d_{\max}}}\Big)\bigg|\nonumber\\
\leq C_5\sqrt{s}t^2\frac{rd_{\max}}{\lambda_r^3}\label{eq:Theta-E}
\end{align}
for some absolute constant $C_5>0$. Now, set $t=C_2$ where $C_2$ is defined in (\ref{eq:Z_norm}). Therefore,  $\phi\Big(\frac{\|X\|}{C_2\cdot\sqrt{d_{\max}}}\Big)=1$ on event $\calE_1$. Meanwhile, the following fact holds
\begin{align*}
\bigg|\EE2\sum_{k\geq 3}\big<\Theta\Theta^{\tran}, \calS_{A,k}(X)\big> \cdot &\phi\Big(\frac{\|X\|}{C_2\cdot \sqrt{d_{\max}}}\Big)-\EE2\sum_{k\geq 3}\big<\Theta\Theta^{\tran}, \calS_{A,k}(X)\big>\bigg|\\
\leq&\bigg|\EE2\sum_{k\geq 3}\big<\Theta\Theta^{\tran}, \calS_{A,k}(X)\big> \cdot \phi\Big(\frac{\|X\|}{C_2\cdot \sqrt{d_{\max}}}\Big)\mathbb{I}_{\calE_1^{\rm c}}\bigg|\\
&\quad +\bigg|\EE2\sum_{k\geq 3}\big<\Theta\Theta^{\tran}, \calS_{A,k}(X)\big>\mathbb{I}_{\calE_1^{\rm c}}\bigg|\\
\leq&4\sum_{k\geq 3}\EE \big|\big<\Theta\Theta^{\tran}, \calS_{A,k}(X)\big>\big|\mathbb{I}_{\calE_1^{\rm c}}\leq 8r\sum_{k\geq 3}\EE^{1/2}\|\calS_{A,k}(X)\|^2\cdot e^{-c_1d_{\max}/2}\\
\leq&e^{-c_1d_{\max}/2}\cdot 8r\sum_{k\geq 3} \EE^{1/2}\frac{16^k\|X\|^{2k}}{\lambda_r^{2k}}\leq e^{-c_1d_{\max}/2}\cdot 8r\sum_{r\geq 3}\Big(\frac{C_6\cdot d_{\max}^{1/2}}{\lambda_r}\Big)^k\\
\leq &e^{-c_1d_{\max}/2}\cdot \frac{C_6 rd_{\max}^{3/2}}{\lambda_r^3}\leq C_6\frac{rd_{\max}}{\lambda_r^3}
\end{align*}
where the last inequality holds as long as $e^{-c_1d_{\max}/2}\leq \frac{1}{\sqrt{d_{\max}}}$ and we used the fact $\EE^{1/p}\|X\|^p\leq C_6 \sqrt{d_{\max}}$ for some absolute constant $C_6>0$ and any positive integer $p$. (See, e.g., \citep{koltchinskii2016perturbation}, \citep{vershynin2010introduction} and \citep{tao2012topics}). Together with (\ref{eq:Theta-E}), it holds  with probability at least $1-e^{-s}-e^{-c_1d_{\max}}$ for any $s\geq 1$ that
$$
\Big|2\sum_{k\geq 3}\big<\Theta\Theta^{\tran},\calS_{A,k}(X)\big>-\EE2\sum_{k\geq 3}\big<\Theta\Theta^{\tran},\calS_{A,k}(X)\big> \Big|\leq C_6s^{1/2}\cdot \frac{rd_{\max}}{\lambda_r^3}
$$
for some absolute constant $C_6>0$. 
Therefore,  for any $s\geq 1$, with probability at least $1-e^{-s}-e^{-c_1d_{\max}}$,
\begin{equation}\label{eq:Ethm_rem}
\frac{\Big|2\sum_{k\geq 3}\big<\Theta\Theta^{\tran},\calS_{A,k}(X)\big>-\EE2\sum_{k\geq 3}\big<\Theta\Theta^{\tran},\calS_{A,k}(X)\big> \Big|}{\sqrt{8(d_1+d_2-2r)}\|\Lambda^{-2}\|_{\rm F}}\leq C_6s^{1/2}\Big(\frac{\sqrt{r}}{\|\Lambda^{-2}\|_{\rm F}\lambda_r^2}\Big)\cdot \frac{\sqrt{rd_{\max}}}{\lambda_r}
\end{equation}
where we assumed $d_{\max}\geq 3r$. 

We next prove the normal approximation of $2\|\mathfrak{P}^{-1}X\mathfrak{P}^{\perp}\|_{\rm F}^2$. Similar as in \citep{xia2018confidence}, by the definition of $\mathfrak{P}^{-1}, X$ and $\mathfrak{P}^{\perp}$, we could write
\begin{align*}
\mathfrak{P}^{-1}X\mathfrak{P}^{\perp}=\left(\begin{array}{cc}U\Lambda^{-1}V^{\tran}Z^{\tran}U_{\perp}U_{\perp}^{\tran}&0\\
0&V\Lambda^{-1}U^{\tran}ZV_{\perp}V_{\perp}^{\tran}\end{array}\right).
\end{align*}
Then,
\begin{align*}
\|\mathfrak{P}^{-1}X\mathfrak{P}^{\perp}\|_{\rm F}^2=&\|U\Lambda^{-1}V^{\tran}Z^{\tran}U_{\perp}U_{\perp}^{\tran}\|_{\rm F}^2
+\|V\Lambda^{-1}U^{\tran}ZV_{\perp}V_{\perp}^{\tran}\|_{\rm F}^2.
\end{align*}
Denote $z_j\in\RR^{d_1}$ the $j$-th column of $Z$ for $1\leq j\leq d_2$. Then, $z_1,\cdots,z_{d_2}$ are independent Gaussian random vector and $\EE z_jz_j^{\tran}=I_{d_1}$ for all $j$. Therefore, 
$$
U^{\tran}Z=\sum_{j=1}^{d_2}(U^{\tran}z_j)e_{j}^{\tran}
$$
where $\{e_j\}_{j=1}^{d_2}$ represent the standard basis vectors in $\RR^{d_2}$. Similarly, 
$$
U_{\perp}^{\tran}Z=\sum_{j=1}^{d_2}(U_{\perp}^{\tran}z_j)e_j^{\tran}.
$$
Sincet $U^{\tran}z_j$ and $U_{\perp}^{\tran}z_j$ are Gaussian random vectors and 
$$
\EE U^{\tran}z_j\big(U_{\perp}^{\tran}z_j\big)^{\tran}=U^{\tran}U_{\perp}=0
$$
, we know that $\{U^{\tran}z_j\}_{j=1}^{d_2}$ are independent with $\{U_{\perp}^{\tran}z_j\}_{j=1}^{d_2}$ . Therefore, $\|U\Lambda^{-1}V^{\tran}Z^{\tran}U_{\perp}U_{\perp}^{\tran}\|_{\rm F}^2$ is independent with $\|V\Lambda^{-1}U^{\tran}ZV_{\perp}V_{\perp}^{\tran}\|_{\rm F}^2$. Denote by $\tilde Z$ an independent copy of $Z$, we conclude that ($Y_1\stackrel{{\rm d}}{=}Y_2$ denotes equivalence of $Y_1$ and $Y_2$ in distribution)
\begin{align*}
\|\mathfrak{P}^{-1}X\mathfrak{P}^{\perp}\|_{\rm F}^2\stackrel{{\rm d}}{=}&\|U\Lambda^{-1}V^{\tran}Z^{\tran}U_{\perp}U_{\perp}^{\tran}\|_{\rm F}^2
+\|V\Lambda^{-1}U^{\tran}\tilde{Z}V_{\perp}V_{\perp}^{\tran}\|_{\rm F}^2\\
=&\sum_{j=r+1}^{d_1}\|U\Lambda^{-1}V^{\tran}Z^{\tran}u_j\|_{\ell_2}^2+\sum_{j=r+1}^{d_2}\|V\Lambda^{-1}U^{\tran}\tilde Zv_j\|_{\ell_2}^2\\
=&\sum_{j=r+1}^{d_1}\|\Lambda^{-1}V^{\tran}Z^{\tran}u_j\|_{\ell_2}^2+\sum_{j=r+1}^{d_2}\|\Lambda^{-1}U^{\tran}\tilde Zv_j\|_{\ell_2}^2
\end{align*}
where $\{u_j\}_{j=r+1}^{d_1}$ and $\{v_j\}_{j=r+1}^{d_2}$ denote the columns of $U_{\perp}$ and $V_{\perp}$, respectively. Observe that $Z^{\tran}u_j\sim\calN(0,I_{d_2})$ for all $r+1\leq j\leq d_1$ and 
$$
\EE (Z^{\tran}u_{j_1})(Z^{\tran}u_{j_2})^{\tran}=0\quad \textrm{ for all } r+1\leq j_1\neq j_2\leq d_1.
$$
Therefore, $\{Z^{\tran}u_j\}_{j=r+1}^{d_1}$ are independent normal random vectors. Similarly, $\tilde Z v_j\sim\calN(0,I_{d_1})$ are independent for all $r+1\leq j\leq d_2$. Clearly,  $V^{\tran}Z^{\tran}u_{j_1}\sim\calN(0,I_r)$ and $U^{\tran}\tilde Z v_{j_2}\sim\calN(0,I_r)$ are all independent for $r+1\leq j_1\leq d_1$ and $r+1\leq j_2\leq d_2$. 

As a result, let $d_{\star}=d_1+d_2-2r$, we conclude that 
\begin{equation}\label{eq:P-1XPperp}
\|\mathfrak{P}^{-1}X\mathfrak{P}^{\perp}\|_{\rm F}^2\stackrel{\rm d}{=}\sum_{j=1}^{d_{\star}}\|\Lambda^{-1}z_j\|_{\ell_2}^2
\end{equation}
where we abuse the notations and denote $\{z_j\}_{j=1}^{d_{\star}}$ where $z_j\stackrel{i.i.d.}{\sim}\calN(0,I_r)$. By Berry-Esseen theorem (\citep{berry1941accuracy} and \citep{esseen1942liapunoff}), it holds for some absolute constant $C_7>0$ that
\begin{equation}
\sup_{x\in\RR}\bigg|\PP\bigg(\frac{2\|\mathfrak{P}^{-1}X\mathfrak{P}^{\perp}\|_{\rm F}^2-2\EE \|\mathfrak{P}^{-1}X\mathfrak{P}^{\perp}\|_{\rm F}^2}{\sqrt{8(d_1+d_2-2r)}\|\Lambda^{-2}\|_{\rm F}}\leq x\bigg) -\Phi(x) \bigg|\leq C_7\Big(\frac{\|\Lambda^{-1}\|_{\rm F}^4}{\|\Lambda^{-2}\|_{\rm F}^2}\Big)^{3/2}\cdot \frac{1}{\sqrt{d_{\star}}}\label{eq:normal_part}
\end{equation}
where we used the fact ${\rm Var}\big(\|\Lambda^{-1}z_j\|_{\ell_2}^2\big)=2\|\Lambda^{-2}\|_{\rm F}^2$ and 
$$
\EE\big\|\Lambda^{-1}z_j\big\|_{\ell_2}^6\leq C_7\sum_{j_1,j_2,j_3\geq 1}^r \frac{1}{\lambda_{j_1}^2\lambda_{j_2}^2\lambda_{j_3}^2}\leq C_7\|\Lambda^{-1}\|_{\rm F}^6. 
$$
In (\ref{eq:normal_part}), the function $\Phi(x)$ denotes the c.d.f. of  standard normal distributions. Recall that, on event $\calE_1$, 
\begin{align*}
\frac{{\rm dist}^2[(\hat U,\hat V), (U,V)]-\EE\ {\rm dist}^2[(\hat U,\hat V), (U,V)]}{\sqrt{8(d_1+d_2-2r)}\|\Lambda^{-2}\|_{\rm F}}=&\frac{2\|\mathfrak{P}^{-1}X\mathfrak{P}^{\perp}\|_{\rm F}^2-2\EE \|\mathfrak{P}^{-1}X\mathfrak{P}^{\perp}\|_{\rm F}^2}{\sqrt{8(d_1+d_2-2r)}\|\Lambda^{-2}\|_{\rm F}}\\
&\quad +\frac{2\sum_{k\geq 3}\big<\Theta\Theta^{\tran}, \calS_{A,k}(X)-\EE\calS_{A,k}(X)\big>}{\sqrt{8(d_1+d_2-2r)}\|\Lambda^{-2}\|_{\rm F}}
\end{align*}
where normal approximation of the first term is given in (\ref{eq:normal_part}) and upper bound of the second term is given in (\ref{eq:Ethm_rem}). Based on (\ref{eq:Ethm_rem}), we get for any $x\in\RR$ and any $s\geq 1$, 
\begin{align*}
\PP\bigg(&\frac{{\rm dist}^2[(\hat U,\hat V), (U,V)]-\EE\ {\rm dist}^2[(\hat U,\hat V), (U,V)]}{\sqrt{8(d_1+d_2-2r)}\|\Lambda^{-2}\|_{\rm F}}\leq x\bigg)\\
&\quad\quad\quad \leq\PP\bigg(\frac{2\|\mathfrak{P}^{-1}X\mathfrak{P}^{\perp}\|_{\rm F}^2-2\EE \|\mathfrak{P}^{-1}X\mathfrak{P}^{\perp}\|_{\rm F}^2}{\sqrt{8(d_1+d_2-2r)}\|\Lambda^{-2}\|_{\rm F}}\leq x+C_6s^{1/2}\cdot \frac{\sqrt{r}}{\|\Lambda^{-2}\|_{\rm F}\lambda_r^2}\cdot \frac{\sqrt{rd_{\max}}}{\lambda_r}\bigg)\\
&\quad\quad\quad\quad\quad + e^{-s}+e^{-c_1d_{\max}}\\
 \leq& \Phi\bigg(x+C_6s^{1/2}\cdot \frac{\sqrt{r}}{\|\Lambda^{-2}\|_{\rm F}\lambda_r^2}\cdot \frac{\sqrt{rd_{\max}}}{\lambda_r}\bigg)+ e^{-s}+e^{-c_1d_{\max}}+C_7\Big(\frac{\|\Lambda^{-1}\|_{\rm F}^4}{\|\Lambda^{-2}\|_{\rm F}^2}\Big)^{3/2}\cdot \frac{1}{\sqrt{d_{\star}}}\\
\leq& \Phi(x)+C_6s^{1/2}\cdot \frac{\sqrt{r}}{\|\Lambda^{-2}\|_{\rm F}\lambda_r^2}\cdot \frac{\sqrt{rd_{\max}}}{\lambda_r}+e^{-s}+e^{-c_1d_{\max}}+C_7\Big(\frac{\|\Lambda^{-1}\|_{\rm F}^4}{\|\Lambda^{-2}\|_{\rm F}^2}\Big)^{3/2}\cdot \frac{1}{\sqrt{d_{\star}}}
\end{align*}
where the last inequality is due to (\ref{eq:normal_part}) and the Lipschitz property of $\Phi(x)$. Similarly, for any $x\in\RR$ and any $s\geq 1$,
\begin{align*}
\PP\bigg(&\frac{{\rm dist}^2[(\hat U,\hat V), (U,V)]-\EE\ {\rm dist}^2[(\hat U,\hat V), (U,V)]}{\sqrt{8(d_1+d_2-2r)}\|\Lambda^{-2}\|_{\rm F}}\leq x\bigg)\\
&\quad\quad\quad \geq\PP\bigg(\frac{2\|\mathfrak{P}^{-1}X\mathfrak{P}^{\perp}\|_{\rm F}^2-2\EE \|\mathfrak{P}^{-1}X\mathfrak{P}^{\perp}\|_{\rm F}^2}{\sqrt{8(d_1+d_2-2r)}\|\Lambda^{-2}\|_{\rm F}}\leq x-C_6s^{1/2}\cdot \frac{\sqrt{r}}{\|\Lambda^{-2}\|_{\rm F}\lambda_r^2}\cdot \frac{\sqrt{rd_{\max}}}{\lambda_r}\bigg)\\
&\quad\quad\quad\quad\quad - e^{-s}-e^{-c_1d_{\max}}\\
 \geq& \Phi\bigg(x-C_6s^{1/2}\cdot \frac{\sqrt{r}}{\|\Lambda^{-2}\|_{\rm F}\lambda_r^2}\cdot \frac{\sqrt{rd_{\max}}}{\lambda_r}\bigg)- e^{-s}-e^{-c_1d_{\max}}-C_7\Big(\frac{\|\Lambda^{-1}\|_{\rm F}^4}{\|\Lambda^{-2}\|_{\rm F}^2}\Big)^{3/2}\cdot \frac{1}{\sqrt{d_{\star}}}\\
\geq& \Phi(x)-C_6s^{1/2}\cdot \frac{\sqrt{r}}{\|\Lambda^{-2}\|_{\rm F}\lambda_r^2}\cdot \frac{\sqrt{rd_{\max}}}{\lambda_r}-e^{-s}-e^{-c_1d_{\max}}-C_7\Big(\frac{\|\Lambda^{-1}\|_{\rm F}^4}{\|\Lambda^{-2}\|_{\rm F}^2}\Big)^{3/2}\cdot \frac{1}{\sqrt{d_{\star}}}.
\end{align*}
Finally, we conclude that for any $s\geq 1$,
\begin{align*}
\sup_{x\in\RR}\bigg|\PP\bigg(&\frac{{\rm dist}^2[(\hat U,\hat V), (U,V)]-\EE\ {\rm dist}^2[(\hat U,\hat V), (U,V)]}{\sqrt{8(d_1+d_2-2r)}\|\Lambda^{-2}\|_{\rm F}}\leq x\bigg)-\Phi(x) \bigg|\\
& \leq C_6s^{1/2}\Big(\frac{\sqrt{r}}{\|\Lambda^{-2}\|_{\rm F}\lambda_r^2}\Big)\cdot \frac{\sqrt{rd_{\max}}}{\lambda_r}+e^{-s}+e^{-c_1d_{\max}}+C_7\Big(\frac{\|\Lambda^{-1}\|_{\rm F}^4}{\|\Lambda^{-2}\|_{\rm F}^2}\Big)^{3/2}\cdot \frac{1}{\sqrt{d_{\star}}}
\end{align*}
where $d_{\star}=d_1+d_2-2r$ and $C_6, C_7,c_1$ are absolute positive constants.

\subsection{Proof of lemmas in Section~\ref{sec:approx_Edist} }
Observe that $\calS_{A,k}(X)$ involves the product of $X$ for $k$ times. If $k$ is odd, we immediately get $\EE\calS_{A,k}(X)=0$ since $Z$ has i.i.d. standard normal entries. Therefore, it suffices to investigate $\EE\big<\Theta\Theta^{\tran}, \calS_{A,k}(X)\big>$ when $k$ is even. 

\paragraph{Proof of Lemma~\ref{lem:first_order_approx}.}
By the definitions of $\mathfrak{P}^{\perp}, X$ and $\mathfrak{P}^{-1}$, 
\begin{align*}
\EE\|\mathfrak{P}^{\perp}X\mathfrak{P}^{-1}\|_{\rm F}^2=&\EE\|U\Lambda^{-1}V^{\tran}Z^{\tran}U_{\perp}U_{\perp}^{\tran}\|_{\rm F}^2+\EE\|V\Lambda^{-1}U^{\tran}ZV_{\perp}V_{\perp}^{\tran}\|_{\rm F}^2\\
=&\EE\|\Lambda^{-1}V^{\tran}Z^{\tran}U_{\perp}\|_{\rm F}^2+\EE\|\Lambda^{-1}U^{\tran}ZV_{\perp}\|_{\rm F}^2.
\end{align*}
By the proof of Theorem~\ref{thm:distE_normal}, we obtain $\EE\|\mathfrak{P}^{\perp}X\mathfrak{P}^{-1}\|_{\rm F}^2=(d_1+d_2-2r)\|\Lambda^{-1}\|_{\rm F}^2$ which is the first claim. To prove the second claim, it holds by Theorem~\ref{thm:hat_Theta} that
\begin{align*}
\Big|\EE\|\hat\Theta\hat\Theta^{\tran}-&\Theta\Theta^{\tran}\|_{\rm F}^2-2d_{\star}\|\Lambda^{-1}\|_{\rm F}^2 \Big|\leq 2\Big|\sum_{k\geq 2}\EE\big<\Theta\Theta^{\tran}, \calS_{A,2k}(X)\big>\Big|\\
\leq&2\sum_{k\geq 2}\bigg|\EE\bigg<\Theta\Theta^{\tran}, \sum_{\bs:s_1+\cdots+s_{2k+1}=2k}(-1)^{1+\tau(\bs)}\cdot\mathfrak{P}^{-s_1}X\mathfrak{P}^{-s_2}X\cdots X\mathfrak{P}^{-s_{2k}}X\mathfrak{P}^{-s_{2k+1}}\bigg>\bigg|\\
=&2\sum_{k\geq 2}\bigg|\EE\bigg<\Theta\Theta^{\tran}, \sum_{\substack{\bs:s_1+\cdots+s_{2k+1}=2k\\ s_1,s_{2k+1}\geq 1}}(-1)^{1+\tau(\bs)}\cdot\mathfrak{P}^{-s_1}X\mathfrak{P}^{-s_2}X\cdots X\mathfrak{P}^{-s_{2k}}X\mathfrak{P}^{-s_{2k+1}}\bigg>\bigg|
\end{align*}
where we used the fact $\Theta\Theta^{\tran}\mathfrak{P}^0=\mathfrak{P}^0\Theta\Theta^{\tran}=0$.  Then, 
\begin{align*}
\Big|\EE\|\hat\Theta\hat\Theta^{\tran}-\Theta\Theta^{\tran}\|_{\rm F}^2-&2d_{\star}\|\Lambda^{-1}\|_{\rm F}^2 \Big|\\
\leq&4r\sum_{k\geq 2}\sum_{{\substack{\bs:s_1+\cdots+s_{2k+1}=2k\\ s_1,s_{2k+1}\geq 1}}}\EE\Big\|\mathfrak{P}^{-s_1}X\mathfrak{P}^{-s_2}X\cdots X\mathfrak{P}^{-s_{2k}}X\mathfrak{P}^{-s_{2k+1}} \Big\|\\
\leq&4r\sum_{k\geq 2}\sum_{{\substack{\bs:s_1+\cdots+s_{2k+1}=2k\\ s_1,s_{2k+1}\geq 1}}} \frac{\EE\|X\|^{2k}}{\lambda_r^{2k}}\\
\leq&4r\sum_{k\geq 2}{4k\choose 2k} \frac{\EE\|X\|^{2k}}{\lambda_r^{2k}}\leq C_2r\sum_{k\geq 2}\frac{4^{2k}\EE\|X\|^{2k}}{\lambda_r^{2k}}.
\end{align*}
for some absolute constant $C_2>0$. 
Therefore, 
\begin{align*}
\Big|\EE\|\hat\Theta\hat\Theta^{\tran}-\Theta\Theta^{\tran}\|_{\rm F}^2-&2d_{\star}\|\Lambda^{-1}\|_{\rm F}^2 \Big|
\leq C_2r\sum_{k\geq 2}\Big(\frac{16C_1^2d_{\max}}{\lambda_r^2}\Big)^{k}\leq C_2\frac{rd_{\max}^2}{\lambda_r^4}
\end{align*}
where the last inequality holds as long as $\lambda_r\geq 5C_1\sqrt{d_{\max}}$. 

\paragraph{Property 1: only even order terms matter.}
In order to calculate higher order approximations, we need the following useful property of $\EE\calS_{2k}(X)$. 

By Theorem~\ref{thm:hat_Theta}, 
\begin{align*}
\big<\Theta\Theta^{\tran}, \calS_{A,2k}(X)\big>=\sum_{\substack{\bs: s_1+\cdots+s_{2k+1}=2k\\ s_1,s_{2k+1}\geq 1}}(-1)^{1+\tau(\bs)}\cdot \tr\Big(\mathfrak{P}^{-s_1}X\cdots X\mathfrak{P}^{-s_{2k+1}}\Big).
\end{align*}
For any $\tau(\bs)=\tau\geq 2$, there exists positive integers $s_{j_1}, s_{j_2}, \cdots, s_{j_\tau}$ and positive integers $t_1,t_2,\cdots, t_{\tau-1}$ so that we can write
$$
\mathfrak{P}^{-s_1}X\cdots X\mathfrak{P}^{-s_{2k+1}}=\mathfrak{P}^{-s_{j_1}}\underbrace{X\mathfrak{P}^{\perp}\cdots \mathfrak{P}^{\perp}X}_{t_1\ {\rm of}\ X}\mathfrak{P}^{-s_{j_2}}\cdots\mathfrak{P}^{-s_{j_{\tau-1}}}\underbrace{X\mathfrak{P}^{\perp}\cdots \mathfrak{P}^{\perp}X}_{t_{\tau-1}\ {\rm of}\ X}\mathfrak{P}^{-s_{j_{\tau}}}
$$
where 
$$
s_{j_1}+\cdots+s_{j_{\tau}}=2k\quad {\rm and}\quad t_1+\cdots+t_{\tau-1}=2k.
$$
Therefore, for positive integers $s_1,\cdots, s_{2k+1}, t_1,\cdots,t_{2k}\geq 1$, 
\begin{align*}
\big<\Theta\Theta^{\tran}, \EE\calS_{A,2k}(X)\big>=\sum_{\tau\geq 2}(-1)^{1+\tau}\sum_{s_1+\cdots+s_\tau=2k}\sum_{t_1+\cdots+t_{\tau-1}=2k}\EE\tr\big(Q_{t_1t_2\cdots t_{\tau-1}}^{(s_1s_2\cdots s_\tau)}\big)
\end{align*}
where the matrix $Q_{t_1t_2\cdots t_{\tau-1}}^{(s_1s_2\cdots s_\tau)}$ is defined by 
\begin{equation}
Q_{t_1t_2\cdots t_{\tau-1}}^{(s_1s_2\cdots s_\tau)}=\mathfrak{P}^{-s_{1}}\underbrace{X\mathfrak{P}^{\perp}\cdots \mathfrak{P}^{\perp}X}_{t_1\ {\rm of}\ X}\mathfrak{P}^{-s_{2}}\cdots\mathfrak{P}^{-s_{{\tau-1}}}\underbrace{X\mathfrak{P}^{\perp}\cdots \mathfrak{P}^{\perp}X}_{t_{\tau-1}\ {\rm of}\ X}\mathfrak{P}^{-s_{{\tau}}}.
\label{eq:matrixQ}
\end{equation}
{\it Case 1}: if any of $t_1,t_2,\cdots,t_{\tau-1}$ equals one. W.L.O.G., let $t_1=1$. Then, $Q_{t_1t_2\cdots t_{\tau-1}}^{(s_1s_2\cdots s_{\tau})}$ involves the product of $\mathfrak{P}^{-s_{1}}X\mathfrak{P}^{-s_{2}}$. Then, 
\begin{align*}
\big|\EE\tr\big( &Q_{t_1t_2\cdots t_{\tau-1}}^{(s_1s_2\cdots s_\tau)}\big) \big|
\leq\sqrt{2r}\cdot\EE \|\mathfrak{P}^{-s_{1}}X\mathfrak{P}^{-s_{2}}\|_{\rm F}\cdot \frac{\|X\|^{2k-1}}{\lambda_r^{2k-s_{1}-s_{2}}}\\
\leq& \sqrt{2r}\cdot\EE \|\Theta\Theta^{\tran}X\Theta\Theta^{\tran}\|_{\rm F}\cdot \frac{\|X\|^{2k-1}}{\lambda_r^{2k}}\\
\leq&\frac{\sqrt{2r}}{\lambda_r^{2k}}\cdot \EE  \|\Theta\Theta^{\tran}X\Theta\Theta^{\tran}\|_{\rm F} \|X\|^{2k-1}\\
\leq& \frac{\sqrt{2r}}{\lambda_r^{2k}}\cdot\EE^{1/2}\|\Theta\Theta^{\tran}X\Theta\Theta^{\tran}\|^2_{\rm F}\EE^{1/2}\|X\|^{4k-2}\\
\leq&C_1\frac{C_2^{k-1}r^{3/2}d_{\max}^{k-\frac{1}{2}}}{\lambda_r^{2k}}
\end{align*}
where we used the fact 
$
\Theta\Theta^{\tran}X\Theta\Theta^{\tran}=\left(\begin{array}{cc}0&UU^{\tran}ZVV^{\tran}\\ VV^{\tran}Z^{\tran}UU^{\tran}&0\end{array}\right)
$
which is of rank at most $2r$ and $\EE^{1/2}\|U^{\tran}ZV\|^2_{\rm F}=O(r)$. We also used the fact $\EE^{1/p}\|X\|^{p}\leq C_2\sqrt{d_{\max}}$ for some absolute constant $C_2>0$ and all positive integers $p\geq 1$. Therefore, if any of $t_1,\cdots, t_{\tau-1}$ equals one, then the magnitude of $\big|\EE\tr\big(Q^{(s_1s_2\cdots s_{\tau})}_{t_1t_2\cdots t_{\tau-1}}\big) \big|$ is of the order $O\Big(\frac{r^{3/2}}{\sqrt{d_{\max}}}\cdot\frac{C_2^kd_{\max}^{k}}{\lambda_r^{2k}}\Big)$. 
\\
{\it Case 2}: if any of $t_1,\cdots,t_{\tau-1}$ is an odd number greater than $1$. W.L.O.G., let $t_1$ be an odd number and $t_1\geq 3$. More specifically, let $t_1=2p+3$ for some non-negative integer $p\geq 0$. Then, 
\begin{align*}
\big|\EE\big<\Theta\Theta^{\tran},& Q_{t_1t_2\cdots t_{\tau-1}}^{(s_1s_2\cdots s_{\tau})}\big>\big|\\
\leq&\Big|\tr\Big(\mathfrak{P}^{-s_1}X\big(\mathfrak{P}^{\perp}X\big)^{t_1-1}\mathfrak{P}^{-s_{2}}X\big(\mathfrak{P}^{\perp}X\big)^{t_2-1}\mathfrak{P}^{-s_{3}}X\cdots \mathfrak{P}^{-s_{{\tau-1}}}X\big(\mathfrak{P}^{\perp}X\big)^{t_{\tau-1}-1}\mathfrak{P}^{-s_{\tau}}\Big)\Big|\\
\leq&\EE\Big\|\mathfrak{P}^{-s_{1}}X(\mathfrak{P}^{\perp}X\mathfrak{P}^{\perp})^{2p+1}X\mathfrak{P}^{-s_{2}} \Big\|_{\rm F}\cdot \frac{\sqrt{2r}\|X\|^{2k-t_1}}{\lambda_r^{2k-s_{1}-s_{2}}}\\
\leq&\EE\Big\|\mathfrak{P}^{-s_{1}}X(\mathfrak{P}^{\perp}X\mathfrak{P}^{\perp})^{2p+1}X\mathfrak{P}^{-s_{2}} \Big\|_{\rm F}\cdot \frac{\sqrt{2r}\|X\|^{2k-t_1}}{\lambda_r^{2k-s_{1}-s_{2}}}\mathbb{I}_{\calE_1}\\
&\quad\quad\quad +\EE\Big\|\mathfrak{P}^{-s_{1}}X(\mathfrak{P}^{\perp}X\mathfrak{P}^{\perp})^{2p+1}X\mathfrak{P}^{-s_{2}} \Big\|_{\rm F}\cdot \frac{\sqrt{2r}\|X\|^{2k-t_1}}{\lambda_r^{2k-s_{1}-s_{2}}}\mathbb{I}_{\calE_1^{\rm c}}
\end{align*}
where, as in the proof of Theorem~\ref{thm:distE_normal}, define the event $\calE_1=\{\|X\|\leq C_2\cdot \sqrt{d_{\max}}\}$ for some absolute constant $C_2>0$ such that $\PP(\calE_1)\geq 1-e^{-c_1d_{\max}}$.
As a result, we get 
\begin{align*}
\big|\EE\big<\Theta\Theta^{\tran},& Q_{t_1t_2\cdots t_{\tau-1}}^{(s_1s_2\cdots s_{\tau})}\big>\big|\\
\leq&\EE\Big\|\mathfrak{P}^{-s_1}X(\mathfrak{P}^{\perp}X\mathfrak{P}^{\perp})^{2p+1}X\mathfrak{P}^{-s_{2}} \Big\|_{\rm F}\cdot \frac{\sqrt{2r}d_{\max}^{(2k-t_1)/2}}{\lambda_r^{2k-s_{1}-s_{2}}}\mathbb{I}_{\calE_1}+C_2^{2k}\cdot\frac{rd_{\max}^k}{\lambda_r^{2k}}\cdot e^{-c_1d_{\max}}\\
\leq&\EE^{1/2}\Big\|\mathfrak{P}^{-s_{1}}X(\mathfrak{P}^{\perp}X\mathfrak{P}^{\perp})^{2p+1}X\mathfrak{P}^{-s_{2}} \Big\|_{\rm F}^2\cdot \frac{\sqrt{2r}d_{\max}^{(2k-t_1)/2}}{\lambda_r^{2k-s_{1}-s_{2}}}+C_2^{2k}\cdot\frac{rd_{\max}^k}{\lambda_r^{2k}}\cdot e^{-c_1d_{\max}}\\
\leq&\frac{C\sqrt{r}d_{\max}^{(2k-t_1)/2}}{\lambda_r^{2k}}\cdot\EE^{1/2}\Big\|\Theta^{\tran}X(\mathfrak{P}^{\perp}X\mathfrak{P}^{\perp})^{2p+1}X\Theta\Big\|_{\rm F}^2+C_2^{2k}\cdot\frac{rd_{\max}^k}{\lambda_r^{2k}}\cdot e^{-c_1d_{\max}}\
\end{align*}
where $\Theta=(\theta_1,\cdots,\theta_r, \theta_{-r},\cdots,\theta_{-1})\in \RR^{(d_1+d_2)\times (2r)}$. 
In addition, we can write 
\begin{align*}
\EE\Big\|\Theta^{\tran}X(\mathfrak{P}^{\perp}X\mathfrak{P}^{\perp})^{2p+1}X\Theta\Big\|_{\rm F}^2=\sum_{1\leq |j_1|, |j_2|\leq r}\EE\big(\theta_{j_1}^{\tran}X(\mathfrak{P}^{\perp}X\mathfrak{P}^{\perp})^{2p+1}X\theta_{j_2}\big)^2.
\end{align*}
Observe that, for any integer $p\geq 0$,
$$
(\mathfrak{P}^{\perp}X\mathfrak{P}^{\perp})^{2p}=\left(
\begin{array}{cc}
\big(U_{\perp}U_{\perp}^{\tran}ZV_{\perp}V_{\perp}^{\tran}Z^{\tran}U_{\perp}U_{\perp}^{\tran}\big)^p&0\\
0&\big(V_{\perp}V_{\perp}^{\tran}Z^{\tran}U_{\perp}U_{\perp}^{\tran}ZV_{\perp}V_{\perp}^{\tran}\big)^p
\end{array}
\right). 
$$
W.L.O.G, let $j_1,j_2\geq 1$. Then, we write 
\begin{align*}
\theta_{j_1}^{\tran}X(\mathfrak{P}^{\perp}X&\mathfrak{P}^{\perp})^{2p+1}X\theta_{j_2}=\frac{1}{2}v_{j_1}^{\tran}Z^{\tran}\big(U_{\perp}U_{\perp}^{\tran}ZV_{\perp}V_{\perp}^{\tran}Z^{\tran}U_{\perp}U_{\perp}^{\tran}\big)^pU_{\perp}U_{\perp}^{\tran}ZV_{\perp}V_{\perp}^{\tran}Z^{\tran}u_{j_2}\\
&+\frac{1}{2}u_{j_1}^{\tran}Z\big(V_{\perp}V_{\perp}^{\tran}Z^{\tran}U_{\perp}U_{\perp}^{\tran}ZV_{\perp}V_{\perp}^{\tran}\big)^pV_{\perp}V_{\perp}^{\tran}Z^{\tran}U_{\perp}U_{\perp}^{\tran}Zv_{j_2}
\end{align*}
and get  the simple bound
\begin{align*}
\EE\big(\theta_{j_1}^{\tran}X(&\mathfrak{P}^{\perp}X\mathfrak{P}^{\perp})^{2p+1}X\theta_{j_2}\big)^2\\
\leq&2^{-1}\EE\Big(v_{j_1}^{\tran}Z^{\tran}\big(U_{\perp}U_{\perp}^{\tran}ZV_{\perp}V_{\perp}^{\tran}Z^{\tran}U_{\perp}U_{\perp}^{\tran}\big)^pU_{\perp}U_{\perp}^{\tran}ZV_{\perp}V_{\perp}^{\tran}Z^{\tran}u_{j_2}\Big)^2\\
&\quad +2^{-1}\EE\Big(u_{j_1}^{\tran}Z\big(V_{\perp}V_{\perp}^{\tran}Z^{\tran}U_{\perp}U_{\perp}^{\tran}ZV_{\perp}V_{\perp}^{\tran}\big)^pV_{\perp}V_{\perp}^{\tran}Z^{\tran}U_{\perp}U_{\perp}^{\tran}Zv_{j_2}\Big)^2.
\end{align*}
Observe that $Zv_{j_1}$ is independent with $ZV_{\perp}$ and $Z^{\tran}u_{j_1}$ is independent with $Z^{\tran}U_{\perp}$. Therefore, 
\begin{align*}
\EE\big(\theta_{j_1}^{\tran}X(&\mathfrak{P}^{\perp}X\mathfrak{P}^{\perp})^{2p+1}X\theta_{j_2}\big)^2\\
\leq&2^{-1}\EE\Big\|\big(U_{\perp}U_{\perp}^{\tran}ZV_{\perp}V_{\perp}^{\tran}Z^{\tran}U_{\perp}U_{\perp}^{\tran}\big)^pU_{\perp}U_{\perp}^{\tran}ZV_{\perp}V_{\perp}^{\tran}Z^{\tran}u_{j_2} \Big\|_{\ell_2}^2\\
&\quad\quad +2^{-1}\EE\Big\|\big(V_{\perp}V_{\perp}^{\tran}Z^{\tran}U_{\perp}U_{\perp}^{\tran}ZV_{\perp}V_{\perp}^{\tran}\big)^pV_{\perp}V_{\perp}^{\tran}Z^{\tran}U_{\perp}U_{\perp}^{\tran}Zv_{j_2} \Big\|_{\ell_2}^2\\
\leq&2^{-1}\EE\Big\|\big(U_{\perp}U_{\perp}^{\tran}ZV_{\perp}V_{\perp}^{\tran}Z^{\tran}U_{\perp}U_{\perp}^{\tran}\big)^pU_{\perp}U_{\perp}^{\tran}ZV_{\perp}V_{\perp}^{\tran} \Big\|_{\ell_2}^2\\
&\quad\quad +2^{-1}\EE\Big\|\big(V_{\perp}V_{\perp}^{\tran}Z^{\tran}U_{\perp}U_{\perp}^{\tran}ZV_{\perp}V_{\perp}^{\tran}\big)^pV_{\perp}V_{\perp}^{\tran}Z^{\tran}U_{\perp}U_{\perp}^{\tran} \Big\|_{\ell_2}^2\\
&\quad\quad\quad\leq C_2^{4p+2}d_{\max}^{2p+1}=C_2^{4p+2}d_{\max}^{t_1-2},
\end{align*}
where the last inequality is due to the independence between $Z^{\tran}u_{j_2}$ and $Z^{\tran}U_{\perp}$, the independence between $Zv_{j_2}$ and $ZV_{\perp}$. 
We conclude that 
\begin{align*}
\big|\EE\big<\Theta\Theta^{\tran},& Q_{t_1t_2\cdots t_{\tau-1}}^{(s_1s_2\cdots s_{\tau})}\big>\big|\leq C_2^{2k}\cdot \frac{r^{3/2}d_{\max}^{k-1}}{\lambda_r^{2k}}+r\Big(\frac{C_2d_{\max}}{\lambda_r^2}\Big)^k\cdot e^{-c_1d_{\max}}\leq \frac{r^{3/2}}{d_{\max}}\cdot\Big(\frac{C_2d_{\max}}{\lambda_r^2}\Big)^k
\end{align*}
where $C_2>0$ is some absolute constant and the last inequality is due to $e^{-c_1d_{\max}}\leq d_{\max}^{-1}$.  

We now finalize the proof. If there exists one odd $t_i$, then there exists at least another $t_j$ which is also odd since the sum of $t_i$s is even. Following the same analysis, we conclude 
$$
\big|\EE\big<\Theta\Theta^{\tran}, Q_{t_1t_2\cdots t_{\tau-1}}^{(s_1s_2\cdots s_{\tau})}\big>\big|\leq \frac{r^2}{d_{\max}}\cdot\Big(\frac{C_2d_{\max}}{\lambda_r^2}\Big)^k
$$
whenever any of $t_1,\cdots, t_{\tau-1}$ is an odd number. Therefore, it suffices to consider the cases that all of $t_1,\cdots, t_{\tau-1}$ are even numbers.

\paragraph{Proof of Lemma~\ref{lem:second_order_approx}.} From the above analysis, to calculate $\EE\big<\Theta\Theta^{\tran},\calS_{A,4}(X)\big>$, it suffices to calculate 
$$
\sum_{\tau=2}^3(-1)^{1+\tau}\sum_{s_1+\cdots+s_{\tau}=4}\sum_{t_1+\cdots+t_{\tau-1}=4}\EE\big<\Theta\Theta^{\tran}, Q_{t_1t_2\cdots t_{\tau-1}}^{(s_1s_2\cdots s_{\tau})}\big>
$$
where $t_1,\cdots, t_{\tau-1}$ are positive even numbers and $s_1,\cdots,s_{\tau}$ are positive numbers. \\
{\it Case 1: $\tau=2$.} In this case, $t_1=4$ and $s_1+s_2=4$. Therefore, for any $s_1, s_2$ such that $s_1+s_2=4$, 
we shall calculate
\begin{align*}
Q^{(s_1s_2)}_{4}=&\mathfrak{P}^{-s_1}X(\mathfrak{P}^{\perp}X\mathfrak{P}^{\perp})^2X\mathfrak{P}^{-s_2}\\
=&\EE\tr\big(Q_{4}^{(s_1s_2)}\big)=\EE\tr\big(\Theta\Theta^{\tran}X(\mathfrak{P}^{\perp}X\mathfrak{P}^{\perp})^2X\Theta\Theta^{\tran}\mathfrak{P}^{-4}\big).
\end{align*}
Clearly, we have
\begin{align*}
\Theta\Theta^{\tran}X&(\mathfrak{P}^{\perp}X\mathfrak{P}^{\perp})^2X\Theta\Theta^{\tran}\\
=&\left(\begin{array}{cc} UU^{\tran}ZV_{\perp}V_{\perp}Z^{\tran}U_{\perp}U_{\perp}^{\tran}ZV_{\perp}V_{\perp}^{\tran}Z^{\tran}UU^{\tran}&0\\
0&VV^{\tran}Z^{\tran}U_{\perp}U_{\perp}^{\tran}ZV_{\perp}V_{\perp}^{\tran}Z^{\tran}U_{\perp}U_{\perp}^{\tran}ZVV^{\tran}\end{array}\right).
\end{align*}
By the independence between $U^{\tran}Z$ and $U_{\perp}^{\tran}Z$, independence between $V^{\tran}Z^{\tran}$ and $V_{\perp}^{\tran}Z^{\tran}$, we immediately obtain
\begin{align*}
\EE \Theta\Theta^{\tran}X(\mathfrak{P}^{\perp}X\mathfrak{P}^{\perp})^2X\Theta\Theta^{\tran}=&\EE\left(
\begin{array}{cc}
d_{1-}UU^{\tran}ZV_{\perp}V_{\perp}^{\tran}Z^{\tran}UU^{\tran}&0\\
0&d_{2-}VV^{\tran}Z^{\tran}U_{\perp}U_{\perp}^{\tran}ZVV^{\tran}
\end{array}
\right)\\
=&d_{1-}d_{2-}\Theta\Theta^{\tran}
\end{align*}
where $d_{1-}=d_1-r$ and $d_{2-}=d_2-r$. Then, 
\begin{align*}
\EE\big<\Theta\Theta^{\tran}, Q^{(s_1s_2)}_4\big>=2d_{1-}d_{2-}\|\Lambda^{-2}\|_{\rm F}^2
\end{align*}
for all $(s_1,s_2)=(1,3)$, $(s_1,s_2)=(2,2)$ and $(s_1,s_2)=(3,1)$. \\
{\it Case 2: $\tau=3$.} In this case, the only possible even numbers are $t_1=2$ and $t_2=2$. There are three pairs of $(s_1,s_2,s_3)\in \big\{(1,1,2), (1,2,1), (2,1,1)\big\}$. 
W.L.O.G., consider $s_1=1, s_2=1, s_3=2$, we have 
$$
Q_{22}^{(112)}=\mathfrak{P}^{-1}X\mathfrak{P}^{\perp}X\mathfrak{P}^{-1}X\mathfrak{P}^{\perp}X\mathfrak{P}^{-2}.
$$
Similarly, we can write 
\begin{align*}
\EE\tr(Q_{22}^{(112)})=&\EE\tr\big(U\Lambda^{-1}V^{\tran}Z^{\tran}U_{\perp}U_{\perp}^{\tran}ZV\Lambda^{-1}U^{\tran}ZV_{\perp}V_{\perp}^{\tran}Z^{\tran}U\Lambda^{-2}U^{\tran}\big)\\
+&\EE\tr\big(V\Lambda^{-1}U^{\tran}ZV_{\perp}V_{\perp}^{\tran}Z^{\tran}U\Lambda^{-1}V^{\tran}Z^{\tran}U_{\perp}U_{\perp}^{\tran}ZV\Lambda^{-2}V^{\tran}\big)\\
=&d_{2-}\EE\tr\big(U\Lambda^{-1}V^{\tran}Z^{\tran}U_{\perp}U_{\perp}^{\tran}ZV\Lambda^{-3}U^{\tran}\big)+d_{1-}\EE\tr\big(V\Lambda^{-1}U^{\tran}ZV_{\perp}V_{\perp}^{\tran}Z^{\tran}U\Lambda^{-3}V^{\tran}\big)\\
=&2d_{1-}d_{2-}\|\Lambda^{-2}\|_{\rm F}^2. 
\end{align*}
By symmetricity, the same equation holds for $\EE\tr(Q_{22}^{(211)})$. Next, we consider $(s_1,s_2,s_3)=(1,2,1)$. We will write
\begin{align*}
\EE\tr(Q_{22}^{(121)})=&\EE\tr\big(U\Lambda^{-1}V^{\tran}Z^{\tran}U_{\perp}U_{\perp}^{\tran}ZV\Lambda^{-2}V^{\tran}Z^{\tran}U_{\perp}U_{\perp}^{\tran}ZV\Lambda^{-1}U^{\tran}\big)\\
+&\EE\tr\big(V\Lambda^{-1}U^{\tran}ZV_{\perp}V_{\perp}^{\tran}Z^{\tran}U\Lambda^{-2}U^{\tran}ZV_{\perp}V_{\perp}^{\tran}Z^{\tran}U\Lambda^{-1}V^{\tran}\big)\\
=&\EE\|\Lambda^{-1}\tilde Z_1\tilde Z_1^{\tran}\Lambda^{-1}\|_{\rm F}^2+\EE\|\Lambda^{-1}\tilde{Z}_2\tilde{Z}_2^{\tran}\Lambda^{-1}\|_{\rm F}^2
\end{align*}
where $\tilde Z_1\in\RR^{r\times d_{1-}}$ and $\tilde Z_2\in\RR^{r\times d_{2-}}$ contain i.i.d. standard normal entries. By Lemma~\ref{lem:LambdaZ_frob} in the Appendix, we obtain
$$
\EE\tr\big(Q_{22}^{(121)}\big)=(d_{1-}^2+d_{2-}^2)\|\Lambda^{-2}\|_{\rm F}^2+(d_{1-}+d_{2-})\big(\|\Lambda^{-2}\|_{\rm F}^2+\|\Lambda^{-1}\|_{\rm F}^4\big).
$$
Therefore, we conclude that 
\begin{align*}
\Big|-\EE\big<\Theta\Theta^{\tran},\calS_{A,4}(X)\big> + (d_{1-}-d_{2-})^2\|\Lambda^{-2}\|_{\rm F}^2\Big|\leq C_1\cdot \frac{r^2d_{\max}}{\lambda_r^4}
\end{align*}
for some absolute constant $C_1>0$ where we also include those smaller terms when some $t_i$ is odd  as discussed in {\bf Property 1}. Together with the proof of Lemma~\ref{lem:first_order_approx}, we conclude that 
\begin{align*}
\Big|\EE\|\hat\Theta\hat\Theta^{\tran}-\Theta\Theta^{\tran}\|_{\rm F}^2-2\big(d_{\star}\|\Lambda^{-1}\|_{\rm F}^2-\Delta_d^2\|\Lambda^{-2}\|_{\rm F}^2\big) \Big|\leq C_1\cdot \frac{r^2d_{\max}}{\lambda_r^4}+C_2\cdot \frac{rd_{\max}^3}{\lambda_r^6}
\end{align*}
where $\Delta_d=d_1-d_2$ and $C_1,C_2>0$ are absolute constants.

\subsection{Proof of Lemma~\ref{lem:fourth_order_approx}.} 
To characterize $\EE\big<\Theta\Theta^{\tran},\calS_{A,2k}(X)\big>$ more easily, we observe the following property.

\paragraph{Property 2: effect from distinct singular values are negligible.} Recall that 
$$
\EE\big<\Theta\Theta^{\tran},\calS_{A,2k}(X)\big>=\sum_{\bs: s_1+\cdots+s_{2k+1}=2k}(-1)^{1+\tau(\bs)}\cdot \EE \big<\Theta\Theta^{\tran}, \mathfrak{P}^{-s_1}X\cdots X\mathfrak{P}^{-s_{2k+1}}\big>.
$$
As proved in {\bf Property 1}, we have 
\begin{align*}
\Big|\EE\big<\Theta\Theta^{\tran},\calS_{A,2k}(X)\big>-\sum_{\tau\geq 2}(-1)^{1+\tau}\sum_{s_1+\cdots+s_{\tau}=2k}\sum_{t_1+\cdots+t_{\tau-1}=2k}\EE\big<&\Theta\Theta^{\tran}, Q_{t_1t_2\cdots t_{\tau-1}}^{(s_1s_2\cdots s_{\tau})}\big>\Big|\\
\leq& \frac{r^{2}}{d_{\max}}\cdot \Big(\frac{C_2d_{\max}}{\lambda_r^{2}}\Big)^k
\end{align*}
where the matrix $Q_{t_1t_2\cdots t_{\tau-1}}^{(s_1s_2\cdots s_{\tau})}$ is defined as in (\ref{eq:matrixQ}) and $t_1,\cdots, t_{\tau-1}$ are positive even numbers. Recall that $\Theta\Theta^{\tran}=\sum_{j=1}^r (P_j+P_{-j})$ and $\mathfrak{P}^{-s_1}=\sum_{j=1}^r \big[\lambda_j^{-s_1}P_j+(\lambda_{-j})^{-s_j}P_{-j}\big]$ where $\lambda_{-j}=-\lambda_j$. For each fixed $(s_1,\cdots,s_{\tau})$ and $(t_1,\cdots,t_{\tau-1})$ where $t_j$s are even numbers, we write 
\begin{align*}
\big<\Theta&\Theta^{\tran}, Q_{t_1t_2\cdots t_{\tau-1}}^{(s_1s_2\cdots s_{\tau})}\big>\\
=&\sum_{|j_1|, |j_2|, \cdots,|j_{\tau-1}|\geq 1}^r\lambda_{j_1}^{-(s_1+s_{\tau})}\lambda_{j_2}^{-s_2}\cdots\lambda_{j_{\tau-1}}^{-s_{\tau-1}}(\theta_{j_1}^{\tran}W_{t_1}\theta_{j_2})(\theta_{j_2}^{\tran}W_{t_2}\theta_{j_3})\cdots (\theta_{j_{\tau-1}}^{\tran}W_{t_{\tau-1}}\theta_{j_1})
\end{align*}
where the matrix $W_{t_1}=\underbrace{X\mathfrak{P}^{\perp}X\mathfrak{P}^{\perp}\cdots\mathfrak{P}^{\perp}X}_{t_1\ {\rm of }\ X}$ for positive even numbers $t_1$. Observe that 
$$
\theta_{j_1}^{\tran}W_{t_1}\theta_{j_2}=\theta^{\tran}_{j_1}W_{t_1}\theta_{j_2}=\theta_{j_1}^{\tran} X (\mathfrak{P}^{\perp}X\mathfrak{P}^{\perp})^{t_1-2}X\theta_{j_2}.
$$
We show that if there exists $1\leq k_0\leq \tau-1$ so that $|j_{k_0}|\neq |j_{k_0+1}|$, then $|\theta_{j_{k_0}}^{\tran}W_{t_{k_0}}\theta_{j_{k_0+1}}|$ is a negligibly smaller term. W.L.O.G., assume $|j_1|\neq |j_2|$ and then 
\begin{align*}
\bigg|\EE\sum_{\substack{|j_1|, |j_2|, \cdots,|j_{\tau-1}|\geq 1\\ |j_1|\neq |j_2|}}^r&\lambda_{j_1}^{-(s_1+s_{\tau})}\lambda_{j_2}^{-s_2}\cdots\lambda_{j_{\tau-1}}^{-s_{\tau-1}}(\theta_{j_1}^{\tran}W_{t_1}\theta_{j_2})(\theta_{j_2}^{\tran}W_{t_2}\theta_{j_3})\cdots (\theta_{j_{\tau-1}}^{\tran}W_{t_{\tau-1}}\theta_{j_1})\bigg|\\
=&\Big|\EE\sum_{|j_1|\neq |j_2|}\lambda_{j_1}^{-(s_1+s_{\tau})}\lambda_{j_2}^{-s_2}(\theta_{j_1}^{\tran}W_{t_1}\theta_{j_2})\theta_{j_2}^{\tran}W_{t_2}\mathfrak{P}^{-s_3}W_{t_3}\mathfrak{P}^{-s_4}\cdots\mathfrak{P}^{-s_{\tau-1}}W_{t_{\tau-1}}\theta_{j_1}\Big|\\
\leq&\frac{1}{\lambda_r^{2k}}\sum_{|j_1|\neq |j_2|}\EE\big|\theta_{j_1}^{\tran}W_{t_1}\theta_{j_2}\big|\|X\|^{2k-t_1}
\end{align*} 
Since $\theta_{j_1}$ and $\theta_{j_2}$ are orthogonal, we conclude that $X\theta_{j_1}$ and $X\theta_{j_2}$ are independent normal vectors from which we get that  $\theta_{j_1}^{\tran}W_{t_1}\theta_{j_2}|\mathfrak{P}^{\perp}X\mathfrak{P}^{\perp}$ is sub-exponential and $\EE |\theta_{j_1}^{\tran}W_{t_1}\theta_{j_2}|=O\big(\|(\mathfrak{P}^{\perp}X\mathfrak{P}^{\perp})^{t_1-2}\|_{\rm F}\big)$. Therefore, we get
\begin{align*}
\EE\big|&\theta_{j_1}^{\tran}W_{t_1}\theta_{j_2}\big|\|X\|^{2k-t_1}\\
=&\EE\big|\theta_{j_1}^{\tran}W_{t_1}\theta_{j_2}\big|\|X\|^{2k-t_1}{\bf 1}(\|X\|\leq C_1\sqrt{d_{\max}})+\EE\big|\theta_{j_1}^{\tran}W_{t_1}\theta_{j_2}\big|\|X\|^{2k-t_1}{\bf 1}(\|X\|\geq C_1\sqrt{d_{\max}})\\
\leq&\EE^{1/2}\big|\theta_{j_1}^{\tran}W_{t_1}\theta_{j_2}\big|^2 \cdot (C_1^2d_{\max})^{k-t_1/2}{\bf 1}(\|X\|\leq C_1d_{\max}^{1/2})+e^{-d_{\max}/2} (C_1d_{\max})^{k}\\
\lesssim &\frac{1}{\sqrt{d_{\max}}}\cdot (C_2d_{\max})^{k}+e^{-d_{\max}/2} (C_2d_{\max})^{k}.
\end{align*}
As a result, we conclude that  
\begin{align*}
\bigg|\EE\sum_{|j_1|, |j_2|, \cdots,|j_{\tau-1}|\geq 1}^r&\lambda_{j_1}^{-(s_1+s_{\tau})}\lambda_{j_2}^{-s_2}\cdots\lambda_{j_{\tau-1}}^{-s_{\tau-1}}(\theta_{j_1}^{\tran}W_{t_1}\theta_{j_2})(\theta_{j_2}^{\tran}W_{t_2}\theta_{j_3})\cdots (\theta_{j_{\tau-1}}^{\tran}W_{t_{\tau-1}}\theta_{j_1})\bigg|\\
\leq& \frac{C_1r^2}{\sqrt{d_{\max}}}\cdot\Big(\frac{C_2d_{\max}}{\lambda_r^2}\Big)^k+C_3e^{-d_{\max}/2}\cdot\Big(\frac{C_2d_{\max}}{\lambda_r^2}\Big)^k\leq  \frac{C_1r^2}{\sqrt{d_{\max}}}\cdot\Big(\frac{C_2d_{\max}}{\lambda_r^2}\Big)^k
\end{align*}
for some absolute constants $C_1,C_2>0$. 

It suggests that the dominating terms come from those tuples $(j_1,j_2,\cdots,j_{\tau-1})$ such that $|j_1|=|j_2|=\cdots=|j_{\tau-1}|$. 
Now, we define $\mathfrak{P}_j=\lambda_j P_j+\lambda_{-j} P_{-j}$.  
To this end, we conclude
\begin{align}
\Big|\EE\big<\Theta\Theta^{\tran},\calS_{A,2k}(X)\big>-\sum_{j=1}^{r}\sum_{\tau\geq 2}(-1)^{1+\tau}&\sum_{\substack{\bs:s_1+\cdots+s_{\tau}=2k, s_1, s_{\tau}>0\\ \bt: t_1+\cdots +t_{\tau-1}=2k}}\EE\tr\big(\mathfrak{P}_j^{-s_1}W_{t_1}\mathfrak{P}_j^{-s_2}W_{t_2}\cdots \mathfrak{P}_j^{-s_{\tau}}\big)\Big|\nonumber\\
\leq& \frac{C_1r^2}{\sqrt{d_{\max}}}\cdot \Big(\frac{C_2d_{\max}}{\lambda_r^2}\Big)^k\label{eq:SA2k}
\end{align}
for some absolute constants $C_1,C_2>0$. The above fact suggests that it suffices to focus on the effect from individual singular values (i.e., for any fixed $1\leq j\leq r$). Moreover, it is easy to check that 
$$
\mathfrak{P}_j^{-s_1}W_{t_1}\mathfrak{P}_j^{-s_2}W_{t_2}\cdots \mathfrak{P}_j^{-s_{\tau}}=\frac{1}{\lambda_j^{2k}}\cdot \tilde{\mathfrak{P}}_j^{-s_1}W_{t_1}\tilde{\mathfrak{P}}_j^{-s_2}W_{t_2}\cdots \tilde{\mathfrak{P}}_j^{-s_{\tau}}
$$
where $\tilde{\mathfrak{P}}_j^{-s}=P_j+(-1)^s P_{-j}$ implying that the $k$-th order error term has dominator $\lambda_j^{2k}$. 
To this end, we prove the following lemma in the Appendix.
\begin{lemma}\label{lem:rank-one}
For any $1\leq j\leq r$ and $k\geq 2$, we obtain 
\begin{align*}
\Big|\sum_{\tau\geq 2}(-1)^{1+\tau}&\sum_{\substack{\bs:s_1+\cdots+s_{\tau}=2k, s_1, s_{\tau}>0\\ \bt: t_1+\cdots +t_{\tau-1}=2k}}\EE\tr\big(\mathfrak{P}_j^{-s_1}W_{t_1}\mathfrak{P}_j^{-s_2}\cdots \mathfrak{P}_j^{-s_{\tau}}\big)-\frac{(-1)^{k}(d_{1-}^{k-1}-d_{2-}^{k-1})(d_{1-}-d_{2-})}{\lambda_j^{2k}}\Big|\\
\leq& \frac{C_1k}{\sqrt{d_{\max}}}\cdot \Big(\frac{C_2d_{\max}}{\lambda_r^2}\Big)^k
\end{align*}
for some absolute constants $C_1,C_2>0$.
\end{lemma}
By Lemma~\ref{lem:rank-one} and (\ref{eq:SA2k}), it holds for all $k\geq 2$ that
$$
\big|\EE\big<\Theta\Theta^{\tran},\calS_{A,2k}(X)\big>-(-1)^k(d_{1-}^{k-1}-d_{2-}^{k-1})(d_{1-}-d_{2-})\|\Lambda^{-k}\|_{\rm F}^2 \big|\leq \frac{C_1(r^2+k)}{\sqrt{d_{\max}}}\cdot\Big(\frac{C_2d_{\max}}{\lambda_r^2}\Big)^k
$$
for some absolute constants $C_1, C_2>0$, which concludes the proof.

\subsection{Proof of CLT theorems in Section~\ref{sec:explicit_CLT}}

\paragraph{Proof of Theorem~\ref{thm:first_order_CLT}}
Recall Theorem~\ref{thm:distE_normal}, we end up with 
\begin{align*}
\sup_{x\in\RR}\bigg|\PP\bigg(&\frac{{\rm dist}^2[(\hat U,\hat V), (U,V)]-\EE{\rm dist}^2[(\hat U,\hat V), (U,V)]}{\sqrt{8(d_1+d_2-2r)}\|\Lambda^{-2}\|_{\rm F}}\leq x\bigg)-\Phi(x) \bigg|\\
& \leq C_2\Big(\frac{\sqrt{r}}{\|\Lambda^{-2}\|_{\rm F}\lambda_r^2}\Big)\cdot \sqrt{\frac{(rd_{\max})^{1/2}}{\lambda_r}}+e^{-c_1d_{\max}}+C_2\Big(\frac{\|\Lambda^{-1}\|_{\rm F}^4}{\|\Lambda^{-2}\|_{\rm F}^2}\Big)^{3/2}\cdot \frac{1}{\sqrt{d_{\max}}}+e^{-\lambda_r/\sqrt{rd_{\max}}}.
\end{align*}
By Lemma~\ref{lem:first_order_approx}, we get
\begin{align*}
\big| \EE{\rm dist}^2[(\hat U,\hat V), (U,V)]-2d_{\star}\|\Lambda^{-1}\|_{\rm F}^2\big|
\leq C_2\frac{rd_{\max}^2}{\lambda_r^4}.
\end{align*}
Therefore, 
\begin{align*}
\Big| \frac{\EE{\rm dist}^2[(\hat U,\hat V), (U,V)]-2d_{\star}\|\Lambda^{-1}\|_{\rm F}^2}{\sqrt{8d_{\star}}\|\Lambda^{-2}\|_{\rm F}}\Big|\leq C_2\frac{rd_{\max}^{3/2}}{\lambda_r^2}.
\end{align*}
By the Lipschitz property of $\Phi(x)$ and applying similar technical as in proof of Theorem~\ref{thm:distE_normal}, we can get 
\begin{align*}
\sup_{x\in\RR}\bigg|\PP\bigg(&\frac{{\rm dist}^2[(\hat U,\hat V), (U,V)]-2d_{\star}\|\Lambda^{-1}\|_{\rm F}^2}{\sqrt{8d_{\star}}\|\Lambda^{-2}\|_{\rm F}}\leq x\bigg)-\Phi(x) \bigg|
\leq C_2\Big(\frac{\sqrt{r}}{\|\Lambda^{-2}\|_{\rm F}\lambda_r^2}\Big)\cdot \sqrt{\frac{(rd_{\max})^{1/2}}{\lambda_r}}\\
&+e^{-c_1d_{\max}}+C_2\Big(\frac{\|\Lambda^{-1}\|_{\rm F}^4}{\|\Lambda^{-2}\|_{\rm F}^2}\Big)^{3/2}\cdot \frac{1}{\sqrt{d_{\max}}}+C_3\frac{rd_{\max}^{3/2}}{\lambda_r^2}+e^{-\lambda_r/\sqrt{rd_{\max}}}. 
\end{align*}

\paragraph{Proof of Theorem~\ref{thm:second_order_CLT}}
By Lemma~\ref{lem:fourth_order_approx}, we have 
$$
\Big|\EE{\rm dist}^2[(\hat U,\hat V), (U,V)]-B_k\Big|\leq C_4\frac{r^2d_{\max}}{\lambda_r^4}+\frac{C_5r^2}{\sqrt{d_{\max}}}\cdot\Big(\frac{d_{\max}}{\lambda_r^2}\Big)^3+C_6r\Big(\frac{C_3d_{\max}}{\lambda_r^2}\Big)^{k+1}.
$$
The rest of the proof is the same as in the proof of Theorem~\ref{thm:first_order_CLT}. 

\section{Appendix}

\subsection{Supporting lemmas}\label{sec:supp_lem}
 
\begin{proof}[Proof of Lemma~\ref{lem:fX_Lip}]
Recall that 
$$
f_t(X_1)=\sum_{k\geq 3}\big<\Theta\Theta^{\tran},\calS_{A,k}(X_1)\big>\phi\Big(\frac{\|X_1\|}{ t\cdot \sqrt{d_{\max}}}\Big). 
$$
{\it Case 1}: if $\|X_1\|>2t\sqrt{d_{\max}}$ and $\|X_2\|>2t\sqrt{d_{\max}}$, then $f_t(X_1)=f_t(X_2)=0$ by definition of $\phi(\cdot)$ where the claimed inequality holds trivially. \\
{\it Case 2}: if $\|X_1\|\leq 2t\sqrt{d_{\max}}$ and $\|X_2\|>2t\sqrt{d_{\max}}$, then $f_t(X_2)=0$. We get, by Lipschitz property of $\phi(\cdot)$, that
\begin{align*}
\Big|f_t(X_1)-f_t(X_2)\Big|=& \bigg|\sum_{k\geq 3}\big<\Theta\Theta^{\tran},\calS_{A,k}(X_1)\big>\cdot \bigg(\phi\Big(\frac{\|X_1\|}{ t\cdot \sqrt{d_{\max}}}\Big)-\phi\Big(\frac{\|X_2\|}{ t\cdot \sqrt{d_{\max}}}\Big)\bigg)\bigg| \\
\leq&\sum_{k\geq 3}2r\big\|\calS_{A,k}(X_1)\big\|\cdot \frac{\|X_1-X_2\|_{\rm F}}{t\cdot \sqrt{d_{\max}}}\\
\leq&\frac{2r\|X_1-X_2\|_{\rm F}}{t\cdot \sqrt{d_{\max}}}\cdot\sum_{k\geq 3}\sum_{\bs: s_1+\cdots+s_{k+1}=k}\Big\|\mathfrak{P}^{-s_1}X_1\mathfrak{P}^{-s_2}X_1\cdots X_1\mathfrak{P}^{-s_{k+1}} \Big\|\\
\leq&\frac{2r\|X_1-X_2\|_{\rm F}}{t\cdot \sqrt{d_{\max}}}\cdot\sum_{k\geq 3}\sum_{\bs: s_1+\cdots+s_{k+1}=k}\frac{\|X_1\|^{k}}{\lambda_r^{k}}\\
\leq&\frac{2r\|X_1-X_2\|_{\rm F}}{t\cdot \sqrt{d_{\max}}}\cdot\sum_{k\geq 3}\Big(\frac{4\|X_1\|}{\lambda_r}\Big)^k\\
\leq& C_4t^2\frac{r\|X_1-X_2\|_{\rm F}}{\sqrt{d_{\max}}}\cdot \frac{d_{\max}^{3/2}}{\lambda_r^3}
\end{align*}
where the last inequality holds as long as $\lambda_r\geq 9t\sqrt{d_{\max}}$. \\
{\it Case 3}: if $\|X_1\|\leq 2td_{\max}^{1/2}$ and $\|X_2\|\leq 2td_{\max}^{1/2}$. Then, 
\begin{align*}
\Big|f_t&(X_1)-f_t(X_2) \Big|\leq 2r\sum_{k\geq 3}\Big\|\calS_{A,k}(X_1)\phi\Big(\frac{\|X_1\|}{t\cdot \sqrt{d_{\max}}}\Big)-\calS_{A,k}(X_2)\phi\Big(\frac{\|X_2\|}{t\cdot \sqrt{d_{\max}}}\Big) \Big\|\\
\leq& 2r\sum_{k\geq 3}\sum_{\bs: s_1+\cdots+s_{k+1}=k}\Big\|\mathfrak{P}^{-s_1}X_1\cdots X_1\mathfrak{P}^{-s_{k+1}}\phi\Big(\frac{\|X_1\|}{t\cdot \sqrt{d_{\max}}}\Big)- \mathfrak{P}^{-s_1}X_2\cdots X_2\mathfrak{P}^{-s_{k+1}}\phi\Big(\frac{\|X_2\|}{t\cdot \sqrt{d_{\max}}}\Big)\Big\|\\
\leq&2r\sum_{k\geq 3}\sum_{\bs:s_1+\cdots+s_{k+1}=k}(k+2)\cdot \frac{(2td_{\max}^{1/2})^{k-1}}{\lambda_r^{k}}\|X_1-X_2\|_{\rm F}\leq C_4t^2\cdot\frac{rd_{\max}}{\lambda_r^3}\|X_1-X_2\|_{\rm F}
\end{align*}
where the last inequality holds as long as $\lambda_r\geq 9t\sqrt{d_{\max}}$. Therefore, we conclude the proof of Lemma~\ref{lem:fX_Lip}.
\end{proof}

 \begin{proof}[Proof of Lemma~\ref{lem:rank-one}]
 Based on {\bf Property 2} and eq. (\ref{eq:SA2k}), it suffices to calculate the quantities 
 $\EE\tr\big(\mathfrak{P}_j^{-s_1}W_{t_1}\mathfrak{P}_j^{-s_2}W_{t_2}\cdots \mathfrak{P}_j^{-s_{\tau}}\big)$ which relies on singular values $\lambda_j$ and singular vectors $u_j, v_j$ only. Moreover, the actual forms of $u_j, v_j$ does not affect the values. By choosing $\{u_j\}_{j=1}^r$ and $\{v_j\}_{j=1}^r$ as the first $r$ canonical basis vectors in $\RR^{d_1}$ and $\RR^{d_2}$, it is easy to check that we can reduce the calculations to the rank-one spiked model with singular value $\lambda_j$. To leverage the dimensionality effect where $U_{\perp}^{\tran}ZV_{\perp}\in\RR^{d_{1-}\times d_{2-}}$ has i.i.d. standard normal entries, we consider the rank-one spiked model with 
\begin{equation}\label{eq:rank-one}
\hat M =\lambda (u\otimes v) +Z\in\RR^{(d_{1-}+1)\times (d_{2-}+1)}
\end{equation}
where $Z$ has i.i.d. standard normal entries and $d_{1-}=d_1-r, d_{2-}=d_2-r$. Let $\hat u$ and $\hat v$ denote the leading left and right singular vectors of $\hat M$. By fact (\ref{eq:SA2k}), it suffices to calculate the $k$-th order approximation of $\|\hat u\hat u^{\tran}-uu^{\tran}\|_{\rm F}^2+\|\hat v\hat v^{\tran}-vv^{\tran}\|_{\rm F}^2$. In the proof, we calculate the errors $\|\hat u\hat u^{\tran}-uu^{\tran}\|_{\rm F}^2$ and $\|\hat v\hat v^{\tran}-vv^{\tran}\|_{\rm F}^2$ separately. W.L.O.G., we just deal with $\|\hat u\hat u^{\tran}-uu^{\tran}\|_{\rm F}^2$ and consider $d_1\leq d_2$\footnote{This condition just simplifies our calculation when dealing with the Marchenko Pastur law. Our results do not rely on the condition $d_1\leq d_2$}. 

Recall that we aim to calculate the $k$-th order error term in $\|uu^{\tran}-\hat u\hat u^{\tran}\|_{\rm F}^2$. To this end, we write the error terms as
\begin{equation}\label{eq:EEk_def}
\EE\|\hat u\hat u^{\tran}-uu^{\tran}\|_{\rm F}^2= 2\sum_{k=1}^{\infty}\frac{E_{2k}}{\lambda^{2k}}.
\end{equation}
We show that $E_{2k}=(-1)^kd_{1-}^{k-1}(d_{1-}-d_{2-})\cdot \Big[1+O\Big(\frac{C_1^k}{\sqrt{d_{\max}}}\Big)\Big]$ for some absolute constant $C_1>0$.  To this end, we consider the second-order (see \citep{xia2019sup}) moment trick (denote $T=\lambda^2 (u\otimes u)$)
\begin{equation}\label{eq:MMtran}
\hat M\hat M^{\tran}=\lambda^2 (u\otimes u) + \Delta\in\RR^{(d_{1-}+1)\times (d_{1-}+1)}
\end{equation}
where $\Delta=\lambda uv^{\tran}Z^{\tran}+\lambda Zvu^{\tran}+ZZ^{\tran}$.  By eq. (\ref{eq:dist_formula}), we can write 
$$
\|\hat u\hat u^{\tran}-uu^{\tran}\|_{\rm F}^2=-2\sum_{k\geq 2}\big<uu^{\top}, \calS_{T,k}(\Delta)\big>
$$
where we define $\mathfrak{P}_u=\lambda (u\otimes u)$ and $\mathfrak{P}_u^{0}=\mathfrak{P}_u^{\perp}=U_{\perp}U_{\perp}^{\top}\in\RR^{(d_{1-}+1)\times d_{1-}}$ and 
\begin{align*}
\calS_{T,k}(\Delta)=\sum_{\bs: s_1+\cdots+s_{k+1}=k}(-1)^{\tau(\bs)+1}\cdot \mathfrak{P}_u^{-s_1}\Delta\mathfrak{P}_u^{s_2}\Delta\cdots\Delta\mathfrak{P}_u^{-s_{k+1}}.
\end{align*}
Now, we investigate $\big<uu^{\top},\calS_{T,k}(\Delta)\big>$ for all $k\geq 2$. Denote $W_{t_1}=\underbrace{\Delta\mathfrak{P}_u^{\perp}\Delta\cdots\mathfrak{P}^{\perp}_u\Delta}_{t_1\ {\rm of}\ \Delta }$ and we can write 
\begin{align*}
\big<uu^{\top},\calS_{T,k}(\Delta)\big>=&\sum_{\tau=2}^{k}(-1)^{\tau+1}\sum_{\substack{t_1+\cdots+t_{\tau-1}=k, t_j\geq 1\\ s_1+\cdots+s_{\tau}=k, s_j\geq 1}}\tr\big(\mathfrak{P}_u^{-s_1}W_{t_1}\mathfrak{P}_u^{-s_2}W_{t_2}\cdots \mathfrak{P}_u^{-s_{\tau-1}}W_{t_{\tau-1}}\mathfrak{P}_u^{-s_{\tau}}\big)\\
=&\frac{1}{\lambda^{2k}}\sum_{\tau=2}^{k}(-1)^{1+\tau}{k-1\choose \tau-1}\sum_{t_1+\cdots +t_{\tau-1}=k, t_j\geq 1}(u^{\tran}W_{t_1}u)(u^{\tran}W_{t_2}u)\cdots (u^{\tran}W_{t_{\tau-1}}u).
\end{align*}
Denote $\beta^{\Delta}_{t_1}=u^{\tran}W_{t_1}u$,  we can write concisely 
\begin{equation}\label{eq:STkDelta}
\EE\big<uu^{\top},\calS_{T,k}(\Delta)\big>=\frac{1}{\lambda^{2k}}\sum_{\tau=2}^{k}(-1)^{1+\tau}{k-1\choose \tau-1}\sum_{t_1+\cdots +t_{\tau-1}=k, t_j\geq 1}\EE\big(\beta_{t_1}^{\Delta}\beta_{t_2}^{\Delta}\cdots\beta_{t_{\tau-1}}^{\Delta}\big). 
\end{equation}
Now, we investigate the concentration property of $\beta^{\Delta}_t=u^{\tran}W_tu$. 
Clearly, we can write 
$$
\beta_1^{\Delta}=\underbrace{2\lambda\cdot(u^{\tran}Zv)}_{\beta_{1,1}^{\Delta}}+\underbrace{u^{\tran}ZZ^{\tran}u}_{\beta_{1,0}^{\Delta}}
$$
and for all $t\geq 2$, we write $\beta_t^{\Delta}=\beta_{t,1}^{\Delta}+\beta_{t,0}^{\Delta}$ where 
\begin{align*}
\beta_{t,0}^{\Delta}=&u^{\tran}ZZ^{\tran}U_{\perp}(U_{\perp}^{\tran}ZZ^{\tran}U_{\perp})^{t-2}U_{\perp}^{\tran}ZZ^{\tran}u+\lambda^2uv^{\tran}Z^{\tran}U_{\perp}(U_{\perp}^{\tran}ZZ^{\tran}U_{\perp})^{t-2}U_{\perp}^{\tran}Zvu^{\tran}\\
\beta_{t,1}^{\Delta}=&2\lambda uv^{\tran}Z^{\tran}U_{\perp}(U_{\perp}^{\tran}ZZ^{\tran}U_{\perp})^{t-2}U_{\perp}^{\tran}ZZ^{\tran}u.
\end{align*}
As a result, we can calculate 
$$
\EE (\beta_{t_1}^{\Delta}\beta_{t_2}^{\Delta}\cdots\beta_{t_{\tau-1}}^{\Delta})=\EE\big((\beta^{\Delta}_{t_1,0}+\beta^{\Delta}_{t_1,1})(\beta^{\Delta}_{t_2,0}+\beta^{\Delta}_{t_2,1})\cdots (\beta^{\Delta}_{t_{\tau-1},0}+\beta^{\Delta}_{t_{\tau-1},1})\big).
$$
It is easy to check that $\EE\beta_1^{\Delta}=d_{2-}+1$ and for $t\geq 2$
\begin{align*}
\EE\beta_t^{\Delta}=&\lambda^2\EE\big(v^{\tran}Z^{\tran}U_{\perp}(U_{\perp}^{\tran}ZZ^{\tran}U_{\perp})^{t-2}U_{\perp}^{\tran}Zv\big)+\EE\tr\big(Z^{\tran}U_{\perp}(U_{\perp}^{\tran}ZZ^{\tran}U_{\perp})^{t-2}U_{\perp}^{\tran}Z\big)\\
&=\Big(1+\frac{\lambda^2}{d_{2-}+1}\Big)\cdot \EE\tr\big(Z^{\tran}U_{\perp}(U_{\perp}^{\tran}ZZ^{\tran}U_{\perp})^{t-2}U_{\perp}^{\tran}Z\big)\\
&=\Big(1+\frac{\lambda^2}{d_{2-}+1}\Big)\cdot \EE\tr\big((U_{\perp}^{\tran}ZZ^{\tran}U_{\perp})^{t-1}\big)
\end{align*}
where the second equality can be checked by choosing $v=e_1\in\RR^{d_{2-}+1}$. Since $Z^{\tran}u$ and $Z^{\tran}U_{\perp}$ are independent, it is easy to check that 
$$
\EE \beta_{t_1,i_1}^{\Delta} \beta_{t_2,i_2}^{\Delta}\cdots \beta_{t_{\tau-1},i_{\tau-1}}^{\Delta}=0,\quad \textrm{ if } \sum_{j=1}^{\tau-1}i_j \textrm{ is an odd number}
$$
for all $i_1,i_2,\cdots,i_{\tau-1}\in\{0,1\}$. As a result, we observe that $\EE\big<uu^{\tran}, \calS_{T,k}\big>$ has contributions to $E_{2k}, E_{2k-2}, E_{2k-4},\cdots, E_{2\ceil{k/2}}$. (Recall that $E_{2k}$ is the coefficient for $\frac{1}{\lambda^{2k}}$.)

Moreover, since $Z^{\tran}u$ and $Z^{\tran}U_{\perp}$ are independent, we can conclude that 
$$
\beta^{\Delta}_{1,1}\sim\calN(0, 4\lambda^2)
$$
and for all $t\geq 2$, 
$$
\beta^{\Delta}_{t,1}\big| U_{\perp}^{\tran}Z\sim \calN\big(0, 4\lambda^2\|Z^{\tran}U_{\perp}(U_{\perp}^{\tran}ZZ^{\tran}U_{\perp})^{t-2}U_{\perp}^{\tran}Zv\|_{\ell_2}^2\big).
$$
We can get, for all $t\geq 2$, that
$$
\EE^{1/2}\big[(\beta_{t,1}^{\Delta})^2\big| U_{\perp}^{\tran}Z\big]\lesssim\EE^{1/4}\big[(\beta_{t,1}^{\Delta})^4\big| U_{\perp}^{\tran}Z\big]\lesssim \lambda \|U_{\perp}^{\tran}Z\|^{2(t-1)}
$$

Therefore, it is easy to check that for any $(i_1, i_2,\cdots, i_{\tau-1})\in\{0,1\}^{\tau-1}$ where there exists some $i_j\geq 1$, then $\EE \beta_{t_1,i_1}^{\Delta} \beta_{t_2,i_2}^{\Delta}\cdots \beta_{t_{\tau-1},i_{\tau-1}}^{\Delta}$'s contribution to any $E_{2k_1}$ is bounded by $\frac{1}{d_{\max}}\cdot \Big(\frac{C_1d_{\max}}{\lambda^2}\Big)^{k_1}$ for some absolute constant $C_1>0$ and $2\ceil{k/2}\leq 2k_1\leq 2k$. To show this, w.l.o.g, let $i_1=i_2=1$ and observe that 
\begin{align}
\EE\beta_{t_1,1}^{\Delta} &\beta_{t_2,1}^{\Delta}\beta_{t_3,i_3}^{\Delta}\cdots \beta_{t_{\tau-1},i_{\tau-1}}^{\Delta}=\EE^{1/2}(\beta_{t_1,1}^{\Delta}\beta_{t_2,1}^{\Delta})^2\EE^{1/2}\big(\beta_{t_3,i_3}^{\Delta}\cdots\beta_{t_{\tau-1},i_{\tau-1}}^{\Delta}\big)^2\nonumber\\
\leq&\EE^{1/4}(\beta_{t_1,1}^{\Delta})^4\EE^{1/4}(\beta_{t_2,1}^{\Delta})^4\EE^{1/2}\big(\beta_{t_3,i_3}^{\Delta}\cdots\beta_{t_{\tau-1},i_{\tau-1}}^{\Delta}\big)^2\nonumber\\
\leq&\lambda^2 d_{\max}^{t_1+t_2-2}\EE^{1/2}\big(\beta_{t_3,i_3}^{\Delta}\cdots\beta_{t_{\tau-1},i_{\tau-1}}^{\Delta}\big)^2\label{eq:beta_tDelta}
\end{align}
and then we get 
\begin{align*}
\frac{1}{\lambda^{2k}}\EE\beta_{t_1,1}^{\Delta} &\beta_{t_2,1}^{\Delta}\beta_{t_3,i_3}^{\Delta}\cdots \beta_{t_{\tau-1},i_{\tau-1}}^{\Delta}\leq \frac{1}{d_{\max}}\cdot \Big(\frac{d_{\max}}{\lambda^2}\Big)^{t_1+t_2-1}\cdot \frac{\EE^{1/2}\big(\beta_{t_3,i_3}^{\Delta}\cdots\beta_{t_{\tau-1},i_{\tau-1}}^{\Delta}\big)^2}{\lambda^{2(k-t_1-t_2)}}.
\end{align*}
The claim follows immediately since 
$$
\frac{\EE^{1/2}\big(\beta_{t_3,i_3}^{\Delta}\cdots\beta_{t_{\tau-1},i_{\tau-1}}^{\Delta}\big)^2}{\lambda^{2(k-t_1-t_2)}}\leq\sum_{k_1=\ceil{(k-t_1-t_2)/2}}^{k_1}\frac{C_1^{k_1}\EE^{1/2}\|Z\|^{4k_1}}{\lambda^{2k_1}} \leq \sum_{k_1=\ceil{(k-t_1-t_2)/2}}^{k_1}\Big(\frac{C_2d_{\max}}{\lambda^2}\Big)^{k_1}
$$
for some absolute constant $C_1, C_2>0$ and where the last inequality is due to $\EE\|Z\|^{4k_1}\leq C_3^{4k_1}d_{\max}^{2k_1}$ for some absolute constant $C_3>0$.  

As a result, in order to calculate eq. (\ref{eq:STkDelta}), it suffices to calculate 
\begin{equation}\label{eq:beta_t0sum}
\frac{1}{\lambda^{2k}}\sum_{\tau=2}^{k}(-1)^{1+\tau}{k-1\choose \tau-1}\sum_{t_1+\cdots +t_{\tau-1}=k, t_j\geq 1}\EE\big(\beta_{t_1,0}^{\Delta}\beta_{t_2,0}^{\Delta}\cdots\beta_{t_{\tau-1},0}^{\Delta}\big).
\end{equation}
Next, we will replace $\EE\big(\beta_{t_1,0}^{\Delta}\beta_{t_2,0}^{\Delta}\cdots\beta_{t_{\tau-1},0}^{\Delta}\big)$ with $\EE\beta_{t_1,0}^{\Delta}\EE\beta^{\Delta}_{t_2,0}\cdots\EE\beta_{t_{\tau-1},0}^{\Delta}$ for which we shall investigate the concentrations of $\beta_{t,0}^{\Delta}$. To this end, we have the sub-exponential inequality
$$
\PP\big(\big|u^{\tran}ZZ^{\tran}u-d_{2-} \big|\geq C_3\sqrt{\alpha d_{2-}}+C_4\alpha \big)\leq C_5e^{-\alpha},\quad \forall \alpha>0
$$
for some constants $C_3,C_4>0$. 
Again, by Gaussian isoperimetric inequality and the proof of Theorem~\ref{thm:first_order_CLT}\footnote{We just need to study the Lipschitz property of the function $f(Z)=u^{\tran}Z(Z^{\tran}U_{\perp}U_{\perp}^{\tran}Z)^tZ^{\tran}u\cdot {\bf 1}(\|Z\|\leq C_1\sqrt{d_{\max}})$}, we can show, for all $\alpha>0$
\begin{align*}
\PP\big(\big|u^{\tran}Z(Z^{\tran}U_{\perp}U_{\perp}^{\tran}Z)^{t}Z^{\tran}u-\EE u^{\tran}Z(Z^{\tran}U_{\perp}U_{\perp}^{\tran}Z)^{t}Z^{\tran}u\big|\geq &C_3\alpha d_{\max}^{t+1/2}+C_4e^{-c_1d_{\max}}d_{\max}^{t+1}\big)\\
&\leq C_5e^{-\alpha^2}+C_6e^{-c_2d_{\max}}
\end{align*}
and 
\begin{align*}
\PP\big(\big|v^{\tran}(Z^{\tran}U_{\perp}U_{\perp}^{\tran}Z)^{t-1}v-\EE v^{\tran}(Z^{\tran}U_{\perp}U_{\perp}^{\tran}Z)^{t-1}v\big|\geq& C_3t\alpha d_{\max}^{t-3/2}+C_4e^{-c_1d_{\max}} d_{\max}^{t-1}\big)\\
&\leq C_5 e^{-\alpha^2}+C_6e^{-c_2d_{\max}}.
\end{align*}
Therefore,  we can show that $\big|\EE(\beta_{t_1,0}^{\Delta}\beta_{t_2,0}^{\Delta}\cdots\beta_{t_{\tau-1},0}^{\Delta})-(\EE\beta_{t_1,0}^{\Delta})(\EE\beta_{t_2,0}^{\Delta})\cdots(\EE\beta_{t_{\tau-1},0}^{\Delta})\big|$'s contribution to any $E_{2k_1}$ is bounded by $\frac{1}{\sqrt{d_{\max}}}\cdot (C_1d_{\max}/\lambda^2)^{k_1}$ for some constant $C_1>0$ and $2\ceil{k/2}\leq 2k_1\leq 2k$. Indeed,  the above concentration inequalities of $\beta_{t,0}^{\Delta}$ imply 
$$
\EE^{1/2}(\beta_{t,0}^{\Delta}-\EE\beta_{t,0}^{\Delta})^2\lesssim d_{\max}^{t-1/2}+\lambda^2 d_{\max}^{t-3/2},\quad \forall t\geq 1.
$$
The claim can be proved as in eq. (\ref{eq:beta_tDelta}). Indeed, we can write 
\begin{align*}
\frac{1}{\lambda^{2k}}\Big|\EE\beta_{t_1,0}^{\Delta}\beta_{t_2,0}^{\Delta}&\cdots\beta_{t_{\tau-1},0}^{\Delta}-(\EE\beta_{t_1,0}^{\Delta})(\EE\beta_{t_2,0}^{\Delta})\cdots\EE(\beta_{t_{\tau-1},0}^{\Delta})\Big|\\
\leq&\frac{1}{\lambda^{2k}}\sum_{i=1}^{\tau-1}\Big(\prod_{j=1}^{i-1}\EE\beta_{t_j,0}^{\Delta}\Big)\Big|\EE\big(\beta_{t_i,0}^{\Delta}-\EE\beta_{t_i,0}^{\Delta}\big)\Big(\prod_{j=i+1}^{\tau-1}\beta_{t_j,0}^{\Delta}\Big)\Big|\\
\leq&\sum_{i=1}^{\tau-1}\frac{\EE^{1/2}\big(\beta_{t_i,0}^{\Delta}-\EE\beta_{t_i,0}^{\Delta}\big)^2}{\lambda^{2t_i}}\cdot \frac{1}{\lambda^{2(k-t_i)}}\prod_{j=1}^{i-1}\Big(\EE\beta_{t_j,0}^{\Delta}\Big)\EE^{1/2}\Big(\prod_{j=i+1}^{\tau-1}\beta_{t_j,0}^{\Delta}\Big)^2\\
\leq&\sum_{i=1}^{\tau-1}\frac{1}{\sqrt{d_{\max}}}\Big((d_{\max}/\lambda^2)^{t_i}+(d_{\max}/\lambda^2)^{t_i-1}\Big)\cdot \frac{1}{\lambda^{2(k-t_i)}}\prod_{j=1}^{i-1}\Big(\EE\beta_{t_j,0}^{\Delta}\Big)\EE^{1/2}\Big(\prod_{j=i+1}^{\tau-1}\beta_{t_j,0}^{\Delta}\Big)^2
\end{align*}
which concludes the proof since $\frac{1}{\lambda^{2(k-t_i)}}\prod_{j=1}^{i-1}\Big(\EE\beta_{t_j,0}^{\Delta}\Big)\EE^{1/2}\Big(\prod_{j=i+1}^{\tau-1}\beta_{t_j,0}^{\Delta}\Big)^2\leq \Big(\frac{C_2d_{\max}}{\lambda^2}\Big)^{k-t_i}$.


To this end, to calculate eq. (\ref{eq:STkDelta}), it suffices to calculate
$$
\frac{1}{\lambda^{2k}}\sum_{\tau=2}^{k}(-1)^{1+\tau}{k-1\choose \tau-1}\sum_{t_1+\cdots +t_{\tau-1}=k, t_j\geq 1}\EE\beta_{t_1,0}^{\Delta}\EE\beta_{t_2,0}^{\Delta}\cdots\EE\beta_{t_{\tau-1},0}^{\Delta}
$$

Now, we compute $\EE\beta_{t,0}^{\Delta}=\big(1+\lambda^2/(d_{2-}+1)\big)\cdot \EE\tr\big((U_{\perp}^{\tran}ZZ^{\tran}U_{\perp})^{t-1}\big)$. Note that the matrix $U_{\perp}^{\tran}Z\in\RR^{d_{1-}\times (d_{2-}+1)}$ has i.i.d. standard normal entries. By the moment of Marchenko-Pastur law (\citep{mingo2017free}), for all $t\geq 2$, we define (additionally, $\beta_1=d_{2-}+1$)
\begin{equation}\label{eq:betat}
\frac{\beta_t}{1+\lambda^2/(1+d_{2-})}=\frac{1}{t-1}\sum_{r=0}^{t-2}d_{1-}^{r+1}(d_{2-}+1)^{t-1-r}{t-1\choose r+1}{t-1\choose r}.
\end{equation}
Note that $\EE\tr\big((U_{\perp}^{\tran}ZZ^{\tran}U_{\perp})^{t-1}\big)=\EE\tr\big((Z^{\tran}U_{\perp}U_{\perp}^{\tran}Z)^{t-1}\big)$ for all $t\geq 2$. 
By the rate of convergence of Marchenko Pastur law (\cite[Theorem~1.1]{gotze2011rate}), we have (as long as $\sqrt{d_{\max}}\geq \log^2 d_{\max}$)
$$
\frac{\big|\beta_t-\EE\beta_{t,0}^{\Delta} \big|}{1+\lambda^2/(d_{2-}+1)}\leq \frac{1}{\sqrt{d_{\max}}}\cdot \big(C_1d_{\max}\big)^{t-1}
$$
for all $t\geq 2$ where $C_1>0$ is an absolute constant. As a result, we get that for all $t_1+\cdots +t_{\tau-1}=k$, the contribution to $E_{2k_1}$ from $\big|\EE\beta_{t_1,0}^{\Delta}\EE\beta_{t_2,0}^{\Delta}\cdots\EE\beta^{\Delta}_{t_{\tau-1},0}-\beta_{t_1}\beta_{t_2}\cdots\beta_{t_{\tau-1}} \big|$ is bounded by $\frac{1}{\sqrt{d_{\max}}}\cdot \Big(\frac{C_1d_{\max}}{\lambda^2}\Big)^{k_1}$. 

Therefore, by eq. (\ref{eq:STkDelta}), to calculate $\EE\big<uu^{\tran}, \calS_{T,k}(\Delta)\big>$, we consider the following term
$$
\frac{1}{\lambda^{2k}}\sum_{\tau=2}^k(-1)^{1+\tau}{k-1\choose \tau-1}\sum_{t_1+\cdots+t_{\tau-1}=k, t_j\geq 1} \beta_{t_1}\beta_{t_2}\cdots\beta_{t_{\tau-1}}
$$
which is the $k$-th order derivative of the function $\frac{1}{\lambda^{2k}\cdot(k!)}(1-g(\alpha))^{k-1}$ at $\alpha=0$ where 
\begin{equation}\label{eq:galpha}
g(\alpha)=\beta_1\alpha+\alpha^2\beta_2+\alpha^3\beta_3+\cdots=\sum_{k\geq 1}\beta_k \alpha^k.
\end{equation}
Now, we calculate the explicit form of the function $g(\alpha)$. Denote $\gamma=\frac{d_{1-}}{d_{2-}+1}$ and $Y$ the random variable obeying the Marchenko-Pastur distribution, i.e., its pdf is given by
$$
f_{Y}(y)=\frac{1}{2\pi}\frac{\sqrt{(\gamma_+-y)(y-\gamma_-)}}{\gamma y}\cdot {\bf 1}(y\in[\gamma_-, \gamma_+])
$$
where $\gamma_+=(1+\sqrt{\gamma})^2$ and $\gamma_-=(1-\sqrt{\gamma})^2$. It is easy to check that (\citep{mingo2017free})
$$
\beta_t=\Big(1+\frac{\lambda^2}{1+d_{2-}}\Big)d_{1-}(d_{2-}+1)^{t-1}\EE Y^{t-1},\quad \forall t\geq 2.
$$
For notational simplicity, we just write $d_{2-}$ instead of $1+d_{2-}$. 
As a result, we get for $\alpha\ll \frac{1}{d_2^2}$, 
\begin{align*}
g(\alpha)=&\beta_1\alpha+\Big(1+\frac{\lambda^2}{d_{2-}}\Big)d_{1-}\alpha\EE\sum_{t\geq 1}d_{2-}^t(\alpha Y)^t\\
=&\beta_1\alpha+\Big(1+\frac{\lambda^2}{d_{2-}}\Big)\EE \frac{d_{1-}\alpha d_{2-}\alpha Y}{1-d_{2-}\alpha Y}\\
=&\alpha d_{2-}+\Big(1+\frac{\lambda^2}{d_{2-}}\Big)\cdot \frac{\big(\sqrt{1-\alpha d_{2-}\gamma_-}-\sqrt{1-\alpha d_{2-}\gamma_+}\big)^2}{4}
\end{align*}
where the last equality comes up by integrating $Y$ according to the p.d.f. $F_Y(y)$. Therefore, we get 
$$
1-g(\alpha)=\frac{1}{2}\Big[g_+(\alpha)-\frac{\lambda^2}{d_{2-}}g_-(\alpha)\Big]
$$
where 
$$
g_-(\alpha)=1-(d_{1-}+d_{2-})\alpha-\sqrt{(1-\alpha d_{2-}\gamma_-)(1-\alpha d_{2-}\gamma_+)}
$$
and 
$$
g_+(\alpha)=1-(d_{2-}-d_{1-})\alpha+\sqrt{(1-\alpha d_{2-}\gamma_-)(1-\alpha d_{2-}\gamma_+)}.
$$
Therefore, in order to calculate $\EE\big<uu^{\top},\calS_{T,k}(\Delta)\big>$, it suffices to calculate the $k$-th order derivative of function $\frac{(1-g(\alpha))^{k-1}}{\lambda^{2k}\cdot(k!)}$ at $\alpha=0$. Write
\begin{equation}\label{eq:galphak}
\frac{\Big[\big(1-g(\alpha)\big)^{k-1}\Big]^{(k)}}{\lambda^{2k}(k!)}\Big|_{\alpha=0}=\frac{1}{\lambda^{2k}\cdot2^{k-1}\cdot (k!)}\sum_{t=0}^{k-1}{k-1\choose t} \Big(-\frac{\lambda^2}{d_{2-}}\Big)^t\Big[g_-^t(\alpha)g_+^{k-1-t}(\alpha)\Big]^{(k)}\Big|_{\alpha=0}.
\end{equation}
Note that $g_-(\alpha)=O(\alpha^2)$. The terms in eq. (\ref{eq:galphak}) with $t>\frac{k}{2}$ are all $0$. 
Recall that we are interested in the $k_0$-th order term in the error $\|\hat u\hat u^{\tran}-uu^{\tran}\|_{\rm F}^2$ whose denominator is $\lambda^{2k_0}$.  By eq. (\ref{eq:galphak}), the $k_0$-th order error term $\frac{1}{\lambda^{2k_0}}$ can be contributed from $\EE\big<uu^{\top},\calS_{T,k}(\Delta)\big>$ for $k=k_0, k=k_0+1,\cdots, k=2k_0$. 

By the above analysis, we conclude that the $k_0$-th error term (except the negligible error terms from translating $\EE(\beta_{t_1}^{\Delta}\beta_{t_2}^{\Delta}\cdots\beta_{t_{\tau-1}}^{\Delta})$ into $\beta_{t_1}\beta_{t_2}\cdots\beta_{t_{\tau-1}}$) of $\|\hat u\hat u^{\tran}-uu^{\tran}\|_{\rm F}^2$ is given by $E_{2k_0}=\sum_{t=0}^{k_0}E_{2k_0,t}$ where (we change $k$ in (\ref{eq:galphak}) to $k_0+t$) 
$$
E_{2k_0,t}=\frac{1}{\lambda^{2k_0}}\frac{1}{2^{k_0+t-1}}\frac{1}{(k_0+t)!}{k_0+t-1\choose t}\Big(-\frac{1}{d_{2-}}\Big)^t\Big[g_-^t(\alpha)g_{+}^{k_0-1}(\alpha)\Big]^{(k_0+t)}\Big|_{\alpha=0}
$$
When $t=k_0$, we have 
$$
g_{-}^{k_0}(\alpha)=\frac{(4d_{1-}d_{2-})^{k_0}\alpha^{2k_0}}{\big[1-\alpha(d_{1-}+d_{2-})+\sqrt{(1-\alpha d_{2-}\gamma_-)(1-\alpha d_{2-}\gamma_+)}\big]^{k_0}}
$$
implying that 
$$
\Big[g_{-}^{k_0}(\alpha)g_{+}^{k_0-1}(\alpha)\Big]^{(2k_0)}\Big|_{\alpha=0}=(2k_0)! \frac{(4d_{1-}d_{2-})^{k_0}}{2}.
$$
Therefore, we get $E_{2k_0,k_0}=(-1)^{k_0}d_{1-}^{k_0}{2k_0-1\choose k_0}$. Now, we consider $t\leq k_0-1$ and we observe
$$
1-\alpha(d_{1-}+d_{2-})+\sqrt{(1-\alpha d_{2-}\gamma_-)(1-\alpha d_{2-}\gamma_+)}=g_{+}(\alpha)-2d_{1-}\alpha
$$
so that $g_{-}(\alpha)=\frac{4d_{1-}d_{2-}\alpha^2}{g_{+}(\alpha)-2\alpha d_{1-}}$, Then, we get 
\begin{align*}
\Big[g_{-}^t(\alpha)g_+^{k_0-1}(\alpha)\Big]^{(k_0+t)}\Big|_{\alpha=0}=&\Big[\frac{(4d_{1-}d_{2-}\alpha^2)^{t}}{\big(g_+(\alpha)-2\alpha d_{1-}\big)^{t}}\cdot g_{+}^{k_0-1}(\alpha)\Big]^{(k_0+t)}\Big|_{\alpha=0}\\
=&{k_0+t\choose 2t}(2t)!(4d_{1-}d_{2-})^{t}\Big[\frac{g_+^{k_0-1}(\alpha)}{(g_+(\alpha)-2\alpha d_{1-})^t}\Big]^{(k_0-t)}\Big|_{\alpha=0}.
\end{align*}
It suffices to calculate the $(k_0-t)$-th derivative of function $g_+^{k_0-1}(\alpha)/\big(g_{+}(\alpha)-2\alpha d_{1-}\big)^t$ at $\alpha=0$. We write 
\begin{align*}
\Big[\frac{g_+^{k_0-1}(\alpha)}{(g_+(\alpha)-2\alpha d_{1-})^t}\Big]^{(k_0-t)}\Big|_{\alpha=0}=\Big[\sum_{t_1=0}^{k_0-1}{k_0-1\choose t_1}(2\alpha d_{1-})^{k_0-1-t_1}\big(g_{+}(\alpha)-2\alpha d_{1-}\big)^{t_1-t}\Big]^{(k_0-t)}\Big|_{\alpha=0}.
\end{align*}
Observe that $\big[(2\alpha d_{1-})^{k_0-1-t_1}\big]^{(k_0-t)}\big|_{\alpha=0}=0$ for all $t_1< t-1$. Then, we get
$$
\Big[\frac{g_+^{k_0-1}(\alpha)}{(g_+(\alpha)-2\alpha d_{1-})^t}\Big]^{(k_0-t)}\Big|_{\alpha=0}=\Big[\sum_{t_1=t-1}^{k_0-1}{k_0-1\choose t_1}(2\alpha d_{1-})^{k_0-1-t_1}\big(g_{+}(\alpha)-2\alpha d_{1-}\big)^{t_1-t}\Big]^{(k_0-t)}\Big|_{\alpha=0}.
$$
If $t_1=t-1$, then 
$$
\Big[{k_0-1\choose t_1}(2\alpha d_{1-})^{k_0-1-t_1}\big(g_{+}(\alpha)-2\alpha d_{1-}\big)^{t_1-t}\Big]^{(k_0-t)}\Big|_{\alpha=0}={k_0-1\choose t-1}(2d_1)^{k_0-t}(k_0-t)!\cdot \frac{1}{2}
$$
If $t_1\geq t$, we have 
\begin{align*}
\Big[(2\alpha d_{1-}&)^{k_0-1-t_1}\big(g_{+}(\alpha)-2\alpha d_{1-}\big)^{t_1-t}\Big]^{(k_0-t)}\Big|_{\alpha=0}\\
&={k_0-t\choose k_0-1-t_1}(2d_1)^{k_0-1-t_1}(k_0-1-t_1)!\Big[\big(g_{+}(\alpha)-2\alpha d_{1-}\big)^{t_1-t}\Big]^{(t_1+1-t)}\Big|_{\alpha=0}.
\end{align*}
Clearly, if $t_1=t$, then $\big[\big(g_{+}(\alpha)-2\alpha d_{1-}\big)^{t_1-t}\big]^{(t_1+1-t)}\big|_{\alpha=0}=0$. For $t_1\geq t+1$, recall that 
$$
g_{+}(\alpha)-2\alpha d_1=1-(d_{1-}+d_{2-})\alpha+\sqrt{(1-\alpha d_{2-}\gamma_-)(1-\alpha d_{2-}\gamma_+)}.
$$
It is easy to check that 
\begin{align*}
\big[\big(g_{+}(\alpha)-&2\alpha d_{1-}\big)^{t_1-t}\big]^{(t_1+1-t)}\big|_{\alpha=0}\\
=&-\big[\big(1-(d_{1-}+d_{2-})\alpha-\sqrt{(1-\alpha d_{2-}\gamma_-)(1-\alpha d_{2-}\gamma_+)}\big)^{t_1-t}\big]^{(t_1+1-t)}\big|_{\alpha=0}\\
=&-\Big[\Big(\frac{4d_{1-}d_{2-}\alpha^2}{g_{+}(\alpha)-2\alpha d_{1-}}\Big)^{t_1-t}\Big]^{(t_1+1-t)}\Big|_{\alpha=0}
\end{align*}
which is non-zero only when $t_1=t+1$. In fact, when $t_1=t+1$, we get 
$$
\big[\big(g_{+}(\alpha)-2\alpha d_{1-}\big)^{t_1-t}\big]^{(t_1+1-t)}\big|_{\alpha=0}=-4d_{1-}d_{2-}.
$$
Therefore, we conclude that 
\begin{align*}
\Big[&\frac{g_+^{k_0-1}(\alpha)}{(g_+(\alpha)-2\alpha d_{1-})^t}\Big]^{(k_0-t)}\Big|_{\alpha=0}\\
=&{k_0-1\choose t-1}(2d_{1-})^{k_0-t}(k_0-t)!\cdot \frac{1}{2}-{k_0-1\choose t+1}{k_0-t\choose 2}(2d_{1-})^{k_0-t-2}(k_0-2-t)!(4d_{1-}d_{2-}).
\end{align*}
As a result, for $t\leq k_0-1$, we get
$$
E_{2k_0,t}=d_{1-}^{k_0}\cdot (-1)^t{k_0+t-1\choose t}{k_0-1\choose t-1}-d_{1-}^{k_0-1}d_{2-}\cdot (-1)^t{k_0+t-1\choose t}{k_0-1\choose t+1}.
$$
Clearly, it also holds for $t=k_0$. Therefore, we have 
\begin{align*}
E_{2k_0}=&\sum_{t=0}^{k_0}E_{2k_0,t}\\
=&d_{1-}^{k_0}\sum_{t=0}^{k_0}(-1)^{t}{k_0+t-1\choose t}{k_0-1\choose t-1} -d_{1-}^{k_0-1}d_{2-}\sum_{t=0}^{k_0-2}(-1)^t{k_0+t-1\choose t}{k_0-1\choose t+1}.
\end{align*}
It is easy to check that 
\begin{align*}
\sum_{t=0}^{k_0}(-1)^{t}{k_0+t-1\choose t}&{k_0-1\choose t-1}=\sum_{t=1}^{k_0}(-1)^{t}{k_0+t-1\choose t}{k_0-1\choose t-1}\\
=&(-1)\sum_{t=0}^{k_0-1}(-1)^t{k_0+t\choose t+1}{k_0-1\choose t}.
\end{align*}
It is interesting to observe that $\sum_{t=0}^{k_0-1}(-1)^t{k_0+t\choose t+1}{k_0-1\choose t}$ equals the coefficient of $x^{k_0-1}$ in the polynomial $(1+x)^{k_0}\big[1-(1+x)\big]^{k_0-1}$. Then, it is easy to check that $\sum_{t=0}^{k_0-1}(-1)^t{k_0+t\choose t+1}{k_0-1\choose t}=(-1)^{k_0-1}$. Similarly, we can observe that 
\begin{align*}
\sum_{t=0}^{k_0-2}(-1)^t{k_0+t-1\choose t}{k_0-1\choose t+1}=&\sum_{t=1}^{k_0-1}(-1)^{t-1}{k_0+t-2\choose t-1}{k_0-1\choose t}\\
=&(-1)\sum_{t=1}^{k_0-1}(-1)^t{k_0+t-2\choose t-1}{k_0-1\choose t}.
\end{align*}
Again, it is easy to check that $\sum_{t=1}^{k_0-1}(-1)^t{k_0+t-2\choose t-1}{k_0-1\choose t}$ equals the coefficient of $x^{k_0-1}$ in the polynomial $(1+x)^{k_0-2}[1-(1+x)]^{k_0-1}$. As a result, we get $\sum_{t=1}^{k_0-1}(-1)^t{k_0+t-2\choose t-1}{k_0-1\choose t}=(-1)^{k_0-1}$. To this end, we conclude that 
$$
E_{2k_0}=(-1)^{k_0}d_{1-}^{k_0-1}(d_{1-}-d_{2-})
$$
, i.e., the $k_0$-th error term in $\EE\|\hat u\hat u^{\tran}-uu^{\tran}\|_{\rm F}^2$ is given by $\frac{(-1)^{k_0}d_{1-}^{k_0-1}(d_{1-}-d_{2-})}{\lambda^{2k_0}}$ (except the negligible error terms). 
In a similar fashion, we can show that the $k_0$-th error term in $\EE\|\hat v\hat v^{\tran}-vv^{\tran}\|_{\rm F}^2$ is given by $\frac{(-1)^{k_0}d_{2-}^{k_0-1}(d_{2-}-d_{1-})}{\lambda^{2k_0}}$. Meanwhile, the negligible error terms from translating $\EE(\beta_{t_1}^{\Delta}\beta_{t_2}^{\Delta}\cdots\beta_{t_{\tau-1}}^{\Delta})$ into $\beta_{t_1}\beta_{t_2}\cdots\beta_{t_{\tau-1}}$ are upper bounded by $\frac{k_0}{\sqrt{d_{\max}}}\cdot\Big(\frac{C_2d_{\max}}{\lambda_r^2}\Big)^{k_0}$ which concludes the proof. 

 \end{proof}

\begin{lemma}\label{lem:LambdaZ_frob}
Let $\Lambda={\rm diag}(\lambda_1,\cdots,\lambda_r)$ and $Z\in\RR^{r\times d}$ be a random matrix containing i.i.d. standard normal entries. 
Then, for any positive numbers $j_1, j_2$, we have 
$$
\EE\|\Lambda^{-j_1} ZZ^{\tran}\Lambda^{-j_2}\|_{\rm F}^2=d^2\|\Lambda^{-j_1-j_2}\|_{\rm F}^2+d\big(\|\Lambda^{-j_1-j_2}\|_{\rm F}^2+\|\Lambda^{-j_1}\|_{\rm F}^2\|\Lambda^{-j_2}\|_{\rm F}^2\big).
$$
\end{lemma}
\begin{proof}[Proof of Lemma~\ref{lem:LambdaZ_frob}]
Let $z_1,\cdots, z_{r}\in\RR^d$ denote the columns of $Z^{\tran}$.  Therefore, we can write  
$$
\|\Lambda^{-j_1} ZZ^{\tran}\Lambda^{-j_2}\|_{\rm F}^2=\sum_{i=1}^r \frac{1}{\lambda_i^{2(j_1+j_2)}}(z_i^{\tran}z_i)^2+\sum_{1\leq i_1\neq i_2\leq r}\frac{1}{\lambda_{i_1}^{2j_1}\lambda_{i_2}^{2j_2}}(z_{i_1}^{\tran}z_{i_2})^2.
$$
Then, we get 
\begin{align*}
\EE\|\Lambda^{-1} ZZ^{\tran}\Lambda^{-1}\|_{\rm F}^2=&\sum_{i=1}^r\frac{d^2+2d}{\lambda_i^{2(j_1+j_2)}}+\sum_{1\leq i_1\neq i_2\leq r}\frac{d}{\lambda_{i_1}^{2j_1}\lambda_{i_2}^{2j_2}}\\
=&d^2\|\Lambda^{-j_1-j_2}\|_{\rm F}^2+d\big(\|\Lambda^{-j_1-j_2}\|_{\rm F}^2+\|\Lambda^{-j_1}\|_{\rm F}^2\|\Lambda^{-j_2}\|_{\rm F}^2\big).
\end{align*}
\end{proof}

 \end{document}